\documentclass{article}

\usepackage{titling}
\usepackage{graphicx}
\usepackage{amsfonts}
\usepackage{amsmath}
\usepackage{amssymb}
\usepackage{amsthm}
\usepackage{algorithm}
\usepackage{color}
\usepackage{hyperref}
\usepackage{cite}
\usepackage{authblk}

\newcommand{\I}{\mathrm{i}\,}
\newtheorem{theorem}{Theorem}

\newtheorem{lemma}{Lemma}
\theoremstyle{remark}
\newtheorem{remark}{Remark}

\newcommand{\HZ}[1]{\textcolor{black}{#1}}
\newcommand{\MG}[1]{\textcolor{black}{#1}} 

\title{Fourier Analysis of Finite Difference Schemes for the Helmholtz Equation in 1D with Dirichlet
  Conditions: Sharp Estimates and Relative Errors}

\author{Martin J. Gander\thanks{Department of Mathematics, University of Geneva, Rue du
    Conseil-Général 9, CP 64, 1211 Genève 4, Switzerland, martin.gander@unige.ch}\quad Hui
  Zhang\thanks{Corresponding author, School of Mathematics \& Physics, Xi'an Jiaotong-Liverpool
    University, Ren'ai Road 111, 215123 Suzhou, China, hui.zhang@xjtlu.edu.cn} \quad Haiyang
  Zhou\thanks{Affiliated with the same as Hui Zhang, and also with Department of Mathematical
    Sciences, University of Liverpool, L69 3BX, Liverpool, UK, haiyang.zhou19@student.xjtlu.edu.cn}}

\date{\today}

\begin{document}
\maketitle

\begin{abstract}
  \HZ{We consider the Dirichlet problem of the indefinite Helmholtz equation in 1D, $u''+k^2u=f$ in
    $(0,1)$, $u(0)=g_0$, $u(1)=g_1$, with a constant wavenumber
    $k\in(0,\infty)\backslash\pi\mathbb{N}$ and a source term $f\in H^p_0(0,1)$, $p\ge 4$.}  We
  propose an approach based on Fourier analysis to \MG{derive} wavenumber explicit sharp estimat\MG{es} of absolute
  and relative errors of \emph{finite difference} methods. \HZ{Such results have been well known for
    \emph{finite element} methods (FEM).} We use the approach to analyze the classical \MG{centered}
  \HZ{finite difference} scheme. For the Fourier interpolants of the discrete solution with
  homogeneous (or inhomogeneous) Dirichlet conditions, we show rigorously, \HZ{under the \MG{two}
    assumptions $k>20$ and $k(kh)^2/\sigma_k\le4/(\pi-2)$ with
    $\sigma_k:=\operatorname{dist}(k,\pi\mathbb{N})$,} that the worst case attainable convergence
  order of the absolute error \HZ{with $\sum_{p=0}^4k^{-p}\|f^{(p)}\|_{L^2}=O(1)$ (or
    $|g_i|\asymp k^{-1}$)} is $\HZ{(kh)^2/\sigma_k^2}$ in the $L^2$-norm and
  $\HZ{k(kh)^2/\sigma_k^2}$ in the $H^1$-semi-norm, and that of the relative error is
  $\HZ{k(kh)^2/\sigma_k}$ in both $L^2$- and $H^1$-semi-norms \HZ{if
    $\|u^{(p)}\|_{L^2}/\|u^{(p-2)}\|_{L^2}\asymp k^2$ for $p=2,3$}. \HZ{In particular, the lower
    bounds of these error estimates are established rigorously in the same orders as the upper
    bounds, which is the main novelty of this work.} We show also that the Fourier analysis approach
  can be used as a convenient visual tool for evaluating finite difference schemes in presence of
  source terms, which is beyond the scope of dispersion analysis. \HZ{The results from the theory
    and visual analysis are corroborated by numerical experiments.}
\end{abstract}


\section{Introduction}

The Helmholtz equation $\Delta u + k^2u=f$ is a common model of time-harmonic waves in acoustics
\cite{ihlenburg1998}, geophysics\cite{santos2016} and electrical engineering\cite{ammari2018},
etc. Standard discretization of the Helmholtz equation suffers from the so-called ``\emph{pollution
  effect}'', explained by \cite{deraemaeker1999} as \HZ{that ``the rule of the thumb taking a
  certain number of elements per wavelength'' is insufficient to ``keep the total error under
  control'' albeit ``sufficient to keep the interpolation error constant''.} It is proved in
\cite{ihlenburg1995} that the linear finite element method for the 1D \HZ{scattering} problem on
$(0,1)$ with $u(0)=0$ and $u'(1)-\mathrm{i}ku(1)=0$ has the error in the $H^1$-norm bounded from
above by $C(hk+h^2k^3)\|f\|_{L^2}$ with $C>0$ independent of $h$, $k$ and $f$, where
$Ch^2k^3\|f\|_{L^2}$ is called the ``\emph{pollution term}'' which indicates the accuracy
deterioration with the ``rule of thumb'' keeping $kh$ constant. Since then, a lot of progress has
been made in wavenumber explicit error estimates of finite element methods \cite{melenk2011,
  wu2014pre, zhu2013, du2015, li2019fem, spence2023}, including the recent extension to
heterogeneous media \cite{barucq2017, chaumont2020, graham2020, lafontaine22, CFS25} and multiscale
methods \cite{peterseim2017, freese2021, hauck2022, chaumont2022, ma2023, CFV20}.

While finite element methods are more advanced and powerful, finite difference methods are still
used for their simplicity in some applications \cite{stolk2016, huang2021, aghamiry2022,
  tournier2022, north2023}. In particular, it was found in \cite{tournier2022} for full waveform
inversion that ``\emph{a 27-point stencil on a regular Cartesian grid}'' is ``\emph{more
  computationally efficient}'' than ``\emph{a $P_3$ finite element method on $h$-adaptive
  tetrahedral mesh}''. It is pressing to have a theoretical support for methods in use. But
wavenumber explicit error estimates are rarely seen in the literature of finite difference methods
\cite{singer1998, turkel2013, feng2023, wang2023, wu2024phase}. Error estimates of some compact
fourth order schemes in 2D are given in \cite{fu2008compact} based on the assumption that the
coefficient matrix of the discretized system is positive definite (which is true only for small
wavenumber $k$). In \cite{wang2014}, arbitrarily high order and dispersion free schemes in 1D are
derived, and an error analysis following the approach of \cite{fu2008compact} is given. Recently, in
\cite{cocquet2024} it is shown for a dispersion free 3-point scheme in 1D that the maximum-norm
error is bounded from above by $(\frac{1}{k}\|f''\|_{L^\infty} + k^2\|f\|_{L^\infty})h^2$ up to a
constant factor. {An earlier dispersion free scheme in 1D was proposed in the seminal paper
  \cite{babuska1997} and analyzed in the finite element framework giving the $H^1$-semi-norm error
  less than $ch\|f\|_{H^1}$ for the scattering problem with $c$ seemingly $O(k)$.}

It has been well noted that the pollution effect is intuitively related to the fact that the
numerical solution ``\emph{has a wavelength that is different from the exact one}'', called
``\emph{dispersion}''\cite{deraemaeker1999}. Therefore, \cite{deraemaeker1999} advocated ``\emph{a
  tool in order to measure quantitatively the dispersion for various of the proposed methods as a
  measure of the pollution}''. ``\emph{This measure allows to compare the different methods that
  have been proposed to reduce the pollution and compare their efficiency.}'' The tool called
``\emph{dispersion analysis}'' \cite{abboud1992} had been used before mainly for time-dependent
problems \cite{trefethen1982}.

The idea of dispersion analysis \HZ{in 1D} is to find Fourier modes
\HZ{$\mathrm{e}^{\mathrm{i}{\xi}{x}}$} satisfying the discretized homogeneous Helmholtz equation,
\HZ{e.g.
  $\mathrm{e}^{\mathrm{i}{\xi}{(x-h)}}+(k^2h^2-2)\mathrm{e}^{\mathrm{i}{\xi}{x}}+\mathrm{e}^{\mathrm{i}{\xi}{(x+h)}}=0$,}
and from the resulting constraint on ${\xi}$\HZ{, e.g. $2\cos(\xi h)=2-k^2h^2$}, find the discrete
wavenumber $k^h:=|{\xi}|$. If the original Helmholtz equation
\HZ{$(\mathrm{e}^{\mathrm{i}{\xi}{x}})''+k^2\mathrm{e}^{\mathrm{i}{\xi}{x}}=0$} is used instead of
the discretized one, then we find the constraint $|{\xi}|=k$. Such Fourier modes
\HZ{$\mathrm{e}^{\mathrm{i}{\xi}{x}}$} are called ``plane waves''. The ``phase difference''
$|k-k^h|$ \HZ{measures} the dispersion error between the original plane wave
\HZ{$\mathrm{e}^{\pm\mathrm{i}k{x}}$} and the discretized dispersive ($k^h\ne k$) plane wave
\HZ{$\mathrm{e}^{\pm\mathrm{i}k^h{x}}$}.

It is shown in \cite{ihlenburg1997} for polynomial finite element methods in 1D that the pollution
term is of the same order as the phase difference. That partially explains why reducing the phase
difference leads to reducing the pollution effect. Sharp dispersion error estimates are obtained in
\cite{ainsworth2004}. By providing a practical criterion, dispersion analysis can also guide
optimization of the discretization parameters; see \cite{zhou2023} for the continuous interior
penalty finite element method and \cite{stolk2016, wu2017, cocquet2021, cocquet2024} for finite
difference methods. Recently, dispersion reduction has been generalized in \cite{li2024} to
unstructured meshes by minimizing the residual of the original plane waves in the discretized
equation{; see also the earlier work \cite{quasi2010}}. A similar approach based on minimizing
``\emph{average truncation error of plane waves}'' is adopted in \cite{feng2023} for some high order
finite difference methods.

Albeit intuitive, it is not trivial to translate the ``dispersion error'' into the ``pollution
effect''. As remarked in \cite{ihlenburg1995dispersion}:``\emph{dispersion analysis gives valuable
  information on several physical phenomena inherent to the discrete solution and thus leads to
  qualitative insight into the sources of numerical error},'' but ``\emph{it does not yield, by its
  nature, quantitative statements on the numerical error itself}''. That is why so many efforts have
been made in wavenumber explicit error estimation, as we mentioned above.

Moreover, dispersion analysis is insufficient because it considers only the zero source problem and
the Fourier modes constrained by the dispersion relation, while nonzero source terms and the other
Fourier modes stimulated by the source term are not taken into account. We shall see in
Section~\ref{visual} that three dispersion free schemes in 1D can behave differently with a source
term. Indeed, in \cite{babuska1997} a discretization in 1D was proposed with zero truncation error
``\emph{for as many right hand sides as possible}'', not only the zero right hand side. An
eigenvalue (Fourier) analysis is carried out in \cite{dwarka2021} for the Dirichlet boundary value
problem \emph{with a point source}.

Inspired by \cite{dwarka2021} which calculated the eigenvalue difference of the continuous and
discrete Helmholtz operators, we develop here an approach to wavenumber explicit error estimation
based on Fourier analysis for general smooth source terms, Dirichlet boundary conditions \HZ{and
  non-resonant wavenumber\MG{s} $k\not\in\pi\mathbb{N}$} in 1D. We analyze in detail a classical \MG{centered}
\HZ{\emph{finite difference}} scheme in 1D because it is simple yet not fully
understood.  
Actually, we have not found any other wavenumber explicit estimate for \emph{finite difference}
schemes than that in \cite{cocquet2024}, and in particular none for the classical \MG{centered} \HZ{\emph{finite
    difference}} scheme, though the folklore says a $p$th order \HZ{\emph{finite
    difference}} scheme has the (relative) error $O(k^{p+1}h^p)$ (see
e.g. \cite{dwarka2021}).

\HZ{It should be well noted that such results have been proved and considered as standard for
  \emph{finite element} methods. For example, as pointed out by \MG{an anonymous reviewer}, for
  arbitrary order finite element methods in \MG{a} general domain $\Omega\subset\mathbb{R}^d$,
  combining \cite[Theorem 4.11]{CFS25} with the estimate $\gamma_s^*\le Ck/\delta$ gives a sharp
  upper bound for the Dirichlet problem.}

\HZ{Regarding our work here for the classical \MG{centered} finite difference scheme in 1D,} all the
upper bounds are attainable and the equality cases are given explicitly. The pollution term is
located in Fourier frequencies, and its \emph{exact} order in $k$ and $h$ is found \emph{by
  two-sided bounds}. \HZ{In particular, the lower bounds contribute to the main novelty of this
  work.}
We also give a sharp estimate \MG{for} relative errors with source terms. Wavenumber explicit
estimates \MG{for} relative errors have \HZ{not been studied as much as} \MG{for} absolute
errors. Recent progress for finite element methods can be found in \cite{lafontaine2022sharp,
  galkowski2023}.

Moreover, by plotting the symbol errors we can see clearly how the error is distributed over Fourier
frequencies. In this sense, it becomes a practical tool for comparing different schemes, e.g. those
of the same order and those before and after dispersion correction, which complements the dispersion
analysis for working with or without a source term. This aspect has been briefly demonstrated in
\cite{zhang2025} {for some compact high order schemes, where sharp orders in $k$, $h$ are found
  in Fourier domain and verified numerically}.

In the following sections, we first introduce the model problem and Fourier decomposition in
Section~\ref{model}, then carry out the theoretical analysis of the classical scheme in
Section~\ref{anal}, and illustrate the practical aspect of the tool in
Section~\ref{visual}. \HZ{Those results are corroborated numerically in Section~\ref{numer}.} Some
discussions are included at the end.

\section{Model problem in 1D and downsampling error}
\label{model}

We first introduce the 1D Helmholtz equation in a closed cavity, and the Fourier form of the
solution. Then we analyze the component of the discretization error due to sampling of the solution
on grids.

\subsection{Model problem}

Let $k>0$ be the wavenumber. The Helmholtz equation with Dirichlet boundary conditions in 1D is
\begin{equation}\label{1dhelm}
u''+k^2u=f \text{ in }(0,1),\quad u(0)=g_0,\quad u(1)=g_1.
\end{equation}
For well-posedness of \eqref{1dhelm}, we assume that $k/\pi$ is not an integer.  By linearity, we
separate the two cases:
\begin{itemize}
\item[(i)] homogeneous Dirichlet boundary condition $g_0=g_1=0$, and
\item[(ii)] zero source $f=0$.
\end{itemize}
In case (i) $g_0=g_1=0$, we assume $f\in H_0^{p}(0,1)$ with $p\ge 1$.  \HZ{For $f\in L^2(0,1)$, we
  have $u\in H^2(0,1)\cap H_0^1(0,1)$.}  
For the corresponding theory, we refer to \cite{evans2022}. Smoothness of $f$ \HZ{is a crucial
  assumption and needed in our } analysis of discretization schemes for \eqref{1dhelm}.  For
example, \cite{babuska1997} used the condition $f\in H^1(0,1)$. Here $f\in H_0^p(0,1)$ and
$u\in H^2(0,1)\cap H_0^1(0,1)$ admit odd extension to the domain $(-1,1)$ and the convergent sine
series in the corresponding spaces $H^p(-1,1)$ and $H^2(-1,1)$:
\[
  \begin{aligned}
    f(x)=\sum_{(\xi/\pi)=1}^{\infty}\hat{f}(\xi)\sin(\xi x),\quad
    \hat{f}(\xi) = 2\int_{0}^1f(x)\sin(\xi x)~\mathrm{d}x,\\
    u(x)=\sum_{(\xi/\pi)=1}^{\infty}\hat{u}(\xi)\sin(\xi x),\quad
    \hat{u}(\xi) = 2\int_{0}^1u(x)\sin(\xi x)~\mathrm{d}x.
    \end{aligned}
\]
From \eqref{1dhelm} and the sine series, we have $(k^2-\xi^2)\hat{u}(\xi)=\hat{f}(\xi)$.
Well-posedness of \eqref{1dhelm} amounts to $\lambda(\xi):=k^2-\xi^2\ne0$ for all
$\xi\in\pi\mathbb{N}$. In case (ii) $f=0$, the exact solution is
\begin{equation}\label{udiri}
  u(x)=\frac{g_1-g_0\cos k}{\sin k}\sin(kx)+g_0\cos(kx).
\end{equation}
To analyze the discretization error for the homogeneous problem (ii), one just needs to find the
closed form solution of the discretized problem. Usually, the discretization incurs a perturbation
of the wavenumber $k$ in the solution basis $\{\sin(kx), \cos(kx)\}$, which explains the error, and
requires dispersion analysis.

\subsection{Downsampling error}\label{secdown}

We shall discretize \eqref{1dhelm} on uniform grids with grid points $x_j:=jh$, $j=0,\ldots,N$ where
$Nh=1$. \HZ{Suppose $\mu:=\frac{kh}{2}\le C_{\mu}$ for a constant $C_{\mu}\in(0,\frac{\pi}{2})$.} In
a finite difference scheme, $f$ is simply evaluated at the grid points $x_j$ to give
$f^h(x_j):=f(x_j)$.  This causes aliasing of the sinusoidal modes in case (i), and gives the
discrete sine transform
\[
  \begin{aligned}
  f^h(x_j)=\sum_{(\xi/\pi)=1}^{N-1}\widehat{f^h}(\xi)\sin(\xi x_j),\quad
  \widehat{f^h}(\xi):=2\sum_{j=1}^{N-1}h\sin(\xi x_j)f^h(x_j)=
  \hat{f}(\xi)+\sum_{s=\pm1}s\sum_{m=1}^{\infty}\hat{f}(s\xi+2mN\pi),
\end{aligned}
\]
where the last equality comes from the fact $f^h(x_j)=f(x_j)$ and the rearrangement of
\[
  f(x_j)=\sum_{(\xi/\pi)=1}^{\infty}\hat{f}(\xi)\sin(\xi x_j)=
  \sum_{(\xi/\pi)=1}^{N-1}\left(\hat{f}(\xi)+\sum_{s=\pm1}s\sum_{m=1}^{\infty}\hat{f}(s\xi+2mN\pi)
  \right)\sin(\xi x_j)
\]
permitted by the absolute convergence
\[
  \sum_{(\xi/\pi)=1}^{\infty}|\hat{f}(\xi)|\le \sqrt{\sum_{(\xi/\pi)=1}^{\infty}\frac{1}{\xi^2}}
  \cdot \sqrt{\sum_{(\xi/\pi)=1}^{\infty}|\xi\hat{f}(\xi)|^2}<\infty.
\]
Indeed, it is well known that
$
  \sum_{n=1}^{\infty}\frac{1}{n^2}=\frac{\pi^2}{6},
$
and by Parseval's identity
$
\sum_{(\xi/\pi)=1}^{\infty}|\xi\hat{f}(\xi)|^2=2\int_0^1|f'|^2~\mathrm{d}x.
$

Let \HZ{$u^h(x_j)$} be the finite difference solution defined \HZ{by some scheme,
  e.g. \eqref{1d3pt}}. To calculate the error $e^h:=u-u^h$ for case (i), we \HZ{interpolate
  $u^h(x_j)$} to the domain $\HZ{x\in}(0,1)$ by the discrete sine transform\footnote{{The sine
    interpolation is chosen to facilitate our analysis because the continuous solution $u$ is
    expanded to sine series. In \cite{babuska1997}, piecewise linear interpolation is used for
    analysis of a linear finite element method.}}
\[
  u^h(x)=\sum_{(\xi/\pi)=1}^{N-1}\widehat{u^h}(\xi)\sin(\xi x),\quad x\in (0,1),\quad
  \widehat{u^h}(\xi):=2\sum_{j=1}^{N-1}h\sin(\xi x_j)u^h(x_j).
\]
So the error for case (i) can be expanded \MG{in} the sine basis into
\begin{equation}\label{errpart}
  e^h(x)=\sum_{(\xi/\pi)=1}^{N-1}(\hat{u}-\widehat{u^h})(\xi)\sin(\xi x) +
  \sum_{(\xi/\pi)=N}^{\infty}\hat{u}(\xi)\sin(\xi x)=:e_1^h(x)+e_2^h(x).
\end{equation}
We first estimate the $H^1$-semi-norm of $e_2^h$. Note that
\begin{equation}\label{e2p}
  (e_2^h)'(x)=\sum_{(\xi/\pi)=N}^{\infty}\hat{u}(\xi)\xi\cos(\xi x)=
  \sum_{(\xi/\pi)=N}^{\infty}\frac{\xi}{k^2-\xi^2}\hat{f}(\xi)\cos(\xi x).
\end{equation}
Parseval's identity for the above cosine series reads
\begin{equation}\label{e2pnorm}
  |e_2^h|_1^2:=\int_0^1|(e_2^h)'|^2~\mathrm{d}x=\frac{1}{2}
  \sum_{(\xi/\pi)=N}^{\infty}\frac{\xi^2}{(k^2-\xi^2)^2}|\hat{f}(\xi)|^2.
\end{equation}
By our assumption that $\frac{kh}{2}\le C_{\mu}<\frac{\pi}{2}$, it holds $k<N\pi$, and for
$\xi\ge N\pi$,
\[
  \frac{\xi^2}{(k^2-\xi^2)^2}\le \frac{(N\pi)^2}{((N\pi)^2-k^2)^2}=\frac{\pi^2h^2}{(\pi^2-k^2h^2)^2}
  \le\frac{\pi^2h^2}{(\pi^2-4C_{\mu}^2)^2}.
\]
Hence, we obtain
\[
  |e_2^h|_1^2\le \frac{\pi^2 h^2}{(\pi^2-4C_{\mu}^2)^2}\cdot
  \frac{1}{2}\sum_{(\xi/\pi)=N}^{\infty}|\hat{f}(\xi)|^2.
\]
We \HZ{distinguish} the quantities \HZ{\MG{that can and can not be resolved} on the mesh}, 
\begin{equation}\label{flowhigh}
  f_{\mathrm{low}}:=\sum_{(\xi/\pi)=1}^{N-1}\hat{f}(\xi)\sin(\xi x),\quad
  f_{\mathrm{high}}:=\sum_{(\xi/\pi)=N}^{\infty}\hat{f}(\xi)\sin(\xi x),\quad
  \|f\|:=\int_0^1|f(x)|^2~\mathrm{d}x,
\end{equation}
\HZ{where low and high are w.r.t. $\pi/h$, but not $k$
\MG{(similarly, we also distinguish $u_{\mathrm{low}}$
and $u_{\mathrm{high}}=e_2^h$)}. An ideal mesh is such that the low
  frequencies cover $k$ i.e. $k<\pi/h$, and $f_{\rm{high}}$ is small compared to $f_{\rm{low}}$.}
Then, it follows that
\[
  |e_2^h|_1\le \frac{\pi h}{\pi^2-4C_{\mu}^2}\|f_{\mathrm{high}}\|.
\]
Using higher regularity $f\in H_0^p(0,1)$ {allows the split
  $\frac{\xi^{2}}{(k^2-\xi^2)^2}\frac{1}{\xi^{2p}}|\xi^p\hat{f}(\xi)|^2$ in \eqref{e2pnorm} and}
yields
\begin{equation}\label{downerr}
  |e_2^h|_1\le \frac{\HZ{h(kh)^{p}}}{(\pi^2-4C_{\mu}^2)\pi^{p-1}} \HZ{k^{-p}}|f_{\mathrm{high}}|_p,\quad
  |f|_p:=\|f^{(p)}\|\text{ with $f^{(p)}$ the $p$th derivative},  
\end{equation}
and also the $L^2$-norm of the downsampling error estimate
\begin{equation}\label{downerrL2}
  \|e_2^h\|\le \frac{\HZ{h^2(kh)^{p}}}{(\pi^2-4C_{\mu}^2)\pi^{p}} \HZ{k^{-p}}|f_{\mathrm{high}}|_p.
\end{equation}

\begin{remark}\label{fpartsconv}
  Note that the separation $f=f_{\mathrm{low}}+f_{\mathrm{high}}$ depends on the mesh size
  $h=1/N$. We have $f_{\mathrm{low}}\to f$ and $f_{\mathrm{high}}\to 0$ in $H^p(0,1)$ as
  $N\to \infty$. By Parseval's identity or orthogonality of the basis, we see that
  $|f|_l^2=|f_{\mathrm{low}}|_l^2+|f_{\mathrm{high}}|_l^2$ for all $l=0,..,p$. {If
    $|f_{\mathrm{high}}|_p=O(k^p)$ and $kh=O(1)$, then \eqref{downerr} and \eqref{downerrL2} say
    $|e_2^h|_1=O(h)$ and $\|e_2^h\|=O(h^2)$, so the higher regularity does not improve the
    estimates.}
\end{remark}

The estimation of $e^h_1$ in case (i) and the error in case (ii) both depend on the particular
finite difference schemes, which we shall discuss as follows.

\section{Analysis of classical 3-point centered scheme}
\label{anal}

The scheme uses the approximation $u''(x)\approx (u(x+h)-2u(x)+u(x-h))/h^2$ to find $u^h$ satisfying
\begin{equation}\label{1d3pt}
  (k^2-\frac{2}{h^2})u^h(x_j) + \frac{1}{h^2}u^h(x_{j-1})+\frac{1}{h^2}u^h(x_{j+1})=f(x_j),\quad
  u^h(x_0)=g_0,\quad u^h(x_N)=g_1.
\end{equation}
In case (i), $g_0=g_1=0$, applying the discrete sine transform to the above equation gives
\[
  \left(k^2-\frac{4}{h^2}\sin^2\frac{\xi h}{2}\right)\widehat{u^h}(\xi)=\widehat{f^h}(\xi)
  =\hat{f}(\xi)+\sum_{s=\pm1}s\sum_{m=1}^{\infty}\hat{f}(s\xi+2mN\pi),\;\;\xi/\pi=1,..,N-1.
\]
In case (ii), $f=0$, the solution of \eqref{1d3pt} can be found in closed form,
\begin{equation}\label{uhdiri}
  u^h(x_j) = \frac{g_1-g_0\cos k^h}{\sin k^h}\sin(k^hx_j)+g_0\cos(k^hx_j),\quad x_j=jh,\;j=0,1,..,N,
\end{equation}
where $k^h:=\frac{2}{h}\arcsin\frac{kh}{2}$ is the discrete wavenumber. The \HZ{r.h.s. of
  \eqref{uhdiri}} is valid if and only if $k^h\not\in\mathbb{Z}\pi$ which we assume to hold.
\HZ{The problem \eqref{1d3pt} is well-posed if and only if $kh\in [2,\infty)$ or
  $k^h\not\in\{1,..,N-1\}\pi$. Indeed, \eqref{1d3pt} can be written as an $(N\!-\!1)\times(N\!-\!1)$
  linear system whose coefficient matrix has the eigenvalues
  $k^2\!-\!\frac{4}{h^2}\sin^2\frac{\xi h}{2}$ and the eigenvectors $\sin(\xi x_j)$ for
  $\xi\in\{1,..,N\!-\!1\}\pi$.} The formula \eqref{uhdiri} \MG{permits} the grid function $u^h$ to
be densely defined in $[0,1]$. Then the error $e^h(x):=u(x)-u^h(x)$ can be explicitly written as
\begin{equation}\label{A0A1}
    e^h(x)\!=\!A_0(x)g_0 \!+\! A_1(x)g_1, \quad
    A_1(x)\!=\!\frac{\sin(kx)}{\sin k}\!-\!\frac{\sin (k^hx)}{\sin k^h},\quad
    A_0(x)=A_1(1-x).
\end{equation}
We shall first analyze the case (ii), $f=0$, in subsection~\ref{secdiri}, and then case (i),
$g_0=g_1=0$, in the remaining subsections.

\subsection{Error for the Dirichlet problem with zero source}
\label{secdiri}

From \eqref{A0A1}, we have
\[
  \|e^h\| \le \|A_0\|\cdot |g_0| + \|A_1\|\cdot |g_1| = \|A_1\|(|g_0|+|g_1|).
\]
This inequality is sharp by considering $g_0=0$ or $g_1=0$. So the key is to estimate $\|A_1\|$. We
calculate
\[
  \int_0^1\frac{\sin^2(kx)}{\sin^2k}~\mathrm{d}x=
  \frac{1}{\sin^2k}\int_0^1\frac{1-\cos(2kx)}{2}~\mathrm{d}x=
  \frac{1}{2\sin^2k}\left(1-\frac{1}{2k}\sin(2k)\right),
\]
\[
  \int_0^1\frac{\sin^2(k^hx)}{\sin^2k^h}~\mathrm{d}x=
  \frac{1}{2\sin^2k^h}\left(1-\frac{1}{2k^h}\sin(2k^h)\right),
\]
\[
  \begin{aligned}
  \int_0^1\frac{\sin(kx)\sin(k^hx)}{(\sin k)\sin k^h}~\mathrm{d}x&=
  \frac{1}{(\sin k)\sin k^h}\frac{1}{2}\int_0^1\cos((k-k^h)x)-\cos((k+k^h)x)~\mathrm{d}x\\
  &=\frac{1}{2(\sin k)\sin k^h}\left(\frac{\sin(k^h-k)}{k^h-k}-\frac{\sin(k+k^h)}{k+k^h}\right).
  \end{aligned}
\]
Therefore, we have
\[
  \begin{aligned}
    \|A_1\|^2=&\int_0^1\left[\frac{\sin(kx)}{\sin k}-\frac{\sin(k^hx)}{\sin k^h}\right]^2~\mathrm{d}x\\
    =&\frac{1}{2\sin^2k}+\frac{1}{2\sin^2k^h}-\frac{\cos k}{2k\sin k}-\frac{\cos k^h}{2k^h\sin k^h}
    -\frac{1}{(\sin k)\sin k^h}\left(\frac{\sin(k^h-k)}{k^h-k}-\frac{\sin(k+k^h)}{k+k^h}\right)\\
    =&\frac{1}{(\sin k)\sin k^h}\left[\frac{1}{2}\frac{\sin k^h}{\sin k}\!+\!
    \frac{1}{2}\frac{\sin k}{\sin k^h}\!-\!\frac{(\sin k^h)\cos k}{2k}\!-\!
    \frac{(\sin k)\cos k^h}{2k^h}\!-\!\frac{\sin(k^h\!-\!k)}{k^h\!-\!k}\!+\!
    \frac{\sin(k\!+\!k^h)}{k\!+\!k^h}\right]\\
    =:&\frac{1}{(\sin k)\sin k^h}(S_1\!+\!S_2),
  \end{aligned}
\]
where we separated the terms in the brackets into two parts
\begin{equation}\label{S1S2}
  S_1:=\frac{1}{2}\frac{\sin k^h}{\sin k}+ %
  \frac{1}{2}\frac{\sin k}{\sin k^h}-\frac{\sin(k^h-k)}{k^h-k},\quad
  S_2:=-\frac{(\sin k^h)\cos k}{2k}- \frac{(\sin k)\cos k^h}{2k^h}+\frac{\sin(k+k^h)}{k+k^h},
\end{equation}
{where $S_1$ mainly accounts for the dispersion error, and $S_2$ is small for large $k$}.  Note
that
\[
  \begin{aligned}
    \frac{\sin k^h}{\sin k}&=\frac{\sin(k^h-k+k)}{\sin k}=\frac{\sin(k^h-k)\cos k}{\sin k} +
    \cos(k^h-k),\\
    \frac{\sin k}{\sin k^h}&=\frac{\sin(k-k^h+k^h)}{\sin k^h}=\frac{\sin(k-k^h)\cos k^h}{\sin k^h} +
    \cos(k-k^h).
  \end{aligned}
\]
So we have
\begin{equation}\label{S1}
  S_1=\frac{1}{2}\sin(k^h-k)\left(\cot k-\cot k^h\right) + \cos(k^h-k) - \frac{\sin(k^h-k)}{k^h-k}.
\end{equation}
We have also
\begin{align}
  S_2=&-\frac{(\sin k^h)\cos k}{2k}- \frac{(\sin k)\cos k^h}{2k^h}+\frac{(\sin k)\cos k^h+(\sin k^h)\cos k}{k+k^h}\nonumber\\
  =&\frac{k^h-k}{2(k+k^h)}\left[\frac{(\sin k)\cos k^h}{k^h}-\frac{(\sin k^h)\cos k}{k}\right]\nonumber\\
  =&\frac{k^h-k}{2(k+k^h)}\left[\frac{(\sin k)\cos k^h-(\sin k^h)\cos k}{k} +
     \left(\frac{1}{k^h}-\frac{1}{k}\right)(\sin k)\cos k^h\right]\nonumber\\
  =&\frac{k^h-k}{2(k+k^h)}\left[\frac{\sin(k-k^h)}{k}+\frac{k-k^h}{kk^h}(\sin k)\cos k^h\right].\label{S2}
\end{align}

\begin{lemma}\label{sink}
  Let $\sigma_k:=\min_{\xi\in \pi\mathbb{N}}|k-\xi|>0$. Then $ |\sin k|\HZ{=\sin\sigma_k\in (\frac{2}{\pi},1)\sigma_k}.  $
\end{lemma}

\begin{proof}
  Let $n_k\in \mathbb{N}$ be such that $|k-n_k\pi|\le \frac{\pi}{2}$. Since
  $\sin x\ge \frac{2}{\pi}x$ for $0\le x\le \frac{\pi}{2}$, we have
  \[
    |\sin k|=\sin |k-n_k\pi|\HZ{=\sin\sigma_k\in \left(\frac{2}{\pi},1\right)\sigma_k},
  \]
  where in the last equation we used the definition of $\sigma_k$.
\end{proof}

\begin{lemma}\label{kkh}
  Let $k^h:=\frac{2}{h}\arcsin\frac{kh}{2}$ and \HZ{$0<c<1$}. If $0<\frac{kh}{2}<1$, then
  $\frac{1}{24}k^3h^2<k^h-k<\frac{\pi-2}{8}k^3h^2$. If in addition
  $\sigma_k:=\min_{\xi\in \pi\mathbb{N}}|k-\xi|>0$ and $k^3h^2<\frac{8c\sigma_k}{\pi -2}$, then
  $k^h-k<c\sigma_k$ and $|\sin k^h|\HZ{\in(\frac{2(1-c)}{\pi},1+c)\sigma_k}$.
\end{lemma}

\begin{proof}
  Since $x+\frac{x^3}{6}<\arcsin x<x + (\frac{\pi}{2}-1)x^3$ for $0<x<1$, we have
  $k\!+\!\frac{1}{24}k^3h^2<k^h< k\!+\! \frac{\pi\!-\!2}{8}k^3h^2<k\!+\!c\sigma_k$.  \HZ{Then
  $|\xi\!-\!k|+|k^h\!-\!k|>|\xi\!-\!k^h|>|\xi\!-\!k|-|k^h\!-\!k|$ gives
  $\min_{\xi\in\pi\mathbb{N}}|\xi\!-\!k^h|\in(1-c,1+c)\sigma_k$.  So
  $|\sin k^h|\in (\frac{2(1\!-\!c)}{\pi},1+c)\sigma_k$}.
\end{proof}

\begin{remark}
  It is important to require $k^3h^2$ being small for convergence of the finite difference
  scheme. Otherwise, Lemma~\ref{kkh} says that $k^h-k$ can be large. For example, when $k$ is large
  and $k^h-k=\pi$, we have $S_1=-1$, $S_2<\frac{\pi^2}{4k^3}$ and $\|A_1\|>1-\frac{\pi^2}{4k^3}$.
\end{remark}

\begin{lemma}\label{S1bounds}
  Under the assumptions of Lemma~\ref{kkh}, for $S_1$ in \eqref{S1} it holds that
  \begin{equation}\label{S1ineq}
    \HZ{\nu(\sigma_k)}(k^h-k)^2<S_1< \left(\frac{\pi^2}{\HZ{8(1-c)^2}\sigma_k^2}+\frac{1}{6}\right)(k^h-k)^2\HZ{,}
  \end{equation}
  \HZ{where, if $(1+c)\sigma_k>\frac{\pi}{2}$, $\nu(\sigma_k):=\frac{1}{8}$, and if
  $(1+c)\sigma_k\le\frac{\pi}{2}$,}
  \[
    \HZ{\nu(\sigma_k):=\max\left\{\frac{1}{\pi(1+c)^2\sigma_k^2}-\frac{3}{8},~~\frac{1}{8}\right\}}.
  \]
\end{lemma}

\begin{proof}
  Let $k\in ((n_k^+-1)\pi,n_k^+\pi)$ for an integer $n_k^+$. By Lemma~\ref{kkh}, it holds that
  $0<k^h-k<c\sigma_k<\frac{\pi}{2}$. Moreover,
  $(n_k^+-1)\pi+\sigma_k\le k <k^h<k+c\sigma_k\le n_k^+\pi-(1-c)\sigma_k$. Hence
  $(n_k^+-1)\pi<k<k^h<n_k^+\pi$, \HZ{$(n_k^+-\frac{1}{2})\pi\in (k,k^h)$ only if $(1+c)\sigma_k>\frac{\pi}{2}$,} and $\cot k - \cot k^h = \frac{k^h-k}{\sin^2\tilde{k}}$ for some
  $\tilde{k}\in(k,k^h)$. Since
  \[
    \HZ{\max\left\{x-\frac{1}{6}x^3,\frac{2}{\pi}x\right\}}<\sin x<x-\frac{1}{8}x^3,\quad 1-\frac{1}{2}x^2<\cos x<1,\quad
    \text{ for }x\in\left(0,\frac{\pi}{2}\right),
  \]
  it can be deduced from \eqref{S1} that
  \[
    S_1<\frac{1}{2}\frac{(k^h-k)^2}{\HZ{\min\{\sin^2k,\sin^2k^h\}}} + 1 -
    \left(1-\frac{1}{6}(k^h-k)^2\right)\HZ{\text{;\;\; and if $(1+c)\sigma_k\le\frac{\pi}{2}$,}}
  \]
  \[
    \HZ{S_1>\frac{1}{2}\frac{2}{\pi}\frac{(k^h-k)^2}{\max\{\sin^2k,\sin^2k^h\}}+1-\frac{1}{2}(k^h-k)^2-(1-\frac{1}{8}(k^h-k)^2)}.\qquad\qquad
  \]
  \HZ{By Lemma~\ref{sink} \MG{and} \ref{kkh}}, \MG{the} two bounds in \eqref{S1ineq} are obtained. The other lower
  bound is derived from \eqref{S1S2},
  \[
    S_1> 1 - \left(1-\frac{1}{8}(k^h-k)^2\right) = \frac{1}{8}(k^h-k)^2,
  \]
  where $0<k^h-k<\frac{\pi}{2}$, $(n_k^+-1)\pi<k<k^h<n_k^+\pi$ and $x+\frac{1}{x}\ge 2$ for $x>0$
  are used.
\end{proof}

\begin{lemma}\label{S2bounds}
  Let $k^h:=\frac{2}{h}\arcsin\frac{kh}{2}$ and $S_2$ given in \eqref{S2}. If $0<\frac{kh}{2}<1$,
  then it holds that
  \[
    -\left(\frac{1}{4k^2}+\frac{1}{4k^3}\right)(k^h-k)^2<S_2<\left(\frac{1}{4k^2}+\frac{1}{4k^3}\right)(k^h-k)^2.
  \]
\end{lemma}

\begin{proof}
  The conclusion follows directly from $k^h>k>0$, $|\sin x|\le |x|$, $|\sin x|\le 1$ and
  $|\cos x|\le 1$.
\end{proof}

\begin{theorem}[$L^2$ Error for $f=0$]
  Under the assumptions of Lemma~\ref{kkh}, for the $L^2$-norm of $A_1$ in \eqref{A0A1},
  it holds that
  \[
    \frac{1}{24\HZ{\sigma_k\sqrt{1+c}}}\sqrt{\HZ{\nu(\sigma_k)}-\frac{1}{4k^2}-\frac{1}{4k^3}}k^3h^2<\|A_1\|<
    \frac{\pi(\pi-2)}{16\sigma_k\sqrt{1-c}}\sqrt{\frac{\pi^2}{\HZ{8(1-c)^2}\sigma_k^2}+\frac{1}{6}+\frac{1}{4k^2}+\frac{1}{4k^3}}k^3h^2,
  \]
  where $\sigma_k=\min_{\xi\in \pi\mathbb{N}}|k-\xi|>0$, $k^3h^2\frac{(\pi-2)}{8\sigma_k}<c<1$ \HZ{and
  $\nu(\sigma_k)$ is defined in Lemma~\ref{S1bounds}}.  Moreover, the error $e^h=u-u^h$ satisfies
  $\|e^h\|\le \|A_1\|(|g_0|+|g_1|)$ with equality attained when $g_0=0$ or $g_1=0$, and the relative
  error satisfies
  \[
    \frac{\|e^h\|}{\|u\|}\le \frac{\sqrt{2}|\sin k|}{\sqrt{1-\frac{\sin 2k}{2k}}}\|A_1\|
  \]
  with equality attained when $g_0=0$ or $g_1=0$. \HZ{Note also that $|\sin k|\asymp \sigma_k$ by
  Lemma~\ref{sink}.}
\end{theorem}

\begin{proof}
  The absolute error estimate is obtained by combining Lemmas~\ref{sink}, \ref{kkh}, \ref{S1bounds}
  and \ref{S2bounds}. The relative error estimate follows from
  \[
    \|u\|\ge \left| |g_1|\cdot\left\|\frac{\sin(kx)}{\sin k}\right\| - %
      |g_0|\cdot\left\|\frac{\sin(k(1-x))}{\sin k}\right\|\right|
  \]
  with equality attained when $g_0=0$ or $g_1=0$.
\end{proof}

It is also interesting to estimate the $H^1$-semi-norm of $A_1(x)$ in \eqref{A0A1}. Some
calculations lead to
\[
  |A_1|_1^2=\int_0^1\left(k\frac{\cos(kx)}{\sin k}-k^h\frac{\cos(k^hx)}{\sin k^h}\right)^2~\mathrm{d}x=\frac{kk^h}{\sin k\sin k^h}(\tilde{S}_1+\tilde{S}_2),
\]
where we introduced
\begin{equation}\label{S1S2p}
  \tilde{S}_1:=\frac{1}{2}\frac{k\sin k^h}{k^h\sin k}+ %
  \frac{1}{2}\frac{k^h\sin k}{k\sin k^h}-\frac{\sin(k^h-k)}{k^h-k},\quad %
  \tilde{S}_2:=\frac{(\sin k^h)\cos k}{2k^h} + \frac{(\sin k)\cos k^h}{2k}-\frac{\sin(k+k^h)}{k+k^h}.
\end{equation}
More calculations as previously done for $S_1$ and $S_2$ yield
\begin{multline}\label{S1p}
  \tilde{S}_1=\frac{1}{2}\sin(k^h-k)\left(\cot k-\cot k^h\right) + \cos(k^h-k) - %
  \frac{\sin(k^h-k)}{k^h-k}\\
  \qquad + (k^h-k)\left[-(\sin(k^h-k))\left(\frac{\cot k}{2k^h}+\frac{\cot k^h}{2k} %
    \right) + (\cos(k^h-k))\frac{k^h-k}{2kk^h}\right],
\end{multline}
\begin{equation}\label{S2p}
  \tilde{S}_2=\frac{k^h-k}{2(k+k^h)}\left[\frac{\sin(k-k^h)}{k}+\frac{k^h-k}{kk^h}(\sin k^h)\cos k\right].
\end{equation}

\begin{lemma}\label{S1pbounds}
  Under the assumptions of Lemma~\ref{kkh}, for $\tilde{S}_1$ in \eqref{S1p} it holds that
  \[
    \HZ{\tilde{\nu}(\sigma_k,k)}(k^h-k)^2<\tilde{S}_1< \left(\frac{\pi^2}{\HZ{8(1-c)^2}\sigma_k^2}+\frac{1}{6}+ %
      \frac{\pi(2-c)}{4\sigma_k(1-c)k}+\frac{1}{2k^2}\right)(k^h-k)^2,
  \]
  \HZ{where, if $(1+c)\sigma_k>\frac{\pi}{2}$, $\tilde{\nu}(\sigma_k,k):=\frac{1}{8}$, and if
  $(1+c)\sigma_k\le\frac{\pi}{2}$,}
  \[
    \HZ{\tilde{\nu}(\sigma_k,k):=\max\left\{\frac{1}{\pi(1+c)^2\sigma_k^2}-\frac{3}{8}-\frac{\pi(2-c)}{4\sigma_k(1-c)k}+\frac{2-c^2\sigma_k^2}{8k^2},~~\frac{1}{8}\right\}}.
  \]
\end{lemma}

\begin{proof}
  The lower bound \HZ{$\frac{1}{8}(k^h-k)^2$} is derived from \eqref{S1S2p} similarly to the proof
  of Lemma~\ref{S1bounds}. For the \HZ{other bounds}, note that the first line of \eqref{S1p} equals
  $S_1$, and \HZ{by Lemma~\ref{sink} \MG{and} \ref{kkh}} the remaining line satisfies
  \[
    \begin{aligned}
      &\,(k^h-k)\left[-(\sin(k^h-k))\left(\frac{\cot k}{2k^h}+\frac{\cot
            k^h}{2k} %
        \right) + (\cos(k^h-k))\frac{k^h-k}{2kk^h}\right]\\
      \HZ{\in}&\,(k^h-k)^2\left(\HZ{-\frac{\pi(2-c)}{4\sigma_k(1-c)k}+\frac{1-\frac{1}{2}(k^h-k)^2}{2kk^h}},~~       \frac{\pi(2-c)}{4\sigma_k(1-c)k}+\frac{1}{2k^2}\right).
    \end{aligned}
  \]
  \HZ{We conclude by using $k^h-k<c\sigma_k$, $k^h=\frac{2}{h}\arcsin\frac{kh}{2}<\frac{\pi}{2}k<2k$
    and} Lemma~\ref{S1bounds}.
\end{proof}

\begin{lemma}\label{S2pbounds}
  Let $k^h:=\frac{2}{h}\arcsin\frac{kh}{2}$ and $\tilde{S}_2$ be given in \eqref{S2p}. If
  $0<\frac{kh}{2}<1$, then it holds that
  \[
    -\left(\frac{1}{4k^2}+\frac{1}{4k^3}\right)(k^h-k)^2<\tilde{S}_2<\left(\frac{1}{4k^2}+\frac{1}{4k^3}\right)(k^h-k)^2.
  \]
\end{lemma}

\begin{proof}
  The conclusion follows directly from $k^h>k>0$, $|\sin x|\le |x|$, $|\sin x|\le 1$ and
  $|\cos x|\le 1$.
\end{proof}

\begin{theorem}[$H^1$ error for $f=0$]
  Under the assumptions of Lemma~\ref{kkh}, for $A_1$ in \eqref{A0A1}, \HZ{we have}
  {\small\[
    \frac{1}{24\HZ{\sigma_k\sqrt{1+c}}}\sqrt{\HZ{\tilde{\nu}(\sigma_k,k)}\!-\!\frac{1}{4k^2}\!-\!\frac{1}{4k^3}}k^4h^2\!<\!|A_1|_1\!<\!\frac{\pi(\pi-2)}{16\sigma_k\sqrt{1-c}}\sqrt{\frac{\pi^2}{\HZ{4(1-c)^2}\sigma_k^2}\!+\!\frac{1}{3}\!+\!\frac{\pi(2-c)}{2\sigma_k(1-c)k}\!+\!\frac{3}{2k^2}\!+\!\frac{1}{2k^3}}k^4h^2,
  \]}\noindent
  where $\sigma_k=\min_{\xi\in \pi\mathbb{N}}|k-\xi|>0$, $k^3h^2\frac{(\pi-2)}{8\sigma_k}<c<1$, and
  \HZ{$\tilde{\nu}(\sigma_k,k)$ is defined in Lemma~\ref{S1pbounds}}.  Moreover, the error $e^h=u-u^h$ satisfies $|e^h|_1\le |A_1|_1(|g_0|+|g_1|)$ with equality
  attained when $g_0=0$ or $g_1=0$, and the relative error satisfies
  \[
    \frac{|e^h|_1}{|u|_1}\le \frac{\sqrt{2}|\sin k|}{\sqrt{1+\frac{\sin 2k}{2k}}}\cdot
    \frac{|A_1|_1}{k}
  \]
  with equality attained when $g_0=0$ or $g_1=0$. \HZ{Note also that $|\sin k|\asymp \sigma_k$ by
    Lemma~\ref{sink}.}
\end{theorem}

\begin{proof}
  The conclusion is obtained by combining Lemmas~\ref{sink}, \ref{kkh}, \ref{S1pbounds},
  \ref{S2pbounds}, and $k^h=\frac{2}{h}\arcsin\frac{kh}{2}<\frac{\pi}{2}k<2k$.
\end{proof}

\subsection{Well-posedness of the classical 3-point centered scheme}

The discretized problem is well-posed if and only if
$\lambda^h(\xi):=k^2\!\!-\!\!\frac{4}{h^2}\sin^2\frac{\xi h}{2}\ne 0$ for
$\xi\in\{1,..,N\!\!-\!\!1\}\pi$. Given that the continuous problem \eqref{1dhelm} is well-posed, we
can expect for \emph{all} sufficiently small $h$ that the discretized problem is also well-posed. A
corresponding upper bound for $h$ is:

\begin{lemma}\label{3ptnonzero}
  Suppose $\sigma_k:=\min_{\xi\in \pi\mathbb{N}}|k-\xi|>0$. Let $n_k^+\in\mathbb{N}$ such that
  $(n_k^+\!\!-\!\!1)\pi<k<n_k^+\pi$. If
  \begin{equation}\label{hk}
    0<h=\frac{1}{N}<h_k:=2\sqrt{3\sigma_k}/\left(\frac{\sigma_k}{2}+n_k^+\pi\right)^{3/2}=O(\sqrt{\sigma_k}k^{-3/2}),
  \end{equation}
  \HZ{or $k(kh)^2/\sigma_k\ll 1$ (as found in \cite{CFS25} for general FEMs)}, then
  \begin{equation}\label{sigmakh}
    \sigma_k^h:=\min_{(\xi/\pi)=1,..,N-1}\left|k-\frac{2}{h}\sin\frac{\xi h}{2}\right|\ge
    \frac{\sigma_k}{2},
  \end{equation}
and $\lambda^h:=k^2\!\!-\!\!\frac{4}{h^2}\sin^2\frac{\xi h}{2}\ge \max\{2k-\frac{\sigma_k}{2}, k+\frac{2}{\pi}\}\frac{\sigma_k}{2}$.
\end{lemma}

\begin{proof}
  If $\frac{2}{h}\sin \frac{\xi h}{2}\ge n_k^+\pi$ or
  $\frac{2}{h}\sin \frac{\xi h}{2}\le (n_k^+-1)\pi$, then
  $\left|k-\frac{2}{h}\sin\frac{\xi h}{2}\right|\ge \sigma_k$. If $\xi< k$, then
  \[
    \frac{2}{h}\sin \frac{\xi h}{2}\le \xi< k\text{ and } \left|k-\frac{2}{h}\sin\frac{\xi
        h}{2}\right|= k-\frac{2}{h}\sin\frac{\xi h}{2}\ge k - \xi = |k-\xi|\ge \sigma_k.
  \]
  We need only to consider the case $(n_k^+-1)\pi< \frac{2}{h}\sin \frac{\xi h}{2}< n_k^+\pi$
  and $\xi\ge n_k^+\pi$. It suffices to show that $n_k^+\pi-\frac{2}{h}\sin\frac{\xi
    h}{2}<\frac{\sigma_k}{2}$, which is equivalent to
  \[
    \frac{\xi h}{2}-\sin\frac{\xi h}{2} < \frac{\sigma_kh}{4} + \frac{\xi h-n_k^+\pi h}{2}.
  \]
  It is thus enough to show that $\theta-\sin\theta<\sigma_kh/4$ for $\theta=\xi h/2$. Since
  $\sin\theta > \theta - \frac{\theta^3}{3!}$ for $\theta\in (0,\pi/2)$, it suffices to show
  $\frac{\theta^3}{3!}<\sigma_kh/4$, which is equivalent to
  $\sin \theta < \sin \sqrt[3]{{3}\sigma_k h/2}=:\sin v$. Since $\sin\theta < n_k^+\pi h/2$ and
  $\sin v> v - \frac{v^3}{3!}$, it is enough to show that $n_k^+\pi h/2 < v - \frac{v^3}{3!}$, which
  is
  \[
    {n_k^+\pi h}/{2}< \sqrt[3]{3\sigma_k h/2} - (\sigma_kh/4).    
  \]
  Solving the inequality for $h$ gives the conclusion.  
\end{proof}

\begin{remark}
  If we define $h_k^*>0$ as the largest reciprocal of an integer such that
\begin{equation}\label{hkstar}
  \min_{(\xi/\pi)=1,..,N-1}\left|k-\frac{2}{h}\sin\frac{\xi h}{2}\right|\ge \frac{\sigma_k}{2}
\end{equation}
for \emph{all} $h=1/N\in(0,h_k^*)$, then $h_k^*$ is likely of the same order in $k$ as the upper
bound $h_k$ in \eqref{hk} is. This is illustrated on the left of Fig.~\ref{fig1}. But
Lemma~\ref{3ptnonzero} does not prohibit using larger mesh size. The existence of a mesh size
$h\sim k^{-1}$ that verifies $\sigma_k^h\ge \sigma_k/2$ is shown on the right of Fig.~\ref{fig1}.
\end{remark}

\begin{figure}
  \centering
  \includegraphics[scale=.42,trim=35 190 50 190,clip]{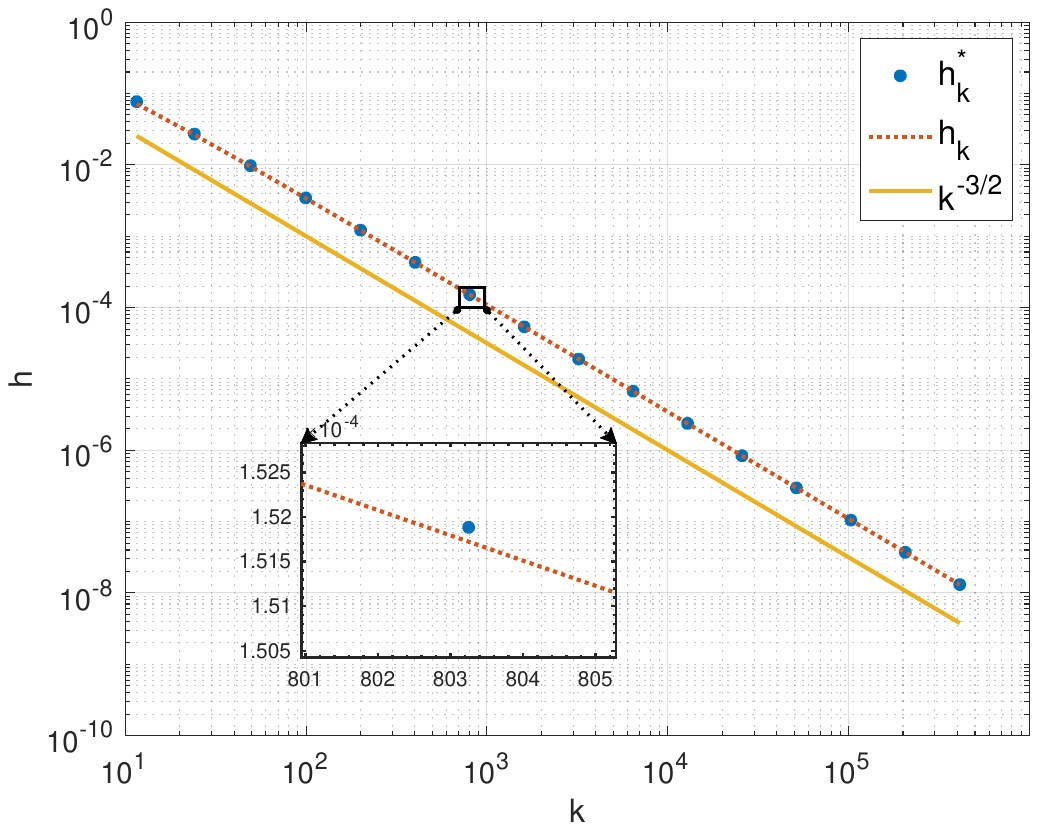}
  \includegraphics[scale=.42,trim=35 190 50 190,clip]{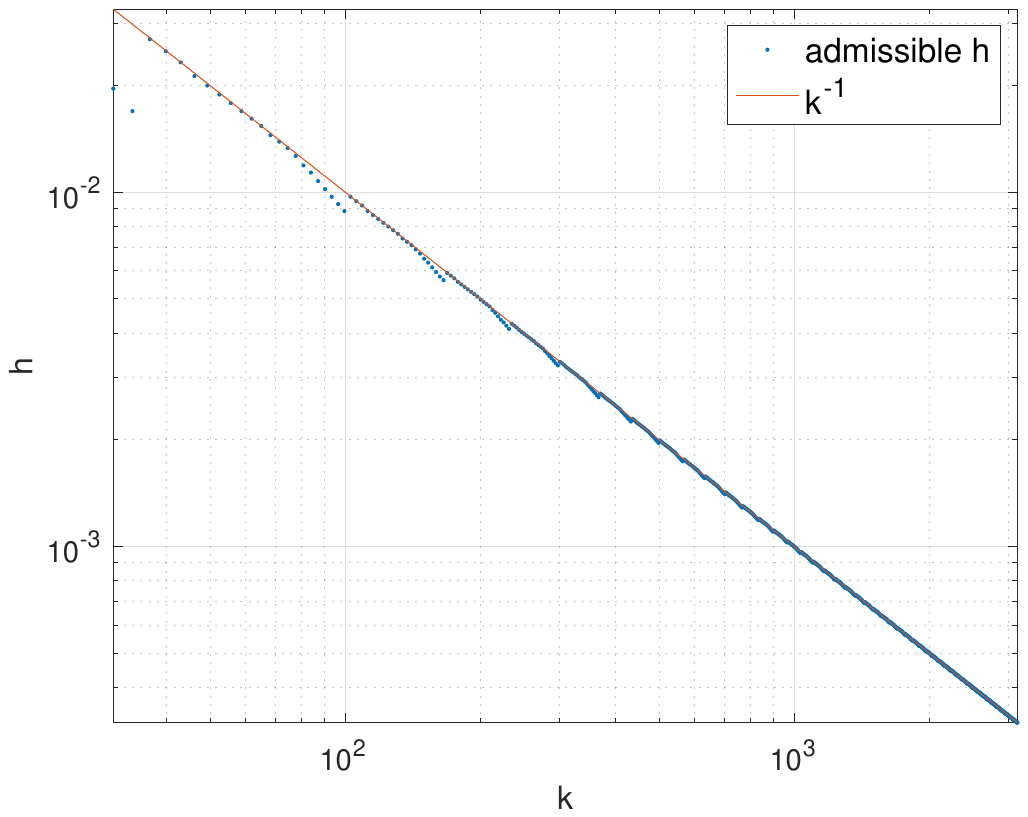}
  \caption{Admissible mesh size $h$ for well-posedness of \eqref{1d3pt}. Left: Upper bounds $h_k$
    \eqref{hk} and $h_k^*$ \eqref{hkstar} so that all smaller $h$ verifies \eqref{sigmakh} with
    $\sigma_k=1$.  Right: Existence of $h=O(k^{-1})$ verifying \eqref{sigmakh} with
    $\sigma_k=1$.}\label{fig1}
\end{figure}

\subsection{Aliasing error of the classical 3-point centered scheme}

From now on, we assume that $\lambda^h\ne 0$. Then the derivative of the partial error $e_1^h$ in
\eqref{errpart} is
\begin{equation}\label{3ptE1E2}
  \begin{aligned}
    (e_1^h)'(x)&=\sum_{(\xi/\pi)=1}^{N-1}\left(\hat{u}(\xi)-\widehat{u^h}(\xi)\right)\xi\cos(\xi x)
    =(E_1^h)'(x) - (E_2^h)'(x)\\
    &:=\sum_{(\xi/\pi)=1}^{N-1}\left(\frac{\xi}{\lambda}-\frac{\xi}{\lambda^h}\right)
    \hat{f}(\xi)\cos(\xi x)- \sum_{(\xi/\pi)=1}^{N-1}\frac{\xi}{\lambda^h}
    \sum_{s=\pm1}s\sum_{m=1}^{\infty}\hat{f}(s\xi+2mN\pi)\cos(\xi x).
  \end{aligned}
\end{equation}
To estimate $E_2^h$, we first note that
\begin{equation}\label{sumf}
  \left|\sum_{m=1}^{\infty}\hat{f}(\pm\xi+2mN\pi)\right|\le
  \sqrt{\sum_{m=1}^\infty\frac{1}{(\pm\xi+2mN\pi)^{2p}}}\cdot
  \sqrt{\sum_{m=1}^\infty\left|(\pm\xi+2mN\pi)^{p}\hat{f}(\pm\xi+2mN\pi)\right|^2},
\end{equation}
and for $0\le \xi\le (N-1)\pi$ that
\begin{equation}\label{sumxi}
  \sum_{m=1}^\infty\frac{1}{(\pm\xi+2mN\pi)^{2p}}\le \frac{1}{(\pm\xi+2N\pi)^{2p}} +
  \frac{1}{2N\pi}\int_{\pm\xi+2N\pi}^{\infty}\frac{1}{t^{2p}}~\mathrm{d}t\le
  \frac{2p}{2p-1}\frac{1}{((\frac{3}{2}\pm\frac{1}{2})N\pi)^{2p}}.
\end{equation}
We can then estimate
\[
  \begin{aligned}
    &\;|E_2^h|_1^2=\sum_{(\xi/\pi)=1}^{N-1}\frac{\xi^2}{\left|\lambda^h\right|^2}%
    \left|\sum_{s=\pm1}\sum_{m=1}^{\infty}\hat{f}(s\xi+2mN\pi)\right|^2\\
    \le&\;\max_{\xi\in\{1,..N-1\}\pi}\frac{\xi^2}{|\lambda^h|^2}\sum_{(\xi/\pi)=1}^{N-1}%
    \left|\sum_{s=\pm1}\sum_{m=1}^{\infty}\hat{f}(s\xi+2mN\pi)\right|^2\\
    \le&\;2\max_{\xi\in\{1,..N-1\}\pi}\frac{\xi^2}{|\lambda^h|^2}\sum_{(\xi/\pi)=1}^{N-1}%
    \left[\left|\sum_{m=1}^{\infty}\hat{f}(\xi+2mN\pi)\right|^2 + %
      \left|\sum_{m=1}^{\infty}\hat{f}(-\xi+2mN\pi)\right|^2\right].
  \end{aligned}
\]
Therefore, for $p\ge1$, using \eqref{sumxi}, we have
\begin{equation}\label{E2max}
  |E_2^h|_1\le 2\max_{\xi\in\{1,..N-1\}\pi}\frac{\xi}{|\lambda^h|}\frac{h^p}{\pi^{p}} |f_{\mathrm{high}}|_p.
\end{equation}

\begin{lemma}\label{lemaliasH1}
  Suppose $k>\pi$, $0<\frac{kh}{2}\le C_\mu<1$, and there is a sequence of $h=\frac{1}{N}$ going to
  zero such that
  \[
    \tilde{\sigma}_k:=\min_{\xi,h}\left|k-\frac{2}{h}\sin\frac{\xi h}{2}\right|>0\; \text{over
      $\xi\in\{1,\ldots,N-1\}\pi$ and the sequence of $h$ going to zero}.
  \]
  Then, for $\lambda^h=k^2-\frac{4}{h^2}\sin^2\frac{\xi h}{2}$, the $H^1$-semi-norm of the aliasing
  error $E_2^h$ in \eqref{3ptE1E2} satisfies
  \begin{equation}\label{E2}
    |E_2^h|_1\le \frac{\HZ{(kh)^p}}{\tilde{\sigma}_k\pi^{p-1}}\HZ{k^{-p}}|f_{\mathrm{high}}|_p.
  \end{equation}
\end{lemma}

\begin{proof}
  Based on \eqref{E2max}, it suffices to find an upper bound of
  \[
    \frac{\xi}{|\lambda^h|}=\frac{\xi}{|k^2-\frac{4}{h^2}\sin^2\frac{\xi h}{2}|}%
    \quad\text{ over }\xi\in\{1,\ldots,N-1\}\pi.
  \]
  Denote $\theta_k:=\arcsin\frac{kh}{2}$ and $\theta:=\frac{\xi h}{2}$.  Then
  \[
    \frac{\xi}{|\lambda^h|}=\frac{h}{2}\frac{\theta}{|\sin^2\theta_k-\sin^2\theta|}%
    =:\frac{h}{2}\phi(\theta)\quad\text{for }\theta\in\{1,\ldots,N-1\}\frac{\pi h}{2}=:\Theta.
  \]
  When $\theta<\theta_k$, the function $\phi(\theta)$ increases with $\theta$. Let
  $\theta_-:=\max\{\theta\in\Theta:~\theta<\theta_k\}$. So
  \[
    \phi(\theta)\le \frac{\theta_-}{\sin^2\theta_k-\sin^2\theta_-}\le %
    \frac{2}{h\tilde{\sigma}_k}\cdot\frac{\theta_-}{\sin\theta_k+\sin\theta_-} %
    \le \frac{2}{h\tilde{\sigma}_k}\cdot\frac{\theta_-}{\frac{2}{\pi}\theta_k+\frac{2}{\pi}\theta_-}
    < \frac{\pi}{2h\tilde{\sigma}_k}\quad\text{when }\theta\in \Theta\cap (0,\theta_k).
  \]
  When $\theta>\theta_k$, the function $\phi(\theta)$ has the derivative
  \[
    \phi'(\theta)=\frac{\sin^2\theta-\sin^2\theta_k-\theta\sin(2\theta)}%
    {(\sin^2\theta-\sin^2\theta_k)^2}.
  \]
  So any stationary point $\theta_c$ of $\phi$ satisfies
  $\sin^2\theta_c-\sin^2\theta_k-\theta_c\sin(2\theta_c)=0$ and
  \[
    \phi(\theta_c)=\frac{\theta_c}{\sin^2\theta_c-\sin^2\theta_k}= \frac{1}{\sin(2\theta_c)}.%
  \]
  Le $\theta_+:=\min\{\theta\in\Theta:~\theta>\theta_k\}$ and $\theta_{\max}:=\max\Theta$. If
  $\theta_+<\frac{\pi}{4}$, then the numerator of $\phi'(\theta_+)$ satisfies
  \[
    \sin^2\theta_+-\sin^2\theta_k-\theta_+\sin(2\theta_+)= %
    \sin(2\tilde{\theta})(\theta_+-\theta_k)-\theta_+\sin(2\theta_+)<0,
  \]
  because by the mean value theorem $\tilde{\theta}<\theta_+$. Note that the derivative of the
  numerator of $\phi'(\theta)$ is $-2\theta\cos(2\theta)$. So if $\theta_+<\frac{\pi}{4}$, then
  $\phi'<0$ on $[\theta_+, \frac{\pi}{4}]$. Therefore, no matter $\theta_+<\frac{\pi}{4}$ or not,
  any stationary point $\theta_c$ of $\phi$ in $[\theta_+,\theta_{\max}]$ must satisfy
  $\theta_c\ge \frac{\pi}{4}$. It follows that
  \[
    \phi(\theta_c)\le \frac{1}{\sin(2\theta_{\max})} = \frac{1}{\sin((N-1)\pi h)}=%
    \frac{1}{\sin(\pi h)}\le \frac{1}{2h}.
  \]
  Since $k>\pi$, the numerator of $\phi'(\theta_{\max})$ satisfies
  \[
    \sin^2\frac{\pi h}{2}-\sin^2\theta_k-(1-h)\frac{\pi}{2}\sin(\pi h) %
    < \frac{\pi^2h^2}{4} - \frac{k^2h^2}{4} - (1-h)\pi h<0.
  \]
  Note that
  \[
    \phi(\theta_+)=\frac{\theta_+}{\sin^2\theta_k-\sin^2\theta_+}\le %
    \frac{2}{h\tilde{\sigma}_k}\cdot\frac{\theta_+}{\sin\theta_k+\sin\theta_+}%
    <\frac{2}{h\tilde{\sigma}_k}\cdot\frac{\theta_+}{\sin\theta_+}<\frac{\pi}{h\tilde{\sigma}_k}.
  \]
  So we obtain
  \[
    \max_{\theta\in\Theta}\phi(\theta)< \max\{\frac{\pi}{h\tilde{\sigma}_k}, \frac{1}{2h}\}%
    =\frac{\pi}{h\tilde{\sigma}_k},\quad\text{and } %
    \max_{\xi\in\{1,\ldots,N-1\}}\frac{\xi}{|\lambda^h|}<\frac{\pi}{2\tilde{\sigma}_k}. 
  \]
  Substituting this into \eqref{E2max} gives the conclusion.  
\end{proof}

\begin{remark}
  By Lemma~\ref{lemaliasH1} the second part $(E_2^h)'$ of $(e_1^h)'$ in \eqref{3ptE1E2} is
  controlled, similarly to $(e_2^h)'$ in \eqref{e2p}, by source $f_{\mathrm{high}}$ \MG{not resolved
    by the grid}, when $k>\pi$, $\frac{kh}{2}\le C_{\mu}<1$ and $\tilde{\sigma}_k>0$. \HZ{But
    $|E_2^h|_1$ behaves as $O((kh)^p/\tilde{\sigma}_k)$, much larger than $|e_2^h|_1=O(h(kh)^p)$, if
    $|f_{\rm{high}}|_p=O(k^p)$.} 
\end{remark}

The $L^2$-norm of the aliasing error can also be derived. Recall that
\begin{equation}\label{3ptE10E20}
  \begin{aligned}
    e_1^h(x)&=\sum_{(\xi/\pi)=1}^{N-1}\left(\hat{u}(\xi)-\widehat{u^h}(\xi)\right)\sin(\xi x)
    =E_{1}^h(x) - E_{2}^h(x)\\
    &=\sum_{(\xi/\pi)=1}^{N-1}\left(\frac{1}{\lambda}-\frac{1}{\lambda^h}\right)
    \hat{f}(\xi)\sin(\xi x)- \sum_{(\xi/\pi)=1}^{N-1}\frac{1}{\lambda^h}
    \sum_{s=\pm1}s\sum_{m=1}^{\infty}\hat{f}(s\xi+2mN\pi)\sin(\xi x).
  \end{aligned}
\end{equation}

\begin{lemma}\label{lemaliasL2}
  Suppose $k>\pi$, $0<\frac{kh}{2}\le C_\mu<1$, and there is a sequence of $h=\frac{1}{N}$ going to
  zero such that
  \[
    \tilde{\sigma}_k:=\min_{\xi,h}\left|k-\frac{2}{h}\sin\frac{\xi h}{2}\right|>0\; \text{over
      $\xi\in\{1,\ldots,N-1\}\pi$ and the sequence of $h$ going to zero}.
  \]
  Then, for $\lambda^h=k^2-\frac{4}{h^2}\sin^2\frac{\xi h}{2}$, the $L^2$-norm of the aliasing error
  $E_{2}^h$ in \eqref{3ptE10E20} satisfies
  \[
    \|E_{2}^h\|\le \frac{\HZ{2h(kh)^{p-1}}}{\tilde{\sigma}_k\pi^{p}}\HZ{k^{-p}}|f_{\mathrm{high}}|_p.
  \]
\end{lemma}

\begin{proof}
  Based on \eqref{sumf} and \eqref{sumxi}, we have
  \[
    \begin{aligned}
      &\;\|E_{2}^h\|^2=\sum_{(\xi/\pi)=1}^{N-1}\frac{1}{\left|\lambda^h\right|^2}%
      \left|\sum_{s=\pm1}\sum_{m=1}^{\infty}\hat{f}(s\xi+2mN\pi)\right|^2\\
      \le&\;\max_{\xi\in\{1,..N-1\}\pi}\frac{1}{|\lambda^h|^2}\sum_{(\xi/\pi)=1}^{N-1}%
      \left|\sum_{s=\pm1}\sum_{m=1}^{\infty}\hat{f}(s\xi+2mN\pi)\right|^2\\
      \le&\;2\max_{\xi\in\{1,..N-1\}\pi}\frac{1}{|\lambda^h|^2}\sum_{(\xi/\pi)=1}^{N-1}%
      \left[\left|\sum_{m=1}^{\infty}\hat{f}(\xi+2mN\pi)\right|^2 + %
        \left|\sum_{m=1}^{\infty}\hat{f}(-\xi+2mN\pi)\right|^2\right].
    \end{aligned}
  \]
  Therefore, for $p\ge 1$, using \eqref{sumxi}, we get
  \[
    \|E_{2}^h\|\le 2\max_{\xi\in\{1,\ldots,N-1\}\pi} %
    \frac{1}{|\lambda^h|}\frac{h^p}{\pi^p}|f_{\mathrm{high}}|_p.
  \]
  Note that
  \[
    \frac{1}{|\lambda^h|}=\frac{1}{|k^2-\frac{4}{h^2}\sin^2\frac{\xi h}{2}|}= %
    \frac{1}{|k-\frac{2}{h}\sin\frac{\xi h}{2}|} %
    \frac{1}{|k+\frac{2}{h}\sin\frac{\xi h}{2}|} \le%
    \frac{1}{\tilde{\sigma}_kk}.
  \]
  Combining the above two inequalities gives the conclusion.
\end{proof}

\begin{remark}
  By Lemma~\ref{lemaliasL2}, the $L^2$-norm of the aliasing error is
  $O\left(\frac{\HZ{(kh)^{p}}}{\tilde{\sigma}_kk}\right)$ if \HZ{$|f_{\rm{high}}|_p=O(k^p)$}.  This
  can be compared to the $H^1$-semi-norm of the aliasing error which is
  $O\left(\frac{\HZ{(kh)^p}}{\tilde{\sigma}_k}\right)$; see Lemma~\ref{lemaliasH1}.
\end{remark}

\subsection{Absolute error estimates of the classical 3-point centered scheme}

We now turn to estimate the first part of $(e_1^h)'$ given in \eqref{errpart} and \eqref{3ptE1E2},
\begin{equation}\label{E1hdef}
  (E_1^h)'(x)=\sum_{(\xi/\pi)=1}^{N-1}\left(\frac{\xi}{\lambda}-\frac{\xi}{\lambda^h}\right)
  \hat{f}(\xi)\cos(\xi x),
\end{equation}
the only part of $(e^h)'=(E_1^h)'-(E_2^h)'+(e_2^h)'$ controlled by the grid well-resolved source
$f_{\mathrm{low}}$. Since
\begin{equation}\label{E1h}
  |E_1^h|_1^2 = \frac{1}{2}
  \sum_{(\xi/\pi)=1}^{N-1}\left|\frac{\xi}{\lambda}-\frac{\xi}{\lambda^h}\right|^2|\hat{f}(\xi)|^2
  \le \|f_{\mathrm{low}}\|^2 \max_{\xi\in\{1,..,N-1\}\pi}\left|\frac{\xi}{\lambda}-\frac{\xi}{\lambda^h}\right|^2,
\end{equation}
we need only to find the maximum of
\begin{equation}\label{3ptpsi}
  \psi(\xi):=\left| \frac{\xi}{\lambda} - \frac{\xi}{\lambda^h} \right|=
  \left| \frac{\xi}{k^2-\xi^2} - \frac{\xi}{k^2-\frac{4}{h^2}\sin^2\frac{\xi h}{2}} \right|
  =\frac{h}{2}\left| \frac{\theta}{\mu^2-\theta^2} -
    \frac{\theta}{\mu^2-\sin^2\theta}\right|=:\frac{h}{2}\phi(\theta)
\end{equation}
for $\mu:=kh/2$ and $\theta:=\xi h/2$.

\begin{remark}\label{remsharp}
  The inequality in \eqref{E1h} is sharp. Let
  $\xi_*:=\arg\max_{\xi\in\{1,..N-1\}\pi}\psi(\xi)$. Then the inequality holds with equality if
  $f_{\mathrm{low}}=\hat{f}(\xi_*)\sin (\xi_* x)$. In the convergence test, one usually takes a
  fixed $f$ as $h\to 0$. So the sharpness may be perturbed if $\xi_*$ depends on $h$. This is not an
  issue, if $\xi_*$ is independent of $h$ or if $\xi_*$ varies slightly around a fixed number, which
  is the case for $\xi_*$ being the continuous frequencies $\xi$ closest to $k$, or their discrete
  counterparts $\frac{2}{h}\sin\frac{\xi h}{2}$ closest to $k$, as we shall see in
  Lemma~\ref{lem3ptmax}. However, for the case $\xi_*=(N-1)\pi$ which we shall also see in
  Lemma~\ref{lem3ptmax}, the sharpness will require a changing $f$ at higher and higher frequency as
  $h\to 0$. In that case, we will derive an alternate estimate to \eqref{E1h} using the smoothness
  of a fixed $f$.
\end{remark}

\begin{lemma}\label{lem3ptmax}
  Suppose $k>2\pi$, $\mu:=\frac{kh}{2}\le C_\mu< 1$,
  $\sigma_k:=\min_{\xi\in \pi\mathbb{N}}|k-\xi|>0$, and there is a sequence of $h=\frac{1}{N}$
  ($N\ge 4$) going to zero such that
  \[
    \tilde{\sigma}_k:=\min_{\xi,h}\left|k-\frac{2}{h}\sin\frac{\xi h}{2}\right|>0\; \text{over
      $\xi\in\{1,\ldots,N-1\}\pi$ and the sequence of $h$ going to zero}.
  \]
  Let $k^h:=\frac{2}{h}\arcsin\frac{kh}{2}$, $\xi_{\max}=(N-1)\pi$, and $k_{\pm}\in \pi\mathbb{N}$
  (resp.  $k^h_{\pm}\in \pi\mathbb{N}$) be the closest points such that $k_-<k<k_+$
  (resp. $k^h_-<k^h<k^h_+$). Then, $k_+<\xi_{\max}$, and
  \[
    \max_{\xi\in\{1,..,N-1\}\pi}\psi(\xi)=
    \max\{\psi(k_-),\psi(k_+),\psi(k_-^h),\psi(k_+^h),\psi(\xi_{\max})\},
  \]
  where $\psi$ is given in \eqref{3ptpsi}, $\psi(k_{\pm}^h)$ can be removed if
  $k_{\pm}^h\ge\xi_{\max}$, and $k_+$ (resp. $k_-^h$) equals $k_+^h$ (resp. $k_-$) when
  $\pi \mathbb{N}\cap(k,k^h)=\emptyset$. Moreover, we have {\allowdisplaybreaks
  \begin{align*}
      \frac{k^3h^2}{(240\sigma_-+k^3h^2)6\sigma_-}<
      &\,\psi(k_-)<\frac{k^3h^2}{(24\sigma_-+k^3h^2)2\sigma_-}\,,&\\
      \frac{1}{15}\frac{(k+\sigma_+)^3h^2}{\sigma_+(4\HZ{\sigma_{-}^h} + k^3h^2)}<
      &\,\psi(k_+)<\frac{2}{3}\frac{k^3h^2}{\sigma_-^h\sigma_+\sqrt{1-C_\mu^2}}\,,&\text{ if }\pi \mathbb{N}\cap(k,k^h)\ne\emptyset,\\
      \frac{(1-\frac{\pi^2}{80}C_{\mu}^2)(k+\sigma_+)^3h^2}{12\HZ{\sigma_-^h}(4\sigma_++k^3h^2)}<&\,\psi(k_-^h)<\frac{\pi^4}{384}\frac{k^3h^2}{\sigma_-^h\sigma_+\sqrt{1-C_{\mu}^2}}\,,&\text{ if }\pi \mathbb{N}\cap(k,k^h)\ne\emptyset,\\
      \frac{(1-\frac{(1+\pi)^2}{80}C_{\mu}^2)(k+\sigma_+^h)^3h^2}{\HZ{3(4\sigma_+^h+k^3h^2)^2}}<\,&\psi(k_+^h)<\frac{(1+\pi)^4}{16}\frac{k^3h^2}{\sigma_+^h(24\sigma_+^h+k^3h^2)\cos(\theta_\mu+\frac{\sigma_+^hh}{2})}\,,&\text{ if }C_{\mu}<\cos\frac{\sigma_+^hh}{2},\\
      \frac{h}{\pi}\left(\frac{9\pi^2}{64}\!-\!1\right)<&\,\psi(\xi_{\max})<\frac{4(\pi^2-4)\pi h}{\left(\cos^2\frac{\pi h}{2}\!-\!C_\mu^2\right)(9\pi^2-64)}\,,&\text{ if }C_{\mu}<\cos\frac{\pi h}{2},
  \end{align*}
} where $\sigma_{\pm}:=|k-k_{\pm}|\ge\sigma_k=\min\{\sigma_-,\sigma_+\}$,
$\sigma_{\pm}^h:=|k^h-k_{\pm}^h|\ge\tilde{\sigma}_k$ and $\theta_\mu:=\arcsin C_{\mu}$.
\end{lemma}


\begin{remark}
  Note that $\pi \mathbb{N}\cap(k,k^h)\ne \emptyset$ if and only if $k_+<k^h$. Using
  $\mu^3/6<\arcsin\mu-\mu<(\pi-2)\mu^3/2$ for $0<\mu<1$, we find a sufficient condition for
  $k_+\ge k^h$ is that $k^3h^2\le 8\sigma_+/(\pi-2)$ and a sufficient condition for $k_+<k^h$ is
  that $k^3h^2>24\sigma_+$.
\end{remark}

\begin{remark}\label{remorder}
  On the one hand, if $k^3h^2\HZ{\gtrsim \sigma_-}$, then the lower bound of $\psi(k_-)$ is greater
  than a positive constant times $1/\sigma_-$, independent of $k$ and $h$, no matter how small $kh$
  is. On the other hand, if $k^3h^2<C$, then based on Lemma~\ref{lem3ptmax} we can find the bounds
  \[
    \frac{k^3h^2}{(240\sigma_-+C)6\sigma_-}<\psi(k_-)< \frac{k^3h^2}{48\sigma_-^2}.
  \]
  \HZ{In particular, if $k^3h^2\lesssim\sigma_k$ and $\sigma_-=\sigma_k$, then
    $\psi(k_-)\asymp k^3h^2/\sigma_k^2$.}  \HZ{If $\sigma_+=\sigma_k$ and the assumptions of
    Lemma~\ref{kkh} hold, then $\sigma_+^h \asymp \sigma_k$ and
    $\psi(k_{+}^h)\asymp k^3h^2/\sigma_k^2$.}
\end{remark}

\begin{proof}[Proof of Lemma~\ref{lem3ptmax}]
  Let $\Theta:=\{1,..,N\!-\!1\}\frac{\pi h}{2}$, $\theta_k:=\arcsin\mu$,
  $\theta_{\max}:=\frac{h}{2}\xi_{\max}=\frac{\pi}{2}(1-h)$, $\mu_\pm:=\frac{h}{2}k_{\pm}$, and
  $\theta_{\pm}:=\frac{h}{2}k^h_{\pm}$. We first show that
  $\sigma_{\pm}^h\ge\tilde{\sigma}_k$. Indeed, by the mean value theorem, it holds
  \[
    |k^h-k_{\pm}^h|=\frac{2}{h}\left|\arcsin\frac{kh}{2}-\frac{k_{\pm}^hh}{2}\right|=%
    \frac{2}{h}|\theta_k-\theta_{\pm}|\ge %
    \frac{2}{h}\left|\sin\theta_k-\sin\theta_{\pm}\right|=%
    \left|k-\frac{2}{h}\sin\frac{k_{\pm}^hh}{2}\right|.
  \]

  Recall the definition of $\phi$ in \eqref{3ptpsi}. When $0<\theta<\mu$, the function
  \[
    \phi(\theta)=\frac{(\theta^2-\sin^2\theta)\theta}{(\mu^2-\theta^2)(\mu^2-\sin^2\theta)}
  \]
  is increasing {because the numerator is increasing and the denominator is decreasing. Indeed,
    $(\theta^2-\sin^2\theta)'=2\theta-\sin (2\theta)>0$.} So the $\max\phi(\theta)$ over $[0,\mu_-]$
  is attained at $\theta=\mu_-$.

  When $\mu<\theta<\theta_k$, we have
  \[
    \phi(\theta)=\frac{\theta}{\theta^2-\mu^2} + \frac{\theta}{\mu^2-\sin^2\theta}=
    \frac{1}{2}\left(\frac{1}{\theta-\mu} + \frac{1}{\theta+\mu}\right) +
    \frac{1}{2\mu}\left(\frac{\theta}{\mu-\sin\theta}+\frac{\theta}{\mu+\sin\theta}\right).
  \]
  Note that $1/(\theta\pm \mu)$ is convex for $\theta\in (\mu,\infty)$. We calculate
  \[
    \left(\frac{\theta}{\mu\pm \sin\theta}\right)'=\frac{1}{\mu\pm \sin\theta}
    -\frac{\pm\theta\cos\theta}{(\mu\pm \sin\theta)^2},\quad \left(\frac{\theta}{\mu\pm
        \sin\theta}\right)''= \frac{\mp2\cos\theta\pm \theta\sin\theta}{(\mu\pm \sin\theta)^2} +
    \frac{2\theta\cos^2\theta}{(\mu\pm \sin\theta)^3}.
  \]
  Since $1\ge \mu=\sin\theta_k>\sin\theta>0$ and $\theta>0$, we have
  \[
    \frac{2\cos\theta}{(\mu-\sin\theta)^2}> \frac{2\cos\theta}{(\mu+\sin\theta)^2},\quad
    \frac{2\theta\cos^2\theta}{(\mu- \sin\theta)^3}\ge \frac{\theta\sin\theta}{(\mu-\sin\theta)^2}.
  \]
  Hence, $\phi$ is convex for $\theta\in(\mu,\theta_k)$. If $\mu_+\le \theta_-$ (that is,
  $\Theta\cap(\mu,\theta_k)\ne \emptyset$), then $\phi(\theta)$ for $\theta\in [\mu_+,\theta_-]$
  attains its maximum at either $\mu_+$ or $\theta_-$.
  
  When $\theta_k<\theta<\frac{\pi}{2}$, we have
  \begin{equation}\label{eqphi}
    \phi(\theta)=\frac{\theta}{\sin^2\theta-\mu^2} - \frac{\theta}{\theta^2-\mu^2}
    {\text{ has no local maximum for }\theta_k<\theta<\frac{\pi}{2}\text{ and }0<\mu<1}.
  \end{equation}
  {The proof can be found in Appendix~\ref{phiproof}.}
  Hence, $\phi(\theta)$ on $\theta\in [\theta_+,\theta_{\max}]$ attains its maximum at either
  $\theta_+$ or $\theta_{\max}$ when $\theta_+<\theta_{\max}$ (otherwise, there is no such interval
  to consider). Note that if $\mu\le C_{\mu}<\frac{1}{2}$ and $k>\pi$ then it holds that
  $\theta_+<\theta_{\max}$. Indeed, we have $h<\frac{1}{\pi}<\frac{2}{\pi+\sigma_+^h}$ so
  $\sin(\theta_{\max}-\frac{\sigma_+^hh}{2})=\cos(\frac{\pi+\sigma_+^h}{2}h)>1-\frac{(\pi+\sigma_+^h)^2}{8}h^2>\frac{1}{2}>\sin\theta_k$.
  
  Next, we shall estimate the candidates for $\max\phi(\theta)$.  Using $\sin x>x-\frac{1}{3!}x^3$
  for $x\in (0,\frac{\pi}{2})$ and $\mu-\mu_-= \frac{\sigma_-h}{2}>0$, we have
  \[
    \begin{aligned}
      \phi(\mu_-)&=\frac{\mu_--\sin\mu_-}{\mu-\sin\mu_-}\cdot\frac{\mu_-+\sin\mu_-}{\mu+\sin\mu_-}
      \cdot \frac{\mu_-}{\mu+\mu_-}\cdot\frac{1}{\mu-\mu_-}\\
      &< \frac{\mu_--(\mu_--\frac{1}{3!}\mu_-^3)}{\mu-(\mu_--\frac{1}{3!}\mu_-^3)}\cdot
      1\cdot\frac{1}{2}\cdot \frac{2}{\sigma_-h}< \frac{\mu^3}{3\sigma_-h +
        \mu^3}\cdot\frac{1}{\sigma_-h}=\frac{k^3h}{(24\sigma_-+k^3h^2)\sigma_-}.
    \end{aligned}
  \]
  Note that $\sigma_-<\pi<\frac{k}{2}$ implies $\mu<2\mu-\sigma_-h=2\mu_-$.  Using also
  $x-\frac{1}{3!}x^3<\sin x<x-\frac{1}{3!}x^3+\frac{1}{5!}x^5$ for $x\in (0,\frac{\pi}{2})$ and
  $\mu_-<\mu< 1$, we get
  \[
    \begin{aligned}
      \phi(\mu_-)&>\frac{\mu_--(\mu_--\frac{1}{3!}\mu_-^3+\frac{1}{5!}\mu_-^5)}
      {\mu-(\mu_--\frac{1}{3!}\mu_-^3+\frac{1}{5!}\mu_-^5)}\cdot
      \frac{\mu_-+\mu_--\frac{1}{3!}\mu_-^3}{\mu+\mu_--\frac{1}{3!}\mu_-^3}\cdot\frac{1}{3}\cdot
      \frac{2}{\sigma_-h}\\
      &=\frac{\mu_-^3(1-\frac{1}{20}\mu_-^2)} {3\sigma_-h+\mu_-^3(1-\frac{1}{20}\mu_-^2)}\cdot
      \frac{\mu_-+\mu_--\frac{1}{3!}\mu_-^3}{\mu+\mu_--\frac{1}{3!}\mu_-^3}\cdot\frac{1}{3}\cdot
      \frac{2}{\sigma_-h}\\
      &>\frac{\frac{\mu^3}{2^3}(1-\frac{1}{20})} {3\sigma_-h+\frac{\mu^3}{2^3}(1-\frac{1}{20})}\cdot
      \frac{\mu_-+\mu_--\frac{1}{3!}\mu_-}{2\mu_-+\mu_--\frac{1}{3!}\mu_-}\cdot\frac{1}{3}\cdot
      \frac{2}{\sigma_-h}>\frac{k^3h}{(240\sigma_-+k^3h^2)3\sigma_-}.
    \end{aligned}
  \]
  
  If $\mu_+<\theta_k$ i.e. $\Theta\cap(\mu,\theta_k)\ne\emptyset$, then $\phi(\mu_+)$ and
  $\phi(\theta_-)$ need to be counted as candidates for the maximum. Since $\theta_k=\arcsin\mu$, we
  know $\sin\mu_+<\mu$, and
  \[
    \phi(\mu_+)=\frac{\mu_+-\sin\mu_+}{\mu-\sin\mu_+}\cdot\frac{\mu_++\sin\mu_+}{\mu+\sin\mu_+}
    \cdot \frac{\mu_+}{\mu+\mu_+}\cdot\frac{1}{\mu_+-\mu}.
  \]
  Using $1>C_{\mu}\ge\mu=\sin\theta_k$, $\mu_+\le\theta_-<\theta_k$ and
  $\sigma_-^hh/2=\theta_k-\theta_-<\pi h/2$ we find
  \[
    \begin{aligned}
      \mu-\sin\mu_+&=\sin\theta_k-\sin\mu_+> (\theta_k-\mu_+)\cos\theta_k\ge
      \left(\frac{\sigma_-^hh}{2}+\theta_--\mu_+\right)\sqrt{1-C_\mu^2},\\
      \mu-\sin\mu_+&=\sin\theta_k-\sin\mu_+<\theta_k-\mu_+ \HZ{=(\theta_k-\theta_{-}) +
        (\theta_--\mu_+)}.
    \end{aligned}
  \]
  Note that $\theta_k>\theta_-\ge \mu_+>\mu$ and $\arcsin x - x < x^3$ for $0<x<1$. So
  \[
    0\le \theta_--\mu_+<\theta_k-\mu=\arcsin \mu - \mu < \mu^3.
  \]
  Using $x - \frac{x^3}{3!}+ \frac{x^5}{5!}\!>\!\sin x\!>\! x\!-\!\frac{x^3}{3!}$ for
  $x\!\in\!(0,\pi/2)$ and $\mu\!<\!\mu_+\!<\!2\mu<2$ (for $\mu\!>\!\pi h/2\!>\!\mu_+\!-\!\mu$), we
  have
  \[
    \frac{2\mu_+^3}{15}<\frac{\mu_+^3}{6}\left(1 - \frac{\mu_+^2}{20}\right)< \mu_+-\sin{\mu_+}<
    \frac{\mu_+^3}{6} < \frac{4}{3}\mu^3.
  \]
  Combining all this, we obtain
  \[
    \frac{2}{15}\frac{(k+\sigma_+)^3h}{\sigma_+(4\HZ{\sigma_-^h} + k^3h^2)}=
    \frac{2}{15}\frac{\mu_+^3}{\frac{\HZ{\sigma_{-}^h} h}{2}+\mu^3}\cdot\frac{1}{\sigma_+h}<\phi(\mu_+)<
    \frac{4}{3}\frac{\mu^3}{\frac{\sigma_-^hh}{2}\sqrt{1-C_\mu^2}}\cdot
    \frac{4}{\sigma_+h}=\frac{4}{3}\frac{k^3h}{\sigma_-^h\sigma_+\sqrt{1-C_\mu^2}}.
  \]
  
  If $\theta_-\!>\!\mu$ i.e. $\Theta\!\cap\!(\mu,\theta_k)\!\ne\!\emptyset$, $\phi(\theta_-)$ is a
  candidate for $\max\phi$. For $\mu\!=\!\sin\theta_k$ and $\theta_k\!>\!\theta_-$, we have
  \[
    \phi(\theta_-)=\frac{\theta_--\sin\theta_-}{\mu-\sin\theta_-}\cdot
    \frac{\theta_-+\sin\theta_-}{\mu+\sin\theta_-}
    \cdot \frac{\theta_-}{\mu+\theta_-}\cdot\frac{1}{\theta_--\mu}.
  \]
  Using $1>C_{\mu}\ge\mu=\sin\theta_k$ and $\sigma_-^hh/2=\theta_k-\theta_-$, we find
  \[
    \HZ{\frac{\sigma_{-}^hh}{2}=\theta_k-\theta_->}\,\mu-\sin\theta_{-}=\sin\theta_k-\sin\theta_->
    (\theta_k-\theta_-)\cos\theta_k\ge \frac{\sigma_-^hh}{2}\sqrt{1-C_{\mu}^2}.
  \]
  Using $x - \frac{x^3}{3!}+ \frac{x^5}{5!}\!>\!\sin x\!>\! x\!-\!\frac{x^3}{3!}$ for
  $x\!\in\!(0,\frac{\pi}{2})$ and $\mu\!+\!\frac{h}{2}\sigma_+\!\le\!\theta_-\!<\!\frac{\pi}{2}\mu$
  (for $\theta_-\!<\!\theta_k\!<\!\frac{\pi}{2}\sin\theta_k$), we have
  \[
    \frac{(1-\frac{\pi^2}{80}C_{\mu}^2)(\mu+\frac{h}{2}\sigma_+)^3}{6}<\frac{\theta_-^3}{6}\left(1 -
      \frac{\theta_-^2}{20}\right)< \theta_--\sin{\theta_-}< \frac{\theta_-^3}{6} <
    \frac{\pi^3}{48}\mu^3.
  \]
  Combining all this and using also
  $\theta_--\mu=\theta_--\mu_++\mu_+-\mu<\mu^3+\frac{\sigma_+h}{2}$, we obtain
  \[
    \frac{(1-\frac{\pi^2}{80}C_{\mu}^2)(k+\sigma_+)^3h}{6\HZ{\sigma_{-}^h}(4\sigma_++k^3h^2)}<\phi(\theta_-)<\frac{\pi^3}{48}\frac{\mu^3}{\frac{\sigma_-^hh}{2}\sqrt{1-C_{\mu}^2}}
    \cdot \frac{\pi}{2}\cdot1\cdot\frac{2}{\sigma_+h}
    =\frac{\pi^4}{192}\frac{k^3h}{\sigma_-^h\sigma_+\sqrt{1-C_{\mu}^2}}.
  \]

  If $\theta_+<\theta_{\max}$, then $\theta_{+}$ needs to be considered. For $\mu\!=\!\sin\theta_k$
  and $\theta_k\!<\!\theta_+$, we have
  \[
    \phi(\theta_+)=\frac{\theta_+-\sin\theta_+}{\sin\theta_+-\mu}\cdot
    \frac{\theta_++\sin\theta_+}{\mu+\sin\theta_+}
    \cdot \frac{\theta_+}{\mu+\theta_+}\cdot\frac{1}{\theta_+-\mu}.
  \]
  Using $1>C_{\mu}\ge\mu=\sin\theta_k$,
  $\sigma_+^hh/2=\theta_+-\theta_k<\pi h/2<\frac{kh}{4}=\frac{\mu}{2}$ and
  $C_{\mu}=\sin\theta_\mu<\cos(\sigma_+^hh/2)$ (so that
  $\theta_\mu+\frac{\sigma_+^hh}{2}<\frac{\pi}{2}$), we find
  \[
    \HZ{\theta_{+}-\mu}>\sin\theta_+-\mu=\sin\theta_+-\sin\theta_k>
    (\theta_+-\theta_k)\cos\theta_+>
    \frac{\sigma_+^hh}{2}\cos\left(\theta_\mu+\frac{\sigma_+^hh}{2}\right).
  \]
  Using $x\!-\!\frac{x^3}{3!}\!+\!\frac{x^5}{5!}\!>\!\sin x\!>\! x\!-\!\frac{x^3}{3!}$ for
  $x\!\in\!(0,\frac{\pi}{2})$ and
  $\mu\!+\!\frac{\sigma_+^hh}{2}\!\le\!\theta_+\!=\!\frac{\sigma_+^hh}{2}\!+\!\theta_k\!<\!(\frac{1}{2}+\frac{\pi}{2})\mu$,
  we have
  \[
    \frac{(1-\frac{(1+\pi)^2}{80}C_{\mu}^2)(\mu+\frac{\sigma_+^hh}{2})^3}{6}<
    \frac{\theta_+^3}{6}\left(1 - \frac{\theta_+^2}{20}\right)< \theta_+-\sin{\theta_+}<
    \frac{\theta_+^3}{6} < \frac{(1+\pi)^3}{48}\mu^3.
  \]
  Note that
  $\frac{h}{2}\sigma_+^h+\frac{\mu^3}{6}<\theta_+\!-\!\mu=(\theta_+\!-\!\theta_k)\!+\!(\theta_k\!-\!\mu)\!<\frac{h}{2}\sigma_+^h+\mu^3\!$.
  Combining all this, we get
  \[
    \begin{aligned}
      \phi(\theta_+)\,&<  \frac{(1+\pi)^3}{48}\frac{\mu^3}{\frac{\sigma_+^hh}{2}\cos(\theta_\mu+\frac{\sigma_+^hh}{2})}(\frac{1}{2}+\frac{\pi}{2})\frac{6}{3h\sigma_+^h+\mu^3}\\
      \,&=\frac{(1+\pi)^4}{8}\frac{k^3h}{\sigma_+^h(24\sigma_+^h+k^3h^2)\cos(\theta_\mu+\frac{\sigma_+^hh}{2})},\\
      \phi(\theta_+)\,&>\frac{(1-\frac{(1+\pi)^2}{80}C_{\mu}^2)(\mu+\frac{\sigma_+^hh}{2})^3}{\HZ{12
        (\frac{h}{2}\sigma_+^h+\mu^3)^{2}}}=\frac{(1-\frac{(1+\pi)^2}{80}C_{\mu}^2)(k+\sigma_+^h)^3\HZ{2}h}{\HZ{3(4\sigma_+^h+k^3h^2)^{2}}}.
    \end{aligned}
  \]

  For $\mu\!=\!\sin\theta_k$ and $\theta_k\!<\!\theta_{\max}$, we have
  \[
    \phi(\theta_{\max})=\frac{\theta_{\max}^2 -\sin^2\theta_{\max} }{\sin^2\theta_{\max} -\mu^2}
    \cdot \frac{\theta_{\max} }{\theta_{\max}^2-\mu^2 }.
  \]
  Note that $\sin^2\theta_{\max}=\cos^2\frac{\pi h}{2}>(1-\frac{\pi^2h^2}{8})^2>1-\frac{\pi^2h^2}{4}$,
  $\theta_{\max}^2=\frac{\pi^2}{4}(1-2h+h^2)$. So
  \[
    \frac{9\pi^2}{64}-1<\theta_{\max}^2-\sin^2\theta_{\max}<\frac{\pi^2}{4}-1,
  \]
  where to get the first inequality we used the assumption $0<h\le\frac{1}{4}$. We have also
  \[
    \frac{2}{\pi}<\frac{1}{\theta_{\max}}<\frac{\theta_{\max} }{\theta_{\max}^2-\mu^2 }<\frac{\pi}{2}\left(\frac{9\pi^2}{64}-1\right)^{-1},  
  \]
  where for the second inequality the assumption $0<h\le\frac{1}{4}$ is used. Note that
  \[
    1>\sin^2\theta_{\max}-\mu^2>\cos^2\frac{\pi h}{2}-C_\mu^2.
  \]
  Combining all this and using the assumption $C_{\mu}<\cos\frac{\pi h}{2}$, we get
  \[
    \frac{2}{\pi}\left(\frac{9\pi^2}{64}-1\right)<\phi(\theta_{\max})<\left(\frac{\pi^2}{4}-1\right)\left(\cos^2\frac{\pi h}{2}-C_\mu^2\right)^{-1}\frac{\pi}{2}\left(\frac{9\pi^2}{64}-1\right)^{-1}.
  \]
  We conclude by multiplying $\frac{h}{2}$ with $\phi$ to get $\psi$ \eqref{3ptpsi}.
\end{proof}

From Lemma~\ref{lem3ptmax} and Remark~\ref{remorder}, we can see that, when $k^3h$ is greater than a
certain constant it holds that $\psi(\xi_{\max})<\max\psi$, and the error is of order $k^3h^2$ for a
fixed $f$ in the worst case; otherwise, $\max \psi=\psi(\xi_{\max})$ will be of order $h$. As
discussed in Remark~\ref{remsharp}, in the case of $\max\psi=\psi(\xi_{\max})$ i.e. $k^3h$ less than
a certain constant, we will derive an alternate estimate using the smoothness of $f$.

Before doing that, we first check whether the smoothness of $f$ improves the convergence order of
$\psi(k_-)$ (the discussion of $\psi(k_+)$, $\psi(k_{\pm}^h)$ would be similar). If
$f\in H^1_0(0,1)$, then the estimate \eqref{E1h} can be modified to
\begin{equation}\label{E1hf1}
  |E_1^h|_1^2 = \frac{1}{2}
  \sum_{(\xi/\pi)=1}^{N-1}\left|\frac{1}{\lambda}-\frac{1}{\lambda^h}\right|^2|\xi\hat{f}(\xi)|^2
  \le |f_{\mathrm{low}}|_1^2 \max_{\xi\in\{1,..,N-1\}\pi}\left|\frac{1}{\lambda}-\frac{1}{\lambda^h}\right|^2.
\end{equation}
So we need only to find the maximum of
\begin{equation}\label{psie}
  \psi_e(\xi):=\left| \frac{1}{\lambda} - \frac{1}{\lambda^h} \right|=
  \left| \frac{1}{k^2-\xi^2} - \frac{1}{k^2-\frac{4}{h^2}\sin^2\frac{\xi h}{2}} \right|
  =\frac{h^2}{4}\left| \frac{1}{\mu^2-\theta^2} -
    \frac{1}{\mu^2-\sin^2\theta}\right|=:\frac{h^2}{4}\phi_e(\theta).
\end{equation}
The subscript `e' indicates the quantity is mainly useful for the evanescent modes with $\xi>k^h_+$,
which we shall see as follows. For $\theta\in (0,\mu)$, we have
\[
  \phi_e(\theta)=\frac{1}{\mu^2-\theta^2}-\frac{1}{\mu^2-\sin^2\theta},\quad
  \phi_e'(\theta)=\frac{2\theta(\mu^2-\sin^2\theta)^2-(\mu^2-\theta^2)^2\sin(2\theta)}
  {(\mu^2-\theta^2)^2(\mu^2-\sin^2\theta)^2}>0.
\]
So $\max_{(0,\mu_-]}\phi_e=\phi_e(\mu_-)$. The estimate of $\phi_e(\mu_-)$ needs only a small
modification of the estimate of $\phi(\mu_-)$ in the proof of Lemma~\ref{lem3ptmax}. More precisely,
the term $1/(\mu+\mu_-)$ is used now instead of $\mu_-/(\mu+\mu_-)$. We already know
$\mu_-<\mu<2\mu_-$. So $\frac{1}{2\mu}<1/(\mu+\mu_-)<\frac{2}{3\mu}$ replaces
$\frac{1}{3}<\mu_-/(\mu+\mu_-)<\frac{1}{2}$ and gives
\[
  \frac{k^2}{(240\sigma_-+k^3h^2)\sigma_-}<\phi_e(\mu_-)<
  \frac{\mu^2}{3\sigma_-h+\mu^3}\cdot\frac{4}{3\sigma_- h}=\frac{8k^2}{(24\sigma_-+k^3h^2)3\sigma_-}.
\]
Hence, in terms of $\psi_e$ we have
\[
  \frac{k^2h^2}{(240\sigma_-+k^3h^2)4\sigma_-}<\psi_e(k_-)<
  \frac{2k^2h^2}{(24\sigma_-+k^3h^2)3\sigma_-}.
\]
The upper bound says that $\psi_e(k_-)$ converges to zero no slower than of order $k^2h^2$. This has
no difference in $h$ to the order $k^3h^2$ of $\psi(k_-)$, but it has one order less in
$k$. However, does this really mean $|E_1^h|_1$ converges at the same speed? Note that to make the
inequality \eqref{E1hf1} sharp, or more precisely here make the estimate for $E_1^h$ restricted to
$\xi \in [\pi,k_-]$ sharp, we need to take $f_{\mathrm{low}}$ along $\sin (k_-x)$ whose derivative
will contribute to $|f_{\mathrm{low}}|_1$ with an extra factor $k$. In other words, when the
frequency $k_-$ is active, $|E_1^h|_1$ is still of order $k^3h^2$, not improved by the smoothness of
$f$. This is actually observed in numerical experiments. Let
$f(x)=C(\sin(10\pi x) + \sin(20\pi x) + \sin(40\pi x) + \sin(80\pi x))$ with $C$ such that
$\|f\|=1$, and $k\in \{10, 20, 40, 80\}\pi + 1$. For each $k$, we use a very fine mesh with
$N\approx 8k^2$ to get a reference solution, and a sequence of doubly refined meshes starting with
$N\approx k$ to compute the $H^1$-semi-norm of the errors. The results are shown in
Figure~\ref{figk3h2} which corroborates the order $k^3h^2$. For example, at a fixed $h$, the marked
points from low to high correspond to successively doubled $k$'s, and form roughly a geometric
progression with the common ratio eight. Also, there are pairs of marked points, approximately at
the same height but on different curves, which correspond to $k$-values in a ratio of $4$ (or
$\frac{1}{4}$) and $h$-values in a ratio of $\frac{1}{8}$ (or $8$).

\begin{figure}
  \centering
  \includegraphics[scale=.5,trim=20 190 20 190,clip]{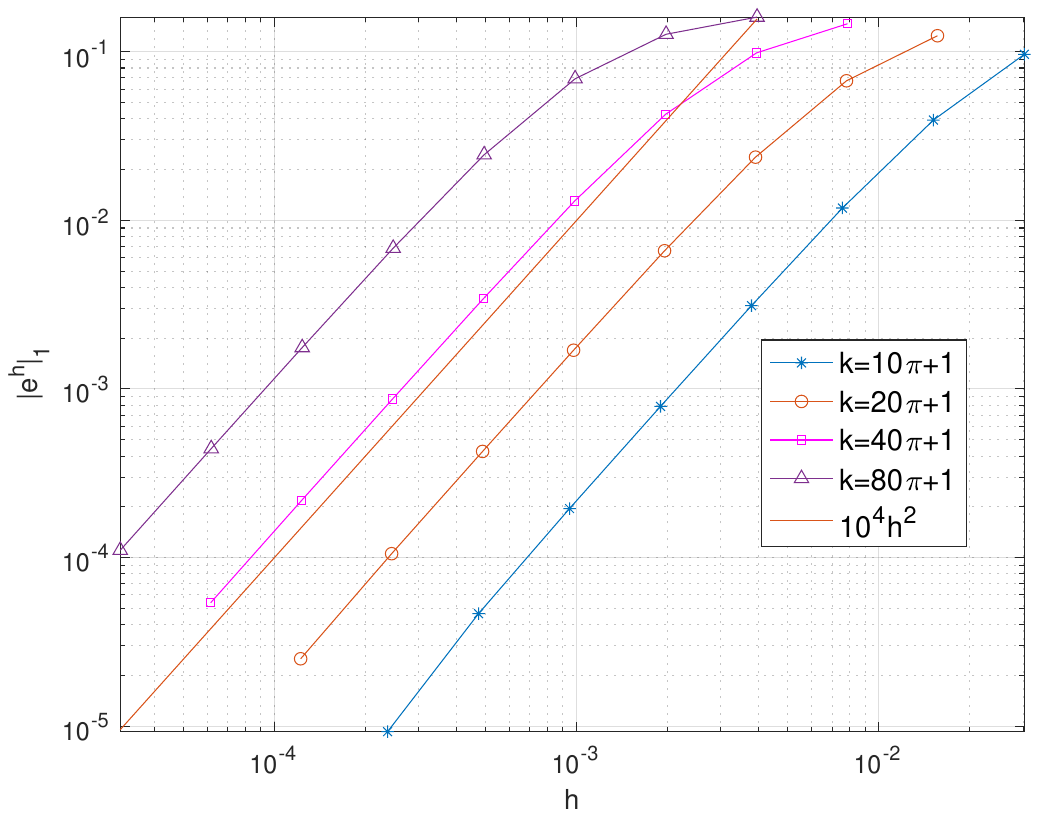}
  \caption{Errors in the $H^1$-semi-norm for
    $f(x)=C(\sin(10\pi x) + \sin(20\pi x) + \sin(40\pi x) + \sin(80\pi x))$.}\label{figk3h2}
\end{figure}

Having understood that we would not gain more from smoothness of $f$ for the propagating modes
(i.e. with $\xi\le k_+^h$), and in the need of a better estimate when $k^3h$ is sufficiently small
(so that $\max\psi=\psi(\xi_{\max})$ which is merely of order $h$), we derive the following
alternate estimates. By the way, we show all the evanescent modes (i.e. with $\xi> k_+^h$) can not
converge faster than of order $h^2$.


\begin{lemma}\label{lemevan}
  Under the assumptions of Lemma~\ref{lem3ptmax} and using the notation therein, it holds that
  \[
    \max_{\xi\in\{1, \ldots, N-1\}\pi}\psi_e=\max\{\psi_e(k_-), \psi_e(k_+), \psi_e(k_-^h),
    \psi_e(k_+^h), \psi_e(\xi_{\max})\},
  \]
  where the function $\psi_e$ is defined in \eqref{psie}.  Moreover, $\psi_e$ is convex on
  $(k,\frac{\pi}{h}]$. Let $\xi_e$ be the unique minimal point of $\psi_e$ over
  $(k,\frac{\pi}{h}]$. Then { \allowdisplaybreaks
    \begin{align*}
      \frac{k^2h^2}{(240\sigma_-+k^3h^2)4\sigma_-}<\psi_e(k_-)&\,<
                                                                \frac{2k^2h^2}{(24\sigma_-+k^3h^2)3\sigma_-},\,&\\
      \frac{1}{15}\frac{(k+\sigma_+)^2h^2}{\sigma_+(4\HZ{\sigma_{-}^h} + k^3h^2)}<
      \,\psi_e(k_+)&\,<\frac{1}{3}\frac{k^2h^2}{\sigma_-^h\sigma_+\sqrt{1-C_\mu^2}},\,&\text{ if
                                                                                        }\pi \mathbb{N}\cap(k,k^h)\ne\emptyset,\\
      \frac{(1-\frac{\pi^2}{80}C_{\mu}^2)(k+\sigma_+)^2h^2}{3\HZ{\sigma_-^h}(2+\pi)(4\sigma_++k^3h^2)}<\,
      \psi_e(k_-^h)\,&<\frac{\pi^4}{768}\frac{k^2h^2}{\sigma_-^h\sigma_+\sqrt{1-C_{\mu}^2}},\,&\text{ if }\pi \mathbb{N}\cap(k,k^h)\ne\emptyset,\\
      \frac{(1-\frac{(1+\pi)^2}{80}C_{\mu}^2)(k+\sigma_+^h)^2h^2}{\HZ
      {3(4\sigma_+^h+k^3h^2)^2}}<\,\psi_e(k_+^h)\,&<\frac{(1+\pi)^4k^2h^2}{32\sigma_+^h(24\sigma_+^h\!+\!k^3h^2)\cos(\theta_\mu\!+\!\frac{\sigma_+^hh}{2})},\,&\text{ if }C_{\mu}<\cos\frac{\sigma_+^hh}{2},\\
      \frac{(9\pi^2-64)h^2}{64\pi^2}<\psi_e(\xi_{\max})&\,<\frac{4(\pi^2-4)h^2}{\left(\cos^2\frac{\pi h}{2}-C_\mu^2\right)\left(9\pi^2-64\right)},\,&\text{ if }C_{\mu}<\cos\frac{\pi h}{2},\\
      \frac{h^2}{18}\,<\psi_e(\xi_e)&\,<\psi_e(\xi_{\max}).\,& 
    \end{align*}
  }
\end{lemma}

\begin{remark}
  Lemma~\ref{lemevan} gives the better and exact order $h^2$ for the evanescent modes, compared to
  Lemma~\ref{lem3ptmax}. It can be seen also from Lemma~\ref{lemevan} that the propagating modes
  usually converge slower than the evanescent modes because the $\psi_e$ at the former frequencies
  have the extra factor $k^2$. We propose to use \eqref{E1hf1} and Lemma~\ref{lemevan} only when
  $k^3h$ is sufficiently small (so that $\max\psi=\psi(\xi_{\max})$; see Lemma~\ref{lem3ptmax}). In
  that case, the estimates of Lemma~\ref{lemevan} can be further improved, but we will not make more
  efforts here.
\end{remark}

\begin{remark}
  Lemma~\ref{lemevan} is useful also for the $L^2$-norm of $E_1^h$ because
  \[
    \|E_1^h\|^2 = \frac{1}{2}
    \sum_{(\xi/\pi)=1}^{N-1}\left|\frac{1}{\lambda}-\frac{1}{\lambda^h}\right|^2|\hat{f}(\xi)|^2 \le
    \|f_{\mathrm{low}}\|^2 \max_{\xi\in\{1,..,N-1\}\pi}\left|\frac{1}{\lambda}-\frac{1}{\lambda^h}\right|^2.
  \]
  Hence, the exact order of $\|E_1^h\|$ is $k^2h^2$.
\end{remark}

\begin{proof}[Proof of Lemma~\ref{lemevan}]
  Recall the notation used in the proof of Lemma~\ref{lem3ptmax}. The estimate of $\psi_e(k_-)$ has
  been given in the above discussion, along with the result $\max_{(0,k_-]}\psi_e=\psi_e(k_-)$.

  When $\mu<\theta<\theta_k$, we have
  \[
    \phi_e(\theta)=\frac{1}{\theta^2-\mu^2} + \frac{1}{\mu^2-\sin^2\theta}=
    \frac{1}{2\mu}\left(\frac{1}{\theta-\mu}-\frac{1}{\theta+\mu}\right) +
    \frac{1}{2\mu}\left(\frac{1}{\mu-\sin\theta}+\frac{1}{\mu+\sin\theta}\right).
  \]
  To show that $\phi_e$ is convex on $(\mu,\theta_k)$, we calculate
  \[
    2\mu\phi_e''(\theta)=\frac{2}{(\theta-\mu)^3}-\frac{2}{(\theta+\mu)^3} -
    \frac{\sin\theta}{(\mu-\sin\theta)^2} + \frac{\sin\theta}{(\mu+\sin\theta)^2} +
    \frac{2\cos^2\theta}{(\mu-\sin\theta)^3} + \frac{2\cos^2\theta}{(\mu+\sin\theta)^3}.
  \]
  Using $0<\sin\theta<\mu<\theta<\frac{\pi}{2}$, we find
  \[
    \frac{2}{(\theta-\mu)^3}>\frac{2}{(\theta+\mu)^3},\quad
    \frac{2\cos^2\theta}{(\mu-\sin\theta)^3}>\frac{\sin\theta}{(\mu-\sin\theta)^2}
    \Leftarrow \sin\theta + \sin^2\theta < 2.
  \]
  Hence, $\phi_e''>0$ on $(\mu,\theta_k)$. So the candidates for $\max\phi_e$ on $(\mu,\theta_k)$
  are $\phi_e(\mu_+)$ and $\phi_e(\theta_-)$ if $\mu_+<\theta_k$. The estimate of $\phi_e(\mu_+)$ is
  quite similar to the estimate of $\phi(\mu_+)$ in the proof of Lemma~\ref{lem3ptmax}. We need only
  to replace $\mu_+/(\mu+\mu_+)$ there with $1/(\mu+\mu_+)$ here. Since $0<\mu<\mu_+$, we find
  \[
    \frac{1}{2\mu_+}<\frac{1}{\mu+\mu_+}<\frac{1}{2\mu}\quad \text{ instead of }\quad
    \frac{1}{2}<\frac{\mu_+}{\mu+\mu_+}<1
  \]
  should be used. Therefore, we get
  \[
    \frac{1}{15}\frac{(k+\sigma_+)^2h^2}{\sigma_+(4\HZ{\sigma_{-}^h} + k^3h^2)}<
    \,\psi_e(k_+)<\frac{1}{3}\frac{k^2h^2}{\sigma_-^h\sigma_+\sqrt{1-C_\mu^2}}\,,\text{ if
    }\pi \mathbb{N}\cap(k,k^h)\ne\emptyset.
  \]
  In a similar way, noting that $\mu+\theta_-<\mu_++\theta_k<\mu_++\frac{\pi}{2}\mu$ and using
  \[
    \frac{2}{(2+\pi) \mu_+}<\frac{1}{\mu+\theta_-}<\frac{1}{2\mu}\quad \text{ instead of }\quad
    \frac{1}{2}<\frac{\theta_-}{\mu+\theta_-}<1,
  \]
  we find
  \[
    \frac{(1-\frac{\pi^2}{80}C_{\mu}^2)(k+\sigma_+)^2h^2}{3\HZ{\sigma_{-}^h}(2+\pi)(4\sigma_++k^3h^2)}<\,\psi_e(k_-^h)<\frac{\pi^4}{768}\frac{k^2h^2}{\sigma_-^h\sigma_+\sqrt{1-C_{\mu}^2}}\,,\text{
      if }\pi \mathbb{N}\cap(k,k^h)\ne\emptyset.
  \]
  
  When $\theta_k<\theta<\frac{\pi}{2}$, with $\mu=\sin\theta_k$ we have that
  \begin{equation}\label{eqphie}
    \phi_e(\theta) = \frac{1}{\sin^2\theta-\mu^2}-\frac{1}{\theta^2-\mu^2}{\text{ has no local maximum for }\theta\in(\theta_k,\frac{\pi}{2})\text{ and }0<\mu<1}.
  \end{equation}
  {The proof can be found in Appendix~\ref{a:phie}.}  So the candidates for $\max\phi_e$ on
  $\Theta\cap(\theta_k,\theta_{\max}]$ are $\phi_e(\theta_+)$ (if $\theta_+<\theta_{\max}$) and
  $\phi_e(\theta_{\max})$.  To estimate $\phi_e(\theta_+)$, we can again modify the estimates of
  $\phi(\theta_+)$ in the proof of Lemma~\ref{lem3ptmax}.  Using
  \[
    \frac{1}{2 \theta_+}<\frac{1}{\mu+\theta_+}<\frac{1}{2\mu}\quad \text{ instead of }\quad
    \frac{1}{2}<\frac{\theta_+}{\mu+\theta_+}<1,
  \]
  we can find
  \[
    \frac{(1-\frac{(1+\pi)^2}{80}C_{\mu}^2)(k+\sigma_+^h)^2h^2}{\HZ{3(4\sigma_+^h+k^3h^2)^2}}<\,\psi_e(k_+^h)<\frac{(1+\pi)^4k^2h^2}{32\sigma_+^h(24\sigma_+^h+k^3h^2)\cos(\theta_\mu+\frac{\sigma_+^hh}{2})}\,,\text{ if }C_{\mu}<\cos\frac{\sigma_+^hh}{2}.
  \]
    
  For $\mu\!=\!\sin\theta_k$ and $\theta_k\!<\!\theta_{\max}$, we have
  \[
    \phi_e(\theta_{\max})=\frac{\theta_{\max}^2 -\sin^2\theta_{\max} }{\sin^2\theta_{\max} -\mu^2}
    \cdot \frac{1}{\theta_{\max}^2-\mu^2 }.
  \]
  Recall from the proof of Lemma~\ref{lem3ptmax},
  \[
    \frac{9\pi^2}{64}-1<\theta_{\max}^2-\sin^2\theta_{\max}<\frac{\pi^2}{4}-1,\quad
    1>\sin^2\theta_{\max}-\mu^2>\cos^2\frac{\pi h}{2}-C_\mu^2.
  \]
  Since $\theta_{\max}^2=\frac{\pi^2}{4}(1-2h+h^2)$, $0<\mu<1$ and $0<h\le \frac{1}{4}$, we find
  \[
    \frac{\pi^2}{4}>\theta_{\max}^2-\mu^2>\frac{9\pi^2}{64}-1.
  \]
  Combining this and using the assumption $C_{\mu}<\cos\frac{\pi h}{2}$, we obtain
  \[
    \left(\frac{9\pi^2}{64}-1\right)\frac{4}{\pi^2}<\phi_e(\theta_{\max})<\left(\frac{\pi^2}{4}-1\right)\left(\cos^2\frac{\pi h}{2}-C_\mu^2\right)^{-1}
    \left(\frac{9\pi^2}{64}-1\right)^{-1}.
  \]

  Denote $\theta_e:={\xi_eh}/{2}$.  To find the minimal point $\theta_e$ of $\phi_e$ on
  $[\theta_k,\frac{\pi}{2}]$, we calculate
  \[
    \phi_e'(\theta)=-\frac{\sin(2\theta)}{(\sin^2\theta-\mu^2)^2} +
    \frac{2\theta}{(\theta^2-\mu^2)^2},\quad\phi_e'\left(\frac{\pi}{2}\right)>0,\quad
    \lim_{\theta\to\theta_k+0}\phi_e'(\theta)<0.
  \]
  Since $\phi_e$ is convex in $[\theta_k,\frac{\pi}{2}]$ and
  $\theta_e=\arg\min_{[\theta_k,{\pi}/{2}]}\phi_e$, we have $\phi_e'(\theta_e)=0$,
  from which we find
  \[
    \frac{1}{\sin^2\theta_e-\mu^2}=\sqrt{\frac{2\theta_e}{\sin(2\theta_e)}}\cdot\frac{1}{\theta_e^2-\mu^2}.
  \]
  Substituting this into $\phi_e(\theta_e)$, and using $\frac{x}{\sin x}\!>\!1\!+\!\frac{x^2}{6}$
  for $0\!<\!x\!<\!\pi$ and $\sqrt{1+x}\!>\!1\!+\!\frac{x}{3}$ for $0\!<\!x\!<\!3$, we get
  \[
    \phi_e(\theta_e)=\left(\sqrt{\frac{2\theta_e}{\sin(2\theta_e)}}-1\right)\frac{1}{\theta_e^2-\mu^2}>\frac{4\theta_e^2}{18}\cdot\frac{1}{\theta_e^2}=\frac{2}{9}.
  \]
  Multiplying $\phi_e$ with $\frac{h^2}{4}$ gives the conclusion for $\psi_e$.
\end{proof}

Putting all the results on the downsampling, aliasing and operator errors together, we have an
estimate of the total error in the $H^1$-semi-norm and $L^2$-norm.

\begin{theorem}[Absolute error with nonzero $f$]\label{abserr}
  Suppose the problem \eqref{1dhelm} has $k>2\pi$, $\sigma_k:=\min_{\xi\in \pi\mathbb{N}}|k-\xi|>0$,
  $f\in H_0^p(0,1)$ with $p\ge 1$, $g_0=g_1=0$. Suppose
  $0<\frac{kh}{2}\le C_\mu<\cos\frac{\pi h}{2}$, and there is a sequence of $h=\frac{1}{N}$
  ($N\ge4$) going to zero such that
  \[
    \tilde{\sigma}_k:=\min_{\xi,h}\left|k-\frac{2}{h}\sin\frac{\xi h}{2}\right|>0\; \text{over
      $\xi\in\{1,\ldots,N-1\}\pi$ and the sequence of $h$ going to zero}.
  \]
  Let $f$ be decomposed as \eqref{flowhigh}. Then the finite difference scheme \eqref{1d3pt} has a
  unique solution $u^h$. Moreover,
  \[
    |u-u^h|_1\le \left(\frac{\HZ{h(kh)^p}}{\pi^2-4C_{\mu}^2}+ \frac{\HZ{(kh)^p}}{\tilde{\sigma}_k}\right)\pi^{1-p} \HZ{k^{-p}}|f_{\mathrm{high}}|_p+ |E_1^h|_1,
  \]
  \[
    \|u-u^h\|\le \left(\frac{\HZ{h^2(kh)^{p}}}{\pi^2-4C_{\mu}^2}+ \frac{2\HZ{(kh)^{p}}}{\tilde{\sigma}_kk}\right)\pi^{-p}\HZ{k^{-p}}|f_{\mathrm{high}}|_p+ \|E_1^h\|,
  \]
  with $E_1^h$ given in \eqref{E1hdef}, and equality is attained when $f_{\mathrm{high}}=0$. In
  turn, $E_1^h$ satisfies
  \[
    \frac{h^2}{18}|f_{\mathrm{low,e}}|_1\le|E_1^h|_1\le \min\{C(k,h)\|f_{\mathrm{low}}\|, C_e(k,h)|f_{\mathrm{low}}|_1\},%
    \quad \frac{h^2}{18}\|f_{\mathrm{low,e}}\|\le\|E_1^h\|\le C_e(k,h)\|f_{\mathrm{low}}\|,
  \]
  where $f_{\mathrm{low,e}}:=\sum_{(\xi/\pi)=\lceil k^h/\pi\rceil}^{N-1}\hat{f}(\xi)\sin(\xi x)$,
  $k^h:=\frac{2}{h}\arcsin\frac{kh}{2}$, the upper bounds are attainable, and there exist positive
  constants $C_1$, $C_2$, $c$, $c_e$, $\tilde{C}$ and $\tilde{C}_e$ independent of $k$ and $h$ such
  that, when $k^3h^2<C_1\HZ{\sigma_k}$ it holds that
  \[
    c k^3h^2\HZ{/\sigma_k^2}<C(k,h)<\tilde{C}k^3h^2\HZ{/\sigma_k^2},\quad
    c_e k^3h^2\HZ{/\sigma_k^2}<kC_e(k,h)<\tilde{C}_ek^3h^2\HZ{/\sigma_k^2},%
  \]
  and when $k^3h^2>C_2\HZ{\max(\sigma_k,\sigma_-^h)}$ it holds that $C(k,h)$ or $kC_e(k,h)$ is
  greater than a positive constant \HZ{times $1/\sigma_k$, where $\sigma_-^h$ is defined in
    Lemma~\ref{lem3ptmax}}.
\end{theorem}

\begin{remark}
  Here $C_e(k,h)$ is multiplied with $k$ because $|E_1^h|_1\le C_e(k,h)|f_{\mathrm{low}}|_1$ holds
  with equality when $f_{\mathrm{low}}=\sin(\xi x)$ for some $\xi\in\pi\mathbb{N}$ satisfying
  $|\xi-k|< \pi$ and in this case $|f_{\mathrm{low}}|_1$ is of order $k$. {Note that the upper
    bound is attained when $f$ is chosen dependently on $k$ and as a result $|f|_p=O(k^p)$.}
\end{remark}

\begin{proof}
  The proof is obtained by combining \eqref{errpart}, \eqref{downerr}, \eqref{downerrL2},
  Lemma~\ref{lemaliasH1}, Lemma~\ref{lemaliasL2}, \eqref{3ptE1E2}, \eqref{E1h},
  Lemma~\ref{lem3ptmax}, \eqref{E1hf1} and Lemma~\ref{lemevan}.
\end{proof}

\subsection{Relative error estimates of the classical 3-point centered scheme}

To estimate the relative error, we recall the solution $u$ of \eqref{1dhelm} and the solution $u^h$
of \eqref{1d3pt} in case (i), $g_0=g_1=0$, are
\[
  u(x) = \sum_{\xi\in\pi\mathbb{N}}\frac{\hat{f}(\xi)}{k^2-\xi^2}\sin(\xi x),\quad %
  u^h(x) = \sum_{(\xi/\pi)=1}^{N-1}\frac{\widehat{f^h}(\xi)}{k^2-\frac{4}{h^2}\sin^2\frac{\xi h}{2}}
  \sin(\xi x),
\]
\[
  \widehat{f^h}(\xi) =\hat{f}(\xi)+\sum_{s=\pm1}s\sum_{m=1}^{\infty}\hat{f}(s\xi+2mN\pi),%
  \text{ for }\xi\in\{1,\ldots,N-1\}\pi.
\]
We have already shown that $f_{\mathrm{high}}=\sum_{(\xi/\pi)=N+1}^{\infty}\hat{f}(\xi)\sin(\xi x)$
causes the downsampling error \eqref{downerr} and aliasing error \eqref{E2}. Let us see how they
behave in the relative sense.

\begin{lemma}\label{lemdownrel}
  Suppose $k>\frac{4\pi}{\pi-\sqrt{16-\pi^2}}\approx 18.88$, $0<\frac{kh}{2}\le C_\mu<1$. 
  Then, for 
  $p\ge 1$, the downsampling error
  $e_2^h$ in \eqref{errpart} satisfies
  \[
    \frac{\|e_2^h\|}{\|u\|}< \frac{\HZ{(kh)^p}}{\pi^{p}}\frac{\HZ{k^{-p}}|f_{\mathrm{high}}|_p}{\|f_{\mathrm{low}}\|},\quad %
    \frac{|e_2^h|_1}{|u|_1}< \frac{\HZ{(kh)^{p-1}}}{\pi^{p-1}}%
    \frac{\HZ{k^{-p}}|f_{\mathrm{high}}|_p}{\HZ{k^{-1}}|f_{\mathrm{low}}|_1}.
  \]
\end{lemma}

\begin{remark}
  Recall from Remark~\ref{fpartsconv} that as $N\to \infty$ we have $f_{\mathrm{high}}\to 0$ and
  $f_{\mathrm{low}}\to f$ for a fixed $f\in H^p_0(0,1)$. So the relative errors converge faster than
  the displayed powers of $h$.  {But if $|f_{\mathrm{high}}|_p=O(k^p)$ and $kh=O(1)$, then
    Lemma~\ref{lemdownrel} says the relative error is bounded in $L^2$ and $H^1$.}
\end{remark}

\begin{remark}
  Compared to the absolute errors of downsampling \eqref{downerr} and \eqref{downerrL2}, the
  exponents of $h$ here are decreased by two. This is caused by the fact that $f_{\mathrm{low}}$
  depends on $h$ and the bound for $u_{\mathrm{low}}$ becomes sharp only for increasingly
  oscillatory $f$ (see the following proof). If $f$ and $k$ are fixed, then $u$ is fixed and the
  relative error should be of the same order in $h$ as the absolute error is.
\end{remark}

\begin{proof}[Proof of Lemma~\ref{lemdownrel}]
  By Parseval's identity, we have
  \[
    \frac{\|e_2^h\|^2}{\|u\|^2}=%
    \frac{\sum_{(\xi/\pi)=N}^\infty\frac{1}{|k^2-\xi^2|^2}|\hat{f}(\xi)|^2}%
    {\left(\sum_{(\xi/\pi)=1}^{N-1}+\sum_{(\xi/\pi)=N}^\infty\right)%
      \frac{1}{|k^2-\xi^2|^2}|\hat{f}(\xi)|^2}=\frac{\|e_2^h\|^2}{\|u_{\mathrm{low}}\|^2+\|e_2^h\|^2}.
  \]
  Note that
  \[
    \|e_2^h\|\le \max_{\xi\ge N\pi}\frac{1}{|k^2-\xi^2|\cdot|\xi|^{p}} |f_{\mathrm{high}}|_p,%
    \quad \|u_{\mathrm{low}}\|\ge \min_{\xi\in\{1,..,N-1\}\pi}\frac{1}{|k^2-\xi^2|}\|f_{\mathrm{low}}\|.
  \]
  We consider the case of $(N-1)^2\pi^2-k^2>k^2-\pi^2>0$. This is guaranteed by
  $\frac{2}{h}\!>\!k\!>\!\frac{4\pi}{\pi-\sqrt{16-\pi^2}}$:
  \[
    (N\!-\!1)^2\pi^2\!-\!k^2\!>\!k^2\!-\!\pi^2\Leftrightarrow ((N\!-\!1)^2\!+\!1)\pi^2 %
    \!>\!2k^2\Leftarrow ((1\!-\!h)^2\!+\!h^2)\pi^2\!>\!8 \Leftarrow
    h<\frac{2}{k}\!<\!\frac{\pi\!-\!\sqrt{{16}\!-\!\pi^2}}{2\pi}.
  \]
  In this case, we have
  \[
    \min_{\xi\in\{1,..,N-1\}\pi}\frac{1}{|k^2-\xi^2|} = \frac{1}{(N-1)^2\pi^2-k^2},\quad %
    \max_{\xi\ge N\pi}\frac{1}{|k^2-\xi^2|\cdot|\xi|^{p}}=\frac{1}{N^2\pi^2-k^2}\frac{1}{N^p\pi^p}.
  \]
  Therefore, it holds that
  \[
    \frac{\|e_2^h\|}{\|u\|}\le \frac{\|e_2^h\|}{\|u_{\mathrm{low}}\|}\le
    \frac{(N-1)^2\pi^2-k^2}{N^2\pi^2-k^2}\frac{1}{N^p\pi^p}\frac{|f_{\mathrm{high}}|_p}{\|f_{\mathrm{low}}\|}<
    \frac{h^{p}}{\pi^{p}}\frac{|f_{\mathrm{high}}|_p}{\|f_{\mathrm{low}}\|}.
  \]
  For the $H^1$-semi-norm of $e_2^h$, we note that
  \[
    |e_2^h|_1\le \max_{\xi\ge N\pi}\frac{1}{|k^2-\xi^2|\cdot|\xi|^{p-1}} |f_{\mathrm{high}}|_p,%
    \quad |u_{\mathrm{low}}|_1\ge \min_{\xi\in\{1,..,N-1\}\pi}\frac{1}{|k^2-\xi^2|}|f_{\mathrm{low}}|_1.
  \]
  The estimation of the relative error is similar to the previous one in the $L^2$-norm.
\end{proof}

\begin{lemma}\label{lemaliasrel}
  Suppose $k>\frac{4\pi}{\pi-\sqrt{16-\pi^2}}\approx 18.88$, $0<\frac{kh}{2}\le C_\mu<1$, and there
  exists a sequence of $h=\frac{1}{N}$ going to zero such that
  \[
    \tilde{\sigma}_k:=\min_{\xi,h}\left|k-\frac{2}{h}\sin\frac{\xi h}{2}\right|>0\; \text{over
      $\xi\in\{1,\ldots,N-1\}\pi$ and the sequence of $h$ going to zero}.
  \]
  Then, for $\lambda^h=k^2-\frac{4}{h^2}\sin^2\frac{\xi h}{2}$ and $p\ge 2$, the aliasing error
  $E_2^h$ in \eqref{3ptE10E20} satisfies
  \[
    \frac{\|E_2^h\|}{\|u\|}< \frac{2\HZ{k(kh)^{p-2}}}{\tilde{\sigma}_k\pi^{p-2}}%
    \frac{\HZ{k^{-p}}|f_{\mathrm{high}}|_p}{\|f_{\mathrm{low}}\|},\quad %
    \frac{|E_2^h|_1}{|u|_1}< \frac{\HZ{k(kh)^{p-2}}}{\tilde{\sigma}_k\pi^{p-3}}%
    \frac{\HZ{k^{-p}}|f_{\mathrm{high}}|_p}{\HZ{k^{-1}}|f_{\mathrm{low}}|_1}.
  \]
\end{lemma}

\begin{proof}
  We invoke Lemma~\ref{lemaliasL2} and Lemma~\ref{lemaliasH1} for bounds of $E_2^h$, and recall from
  the proof of Lemma~\ref{lemdownrel} that
  \[
    \|u_{\mathrm{low}}\|\ge \min_{\xi\in\{1,..,N-1\}\pi}\frac{1}{|k^2-\xi^2|}\|f_{\mathrm{low}}\|,\quad%
    |u_{\mathrm{low}}|_1\ge \min_{\xi\in\{1,..,N-1\}\pi}\frac{1}{|k^2-\xi^2|}|f_{\mathrm{low}}|_1,
  \]
  \[
    \min_{\xi\in\{1,..,N-1\}\pi}\frac{1}{|k^2-\xi^2|} = \frac{1}{(N-1)^2\pi^2-k^2}.
  \]
  Combining all this gives the upper bounds for $\|E_2^h\|/\|u_{\mathrm{low}}\|\!\ge\!\|E_2^h\|/\|u\|$
  and $|E_2^h|_1/|u_{\mathrm{low}}|_1\!\ge\!|E_2^h|_1/|u|_1$.
\end{proof}

Now we start estimating the relative error due to operator discretization. That is, for $E_1^h$ in
\eqref{3ptE10E20},
\begin{equation}\label{E1hrel}
  \frac{\|E_1^h\|^2}{\|u\|^2}=\frac{\sum_{(\xi/\pi)=1}^{N-1}%
    \left|\frac{1}{\lambda}-\frac{1}{\lambda^h}\right|^2|\hat{f}(\xi)|^2}{\|u\|^2}%
  =\sum_{(\xi/\pi)=1}^{N-1}w^h(\xi)\frac{|\lambda(\xi)-\lambda^h(\xi)|^2}{|\xi^p\lambda^h(\xi)|^2},
\end{equation}
where the weight function $w^h(\xi)$ is given by
\begin{equation}\label{wh}
  w^h(\xi):=\frac{|\xi^p\hat{u}(\xi)|^2}{\|u\|^2}=%
  \frac{|\xi^p\hat{f}(\xi)|^2}{|\lambda(\xi)|^2\|u\|^2},\quad %
  \text{ and }\sum_{(\xi/\pi)=1}^{N-1}w^h(\xi)=\frac{|u_{\mathrm{low}}|_p^2}{\|u\|^2}.
\end{equation}
Note that as $h=\frac{1}{N}\to 0$, $|u_{\mathrm{low}}|_p\to |u|_p$. It can be deduced from \eqref{E1hrel} and
\eqref{wh} that
\begin{equation}\label{E1hrelL2max}
  \frac{\|E_1^h\|}{\|u\|}\le \frac{|u_{\mathrm{low}}|_p}{\|u\|} %
  \max_{\xi\in\{1,..,N-1\}\pi}\left|\frac{\lambda-\lambda^h}{\xi^p\lambda^h}\right|,
\end{equation}
which holds with equality if and only if $\hat{f}(\xi)=0$ for
$\xi\in\{\pi,\ldots,(N-1)\pi\}\backslash\{\xi_*\}$ where $\{\xi_*\}$ is the maximizer set of
$|\lambda-\lambda^h|/|\xi^p\lambda^h|$ over $\xi\in\{1,\ldots,N-1\}\pi$. So, if $\xi_*$ does not
change with $N$, then the equality holds for the fixed $f(x)=\sin(\xi_*x)$ and
$|u_{\mathrm{low}}|_p/\|u\|=\xi_*^p$. Our intention is to find, for a given wavenumber $k$, a source
$f$ independent of the mesh size $h=1/N$ (but $f$ may depend on $k$) such that the upper bound is
attainable. We will see that $p=2$ is an appropriate choice.

The relative error in the $H^1$-semi-norm can be analyzed in a similar way. Specifically, we have
\[
  \frac{|E_1^h|_1^2}{|u|_1^2}=\frac{\sum_{(\xi/\pi)=1}^{N-1}%
    \left|\frac{1}{\lambda}-\frac{1}{\lambda^h}\right|^2|\xi\hat{f}(\xi)|^2}{|u|_1^2}%
  =\sum_{(\xi/\pi)=1}^{N-1}\frac{|\xi^{p+1}\hat{f}(\xi)|^2}{|\lambda(\xi)|^2|u|_1^2}\cdot %
  \frac{|\lambda(\xi)-\lambda^h(\xi)|^2}{|\xi^p\lambda^h(\xi)|^2},
\]
which leads to
\begin{equation}\label{E1hrelH1max}
  \frac{|E_1^h|_1}{|u|_1}\le\frac{|u_{\mathrm{low}}|_{p+1}}{|u|_1} %
  \max_{\xi\in\{1,..,N-1\}\pi}\left|\frac{\lambda-\lambda^h}{\xi^p\lambda^h}\right|.
\end{equation}
So for both the $L^2$- and $H^1$-norms it is essential to find the same maximum.

\begin{lemma}\label{maxL2rel}
  Suppose $2\pi<k\not\in \pi\mathbb{N}$, $\frac{kh}{2}\le C_{\mu}\le \frac{3}{4}$ and a sequence of
  $h=\frac{1}{N}<\frac{1}{2\pi}$ satisfying
  \[
    \tilde{\sigma}_k:=\min_{\xi,h}\left|k-\frac{2}{h}\sin\frac{\xi h}{2}\right|>0\; \text{over
      $\xi\in\{1,\ldots,N-1\}\pi$ and the sequence of $h$ going to zero}.
  \]
  Let $\xi_{\max}:=(N-1)\pi$, $k^h:=\frac{2}{h}\arcsin\frac{kh}{2}$, $k_{\pm}^h\in \pi\mathbb{N}$
  such that $k_-^h<k^h<k_+^h$ and $k_+^h-k_-^h=\pi$. Denote
  $\tilde{\sigma}_{\pm}^h:=|k-\frac{2}{h}\sin\frac{k_{\pm}^hh}{2}|$.  Let $\lambda:=k^2-\xi^2$,
  $\lambda^h:=k^2-\frac{4}{h^2}\sin^2\frac{\xi h}{2}$ and
  $\psi_{rel}(\xi):=\left|\frac{\lambda-\lambda^h}{\xi^2\lambda^h}\right|$. Then
  $\max_{\xi\in\{1,..,N-1\}\pi}\psi_{rel}(\xi)$ is attained at one of $k_-^h$, $k_+^h$ and
  $\xi_{\max}$ with
  \[
    \frac{1}{196\tilde{\sigma}_-^h}kh^2\!<\!\psi_{rel}(k_-^h)\!<\!%
    \frac{\pi^3}{48(2+\pi)\tilde{\sigma}_-^h}kh^2,\quad%
    \frac{1}{64\tilde{\sigma}_+^h}kh^2\!<\!\psi_{rel}(k_+^h)\!<\! %
    \frac{3\pi}{32\tilde{\sigma}_+^h}kh^2,\quad
    \frac{9}{100}h^2\!<\!\psi_{rel}(\xi_{\max})\!<\!\frac{2}{3}h^2,
  \]
  where $\pi\ge \tilde{\sigma}_{\pm}^h\ge \tilde{\sigma}_k$.
\end{lemma}

\begin{proof}
  Let $\mu:=\frac{kh}{2}$, $\theta:=\frac{\xi h}{2}$. We find
  $\psi_{rel}(\xi)=\frac{h^2}{4}\cdot\frac{\theta^2-\sin^2\theta}{\theta^2|\mu^2-\sin^2\theta|} %
  =:\frac{h^2}{4}\phi_{rel}(\theta)$. When $\sin\theta<\mu$,
  \[
    \phi_{rel}(\theta)=\frac{1-\frac{\sin^2\theta}{\theta^2}}{\mu^2-\sin^2\theta}%
    \quad\text{increases with }\theta\text{ for }0<\theta<\frac{\pi}{2}\text{ and }\sin\theta<\mu.
  \]
  So $k_-^h$ is the unique maximizer of $\psi_{rel}(\xi)$ on $(0,k_-^h]$. When $\sin\theta>\mu$, we
  have
  \[
    \phi_{rel}(\theta)=\frac{\theta^2-\sin^2\theta}{\theta^2(\sin^2\theta-\mu^2)} \quad{\text{has
        no local maximum for }0<\theta_k:=\frac{h}{2}k^h=\arcsin\mu<\theta<\frac{\pi}{2}}.
  \]
  {Indeed, we calculate
  \[
    \theta ^3 \left(\sin ^2\theta-\mu ^2\right)^2\phi_{rel}'(\theta)= -2 \left(\theta ^3 \sin \theta
      \cos \theta-\sin ^4\theta+\mu^2 \sin ^2\theta-\theta \mu^2 \sin \theta \cos \theta\right)
  \]
  and a critical point $\theta_0\in (\arcsin\mu,\frac{\pi}{2})$ of $\phi_{rel}$ is a root of
  \begin{equation}\label{theta0}
    \mu^2=\frac{\theta ^3 \cos \theta-\sin ^3\theta}{\theta  \cos \theta-\sin \theta}
    =:v(\theta).
  \end{equation}
  By calculations, we find
  \[
    (\theta \cos \theta-\sin \theta)^2v'(\theta)= (\theta -\sin \theta) (\theta +\sin \theta)
    \left(\theta \sin ^2\theta+3 \theta \cos ^2\theta-3 \sin \theta \cos\theta\right)>0,
  \]
  because
  $\left(\theta \sin ^2\theta+3 \theta \cos ^2\theta-3 \sin \theta \cos \theta\right)'=4 \sin
  ^2\theta-4 \theta \sin \theta \cos \theta>0$ for $\theta\in(0,\frac{\pi}{2})$. Hence,
  $v'(\theta)>0$ and there is at most one critical point of $\phi_{rel}$. Suppose
  $\theta_0\in (\arcsin\mu,\frac{\pi}{2})$ is a critical point of $\phi_{rel}$. By substitution with
  \eqref{theta0}, we find
  \[
    \theta_0 ^5 \left(\theta_0 ^2-\sin ^2\theta_0\right)\phi_{rel}''(\theta_0)= 2 \sin\theta_0 \cos
    \theta_0 (\tan\theta_0-\theta_0) \left(3 \theta_0 +\theta_0 \tan ^2\theta_0-3 \tan
      \theta_0\right)>0,
  \]
  because
  $\left(3 \theta +\theta \tan ^2\theta-3 \tan\theta\right)'=2 \tan\theta \left(\theta \sec ^2\theta
    -\tan\theta\right)>0$ for $\theta\in(0,\frac{\pi}{2})$. That means $\theta_0$ can only be a
  local minimum.}
  
  It remains to estimate all the candidates $\phi_{rel}(\theta_-)$, $\phi_{rel}(\theta_+)$ and
  $\phi_{rel}(\theta_{\max})$ for $\max\phi_{rel}$, where $\theta_{\pm}:=\frac{k_{\pm}^hh}{2}$ and
  $\theta_{\max}:=\frac{(N-1)\pi h}{2}$, corresponding to $\psi_{rel}(\xi)$ at
  $\xi=k_{\pm}^h, \xi_{\max}$. Note that
  $\frac{h}{2}\pi>\frac{h}{2}|k^h-k_{\pm}^h|=|\theta_k-\theta_{\pm}|\ge
  |\sin\theta_k-\sin\theta_{\pm}|=\frac{h}{2}\left|k-\frac{2}{h}\sin\frac{k_{\pm}^hh}{2}\right|
  =\frac{h}{2}\tilde{\sigma}_{\pm}^h\ge \frac{h}{2}\tilde{\sigma}_k$,
  $k_-^h<k^h=\frac{2}{h}\arcsin\frac{kh}{2}<\frac{\pi}{2}k$, $k^h_->k^h-\pi>k-\pi>\frac{k}{2}$, and
  $k<k_+^h<k^h+\pi<\frac{\pi}{2}k+\frac{1}{2}k<3k$. These are useful in the following.

  First, at $\theta_-$, we have
  \[
    \phi_{rel}(\theta_-)=\frac{\theta_-^2-\sin^2\theta_-}{\theta_-^2(\sin^2\theta_k-\sin^2\theta_-)}
    =\left(1-\frac{\sin\theta_-}{\theta_-}\right)\left(1+\frac{\sin\theta_-}{\theta_-}\right) %
    \frac{1}{\sin\theta_k-\sin\theta_-}\cdot \frac{1}{\sin\theta_k+\sin\theta_-}.
  \]
  Note that $x-\frac{x^3}{6}<\sin x<x-\frac{x^3}{8}$ and $x<\frac{\pi}{2}\sin x$ for
  $x\in(0,\frac{\pi}{2})$. We get
  \[
    \phi_{rel}(\theta_-)<\frac{1}{24}(k_-^h)^2h^2\cdot 2\cdot\frac{2}{h\tilde{\sigma}_-^h}%
    \cdot \frac{\pi}{(k+k_-^h)h}<\frac{\pi^3}{12(2+\pi)\tilde{\sigma}_-^h}k,
  \]
  \[
    \phi_{rel}(\theta_-)>\frac{1}{32}(k_-^h)^2h^2\cdot 1\cdot\frac{2}{h\tilde{\sigma}_-^h}%
    \cdot \frac{2}{h(k+k_-^h)}>\frac{1}{48\tilde{\sigma}_-^h}k.
  \]

  Second, at $\theta_+$, we have
  \[
    \phi_{rel}(\theta_+)=\frac{\theta_+^2-\sin^2\theta_+}{\theta_+^2(\sin^2\theta_+-\sin^2\theta_k)}
    =\left(1-\frac{\sin\theta_+}{\theta_+}\right)\left(1+\frac{\sin\theta_+}{\theta_+}\right) %
    \frac{1}{\sin\theta_+-\sin\theta_k}\cdot \frac{1}{\sin\theta_k+\sin\theta_+}.
  \]
  Similar to the estimate of $\phi_{rel}(\theta_-)$, we get
  \[
    \phi_{rel}(\theta_+)<\frac{1}{24}(k_+^h)^2h^2\cdot 2\cdot\frac{2}{h\tilde{\sigma}_+^h}%
    \cdot \frac{\pi}{(k+k_+^h)h}<\frac{3\pi}{8\tilde{\sigma}_+^h}k,
  \]
  \[
    \phi_{rel}(\theta_+)>\frac{1}{32}(k_+^h)^2h^2\cdot 1\cdot\frac{2}{h\tilde{\sigma}_+^h}%
    \cdot \frac{2}{h(k+k_+^h)}>\frac{1}{16\tilde{\sigma}_+^h}k.
  \]

  Finally, at $\theta_{\max}=\frac{\pi}{2}(1-h)$, we have
  \[
    \phi_{rel}(\theta_{\max})=\frac{\theta_{\max}^2-\sin^2\theta_{\max}}%
    {\theta_{\max}^2(\sin^2\theta_{\max}-\sin^2\theta_k)}=
    \frac{\frac{\pi^2}{4}(1-h)^2-\cos^2\frac{\pi h}{2}}%
    {\frac{\pi^2}{4}(1-h)^2(\cos^2\frac{\pi h}{2}-\frac{k^2h^2}{4})}.
  \]
  From the assumption $0<h<\frac{1}{2\pi}$, using also $\cos x > 1-\frac{1}{2}x^2$ for
  $x\in(0,\frac{\pi}{2})$ we have
  $\cos^2\frac{\pi h}{2}=\frac{1}{2}(1+\cos(\pi h))>1-\frac{\pi^2h^2}{4}>\frac{15}{16}$. Combining
  this with the assumption $\frac{kh}{2}<\frac{3}{4}$ gives
  \[
    \frac{9}{25}=1-\frac{4}{\pi^2(\frac{3}{\pi}-\frac{1}{2\pi})^2}<\sec^2\frac{\pi
      h}{2}-\frac{4}{\pi^2(1-h)^2}<\phi_{rel}(\theta_{\max})<\frac{1}{\cos^2\frac{\pi
        h}{2}-\frac{k^2h^2}{4}}<\frac{8}{3}.
  \]
  To conclude for $\psi_{rel}$, we need only to multiply $\phi_{rel}$ with $\frac{h^2}{4}$.
\end{proof}

Now we can summarize the relative error.

\begin{theorem}[Relative error with nonzero $f$]\label{relerr}
  Suppose the problem \eqref{1dhelm} has $k>\frac{4\pi}{\pi-\sqrt{16-\pi^2}}\approx 18.88$,
  $\sigma_k:=\min_{\xi\in \pi\mathbb{N}}|k-\xi|>0$, $f\in H_0^p(0,1)$ with $p\ge 2$,
  $g_0=g_1=0$. Suppose $0<\frac{kh}{2}\le C_\mu\le\frac{3}{4}$, and there is a sequence of
  $h=\frac{1}{N}$ going to zero such that
  \[
    \tilde{\sigma}_k:=\min_{\xi,h}\left|k-\frac{2}{h}\sin\frac{\xi h}{2}\right|>0\; \text{over
      $\xi\in\{1,\ldots,N-1\}\pi$ and the sequence of $h$ going to zero}.
  \]
  Let $k^h:=\frac{2}{h}\arcsin\frac{kh}{2}$, $k_{\pm}^h\in \pi\mathbb{N}$ such that
  $k_-^h<k^h<k_+^h$ and $k_+^h-k_-^h=\pi$. Let $f$ be decomposed as in \eqref{flowhigh}. Then the
  finite difference scheme \eqref{1d3pt} has a unique solution $u^h$. Moreover,
  \[
    \frac{\|u-u^h\|}{\|u\|}\le \left(\frac{\HZ{(kh)^{p}}}{\pi^{p}}+
      \frac{2\HZ{k(kh)^{p-2}}}{\tilde{\sigma}_k\pi^{p-2}}\right)\frac{\HZ{k^{-p}}|f_{\mathrm{high}}|_p}{\|f_{\mathrm{low}}\|} %
    + \frac{\|E_1^h\|}{\|u\|}, %
  \]
  \[
    \frac{|u-u^h|_1}{|u|_1}\le\left(\frac{\HZ{k(kh)^{p-1}}}{\pi^{p-1}}+%
      \frac{\HZ{k(kh)^{p-2}}}{\tilde{\sigma}_k\pi^{p-2}}\right)\frac{\HZ{k^{-p}}|f_{\mathrm{high}}|_p}{\HZ{k^{-1}}|f_{\mathrm{low}}|_1} +
    \frac{|E_1^h|_1}{|u|_1},
  \]
  with $E_1^h$ given in \eqref{E1hdef}, and equality is attained when $f_{\mathrm{high}}=0$. In
  turn, $E_1^h$ satisfies
  \[
    \frac{\|E_1^h\|}{\|u\|}\le \frac{\HZ{k^{-2}}|u_{\mathrm{low}}|_2}{\|u\|}\HZ{k^2}C_{rel}(k,h),\quad %
    \frac{|E_1^h|_1}{|u|_1}\le \frac{\HZ{k^{-3}}|u_{\mathrm{low}}|_{3}}{\HZ{k^{-1}}|u|_1}\HZ{k^2}C_{rel}(k,h),
  \]
  where the upper bounds are attained when $\hat{f}(\xi)=0$ for all $\xi\in\{1,\ldots,N-1\}\pi$ but
  one of $k_{\pm}^h$, and there exist positive constants $c_{rel}$ and $\tilde{C}_{rel}$ independent
  of $k$, \HZ{$\tilde{\sigma}_{\pm}^h$} and $h$ such that
  \[
    c_{rel}
    \HZ{k^3}h^2\HZ{/\tilde{\sigma}_{k}^h}<\HZ{k^2}C_{rel}(k,h)<\tilde{C}_{rel}\HZ{k^3}h^2\HZ{/\tilde{\sigma}_{k}^h}\HZ{,}
  \]
  \HZ{where $\tilde{\sigma}^h_k=\min\{\tilde{\sigma}_{\pm}^h\}$, and $\tilde{\sigma}_{\pm}^h$ are
    defined in Lemma~\ref{maxL2rel}}.  In particular, when the upper bounds for $E_1^h$ relative to
  $u$ are attained, e.g.  $f(x)=2\sin(k_{\pm}^hx)$, we have
  \[
    u=\frac{f}{k^2-(k_{\pm}^h)^2},\quad%
    \|u\|=\frac{\|f\|}{k^2-(k_{\pm}^h)^2},\quad
    |u|_p=|u_{\mathrm{low}}|_p=\frac{\|f\|(k_{\pm}^h)^p}{k^2-(k_{\pm}^h)^2}, %
  \]
  and hence the exact order of the relative errors are $k^3h^2\HZ{/\tilde{\sigma}_k^h}$, or more
  precisely
  \[
    \frac{\|E_1^h\|}{\|u\|}=\frac{|E_1^h|_1}{|u|_1}= %
    (k_{\pm}^{h})^2C_{rel}(k,h)\in
    (k_{\pm}^{h})^2kh^2\HZ{(\tilde{\sigma}_k^h)^{-1}}(c_{rel},\tilde{C}_{rel}).
  \]
\end{theorem}

\begin{proof}
  The proof is obtained by combining Lemma~\ref{lemdownrel}, Lemma~\ref{lemaliasrel},
  \eqref{E1hrelL2max} \eqref{E1hrelH1max} and Lemma~\ref{maxL2rel}.
\end{proof}

\begin{remark}
  \HZ{Under the assumptions of Lemma~\ref{kkh}, we have $k_-<k<k^h<k_+$ for some
  $k_{\pm}\in\pi\mathbb{N}$ and $k_+-k_-=\pi$. Then $k_{\pm}^h=k_{\pm}$, and for $0<c<1$ from
  Lemma~\ref{kkh}, $k>6\pi$, it holds that
  \[
    \begin{aligned}
      \tilde{\sigma}_-^h&=\left|k-\frac{2}{h}\sin\frac{k^h_-h}{2}\right|\in
      \left(k-k_-,k-k_-+\frac{1}{24}k_-^3h^2\right)\subset(k-k_-)\left(1,1+\frac{c}{3}\right),\\
      \tilde{\sigma}_+^h&=\left|k-\frac{2}{h}\sin\frac{k^h_+h}{2}\right|\in
      \left(k_+-k-\frac{1}{24}k_+^3h^2,k_+-k\right)\subset(k_+-k)\left(1-\frac{7^3c}{6^33},1\right).
    \end{aligned}
  \]
  So $\tilde{\sigma}_k^h=\min\{\tilde{\sigma}_{\pm}^h\}$ in Theorem~\ref{relerr} is of the same
  order as $\sigma_k=\min\{|k-k_{\pm}|\}$.}
  
\end{remark}

\section{Visual analysis of dispersion correction schemes}
\label{visual}

We have already seen from \eqref{udiri}, \eqref{uhdiri} and Section~\ref{secdiri} that for the zero
source problem the discretization error is essentially in the discrete wavenumber $k^h$. This
becomes the motivation of dispersion correction. In 1D, $k^h$ is simply a constant depending on $kh$
for a given linear scheme, and the error in $k^h$ can be completely removed; see
e.g. \cite{babuska1997, ernst2013}.

For the classical 3-point centered scheme \eqref{1d3pt}, one possibility adopted in \cite{ernst2013}
is to modify the wavenumber $k$ to be used in the scheme \eqref{1d3pt} and get the new scheme
\begin{equation}\label{kmod}
  (\tilde{k}^2-\frac{2}{h^2})u^h(x_j) + \frac{1}{h^2}u^h(x_{j-1})+\frac{1}{h^2}u^h(x_{j+1})=%
  f(x_j),\quad u^h(x_0)=g_0,\quad u^h(x_N)=g_1.
\end{equation}
When $f=0$, the solution of \eqref{kmod} is given by \eqref{uhdiri} which is copied here,
\[
  u^h(x_j) = \frac{g_1-g_0\cos k^h}{\sin k^h}\sin(k^hx_j)+g_0\cos(k^hx_j),\quad x_j=jh,\;j=0,1,..,N,
\]
but with $k^h:=\frac{2}{h}\arcsin\frac{\tilde{k}h}{2}$ ($\tilde{k}$ replacing $k$). The idea is to
let $k^h=k$ by well choosing $\tilde{k}$. That is, solve
\[
  \frac{2}{h}\arcsin\frac{\tilde{k}h}{2} = k\quad\text{ to find }\quad%
  \tilde{k}=\frac{2}{h}\sin\frac{kh}{2}.
\]
Then, the discretization error of the new scheme vanishes when $f=0$. An error analysis when
$f\neq 0$ is given by \cite{cocquet2024} using the framework ``consistency + stability $\Rightarrow$
convergence'', which shows the maximum error is bounded from above by
$(\frac{1}{k}\|f''\|_{L^\infty} + k^2\|f\|_{L^\infty})h^2$ up to a constant factor.

Another possibility adopted by \cite{babuska1997} is to scale the discrete Laplacian but retain the
original wavenumber $k$. This amounts to multiply the left hand side (only) of \eqref{kmod} with
$\frac{k^2}{\tilde{k}^2}$ and get
\begin{equation}\label{Lmod}
  \left(k^2-\frac{k^2}{\tilde{k}^2}\frac{2}{h^2}\right)u^h(x_j) + %
  \frac{k^2}{\tilde{k}^2}\frac{1}{h^2}u^h(x_{j-1})+ %
  \frac{k^2}{\tilde{k}^2}\frac{1}{h^2}u^h(x_{j+1})=%
  f(x_j),\quad u^h(x_0)=g_0,\quad u^h(x_N)=g_1.
\end{equation}
When $f=0$, the scheme \eqref{Lmod} is equivalent to \eqref{kmod}. The original work of
\cite{babuska1997} is in the finite element framework with the source treated by integral. The idea
is to ``\emph{eliminate the phase lag $k-k^h$ for as many right-hand sides as possible}'', not only
for $f=0$ but also for piecewise constant $f$.  For a different problem than \eqref{1dhelm} with the
right boundary condition being $u'-\I ku=0$, it was shown in \cite{babuska1997} that the
$H^1$-semi-norm of the error is bounded from above by $h\|f\|_{H^1}$ up to a constant factor.

\begin{remark}
  If we multiply $f$ (only) in \eqref{kmod} with $\frac{\tilde{k}^2}{k^2}$, then we obtain a scheme
  equivalent to \eqref{Lmod}.
\end{remark}

{As we have seen, the 1D dispersion correction is determined up to a constant factor. If we
  follow a third approach proposed by \cite{stolk2016} originally for 2D and 3D problems, then we
  would get
\begin{equation}\label{Lfmod}
  \left(k^2\!-\!\frac{k^2}{\tilde{k}^2}\frac{2}{h^2}\right)u^h(x_j)\!+\!%
  \frac{k^2}{\tilde{k}^2}\frac{1}{h^2}u^h(x_{j-1})\!+\!%
  \frac{k^2}{\tilde{k}^2}\frac{1}{h^2}u^h(x_{j+1})\!=\!%
  \frac{kh}{2}\cot\frac{kh}{2}f(x_j),\; u^h(x_0)\!=\!g_0,\; u^h(x_N)\!=\!g_1.
\end{equation}
}

To evaluate the {three} dispersion free schemes \eqref{kmod}, \eqref{Lmod} {and \eqref{Lfmod}}, we
use the Fourier analysis as a visual tool. Since the downsampling and aliasing errors are smaller
(for smooth $f$) than the operator error, we will focus on the latter. The operator symbols of
\eqref{kmod}, \eqref{Lmod} {and \eqref{Lfmod}, that map $u^h$ to $f^h$,} are
\begin{equation}\label{ops}
  \lambda^h_k\!:=\!\tilde{k}^2\!-\!\frac{4}{h^2}\sin^2\frac{\xi h}{2},\;\;%
  \lambda_{\Delta}^h\!:=\!k^2\!-\!\frac{k^2}{\tilde{k}^2}\frac{4}{h^2}\sin^2\frac{\xi h}{2},\;\;%
  {\lambda_{\Delta,f}^h\!:=\!\frac{2}{kh}\tan\frac{kh}{2}\lambda_{\Delta}^h,}\;\;%
  \xi\!\in\!\{1,\ldots,N\!-\!1\}\pi\text{ for }h\!=\!\frac{1}{N}.
\end{equation}
Recall that the continuous symbol is $\lambda=k^2-\xi^2$, and the symbol errors are defined as
\begin{equation}\label{symerr}
  \psi:=\left|\frac{\xi}{\lambda}-\frac{\xi}{\lambda^h}\right|,\quad%
  \psi_e:=\left|\frac{1}{\lambda}-\frac{1}{\lambda^h}\right|,\quad%
  \psi_{rel}:=\left|\frac{\lambda-\lambda^h}{\xi^2\lambda^h}\right|,\quad%
  \xi\in\{1,\ldots,N-1\}\pi\text{ for }h=\frac{1}{N},
\end{equation}
where $\lambda^h$ is the discrete symbol. For the classical centered scheme,
$\lambda^h=k^2-\frac{4}{h^2}\sin^2\frac{\xi h}{2}$. For the scheme \eqref{kmod}, \eqref{Lmod} or
{\eqref{Lfmod}}, it is $\lambda^h_k$, $\lambda_{\Delta}^h$ {or $\lambda_{\Delta,f}^h$} in
\eqref{ops}. From \eqref{E1h}, \eqref{E1hf1}, \eqref{E1hrelL2max} and \eqref{E1hrelH1max}, it is
essentially the symbol error $\psi$, $\psi_e$ or $\psi_{rel}$ that determines the convergence in
terms of absolute/relative error in the $L^2$- or $H^1$-semi-norm.

Figure~\ref{href} shows the symbol errors $\psi_e$ \eqref{symerr} for the {four} schemes under
$h$-refinement. We see that the dispersion free schemes have lower maxima of $\psi_e$ than the
classical one, at the price of larger $\psi_e$ at smaller $\xi$ (roughly $<\frac{5}{6}k$ or
$<\frac{2}{3}k$). We also see that the scheme \eqref{Lfmod} modifying both the discrete Laplacian
and the source performs the best. The order of $\max\psi_e$ is $h^2$, which corroborates
Lemma~\ref{lemevan}.

\begin{figure}
  \centering
  \includegraphics[scale=.3,trim=5 0 0 0,clip]{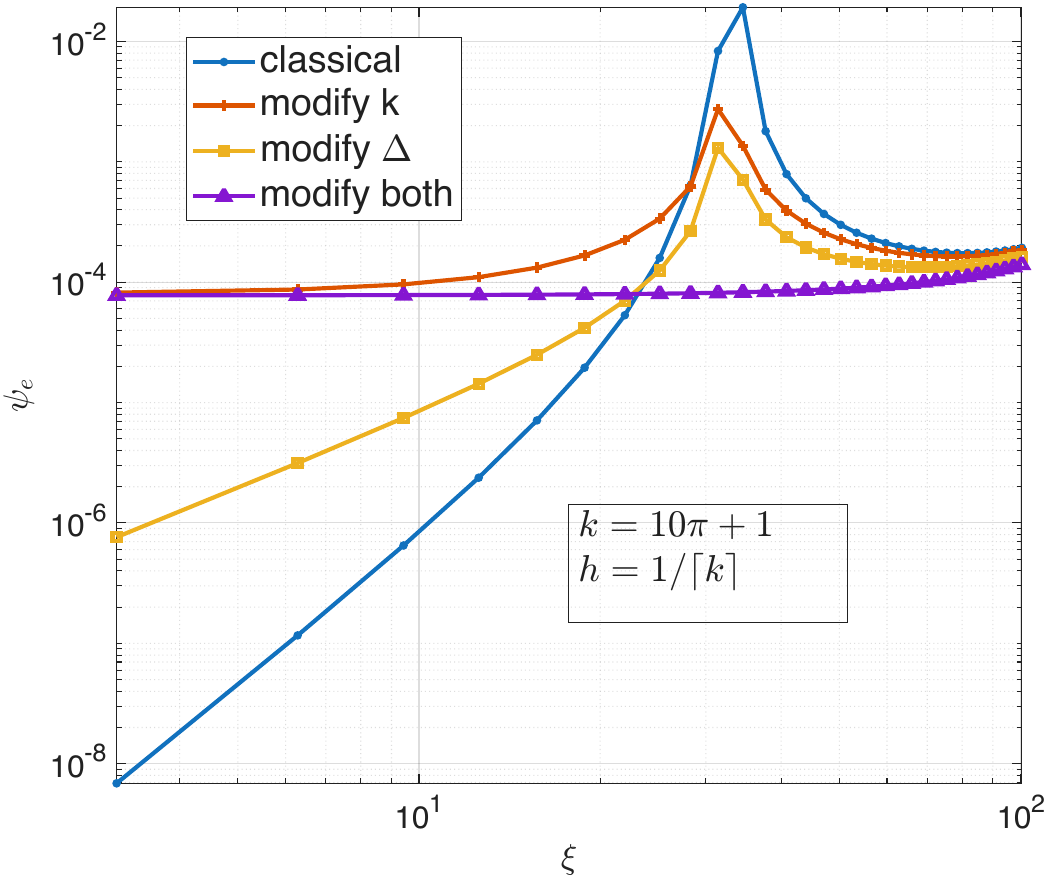}
  \includegraphics[scale=.3,trim=5 0 0 0,clip]{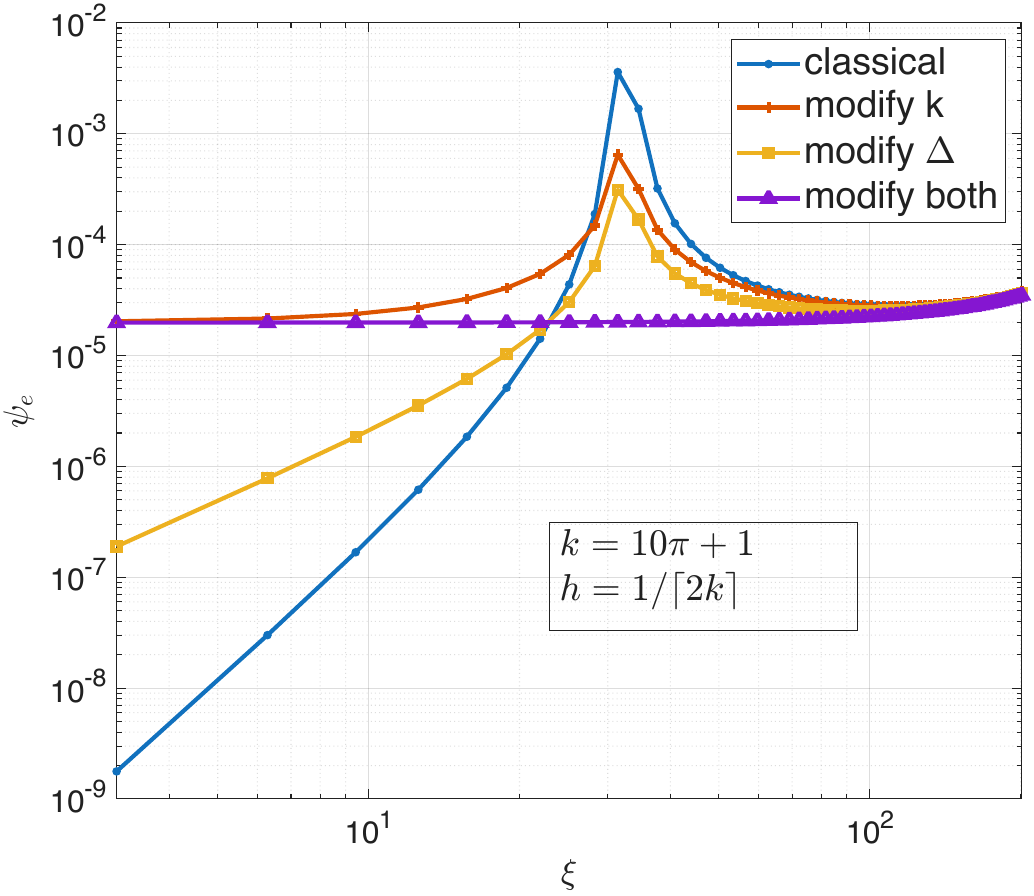}
  \includegraphics[scale=.3,trim=6 0 0 0,clip]{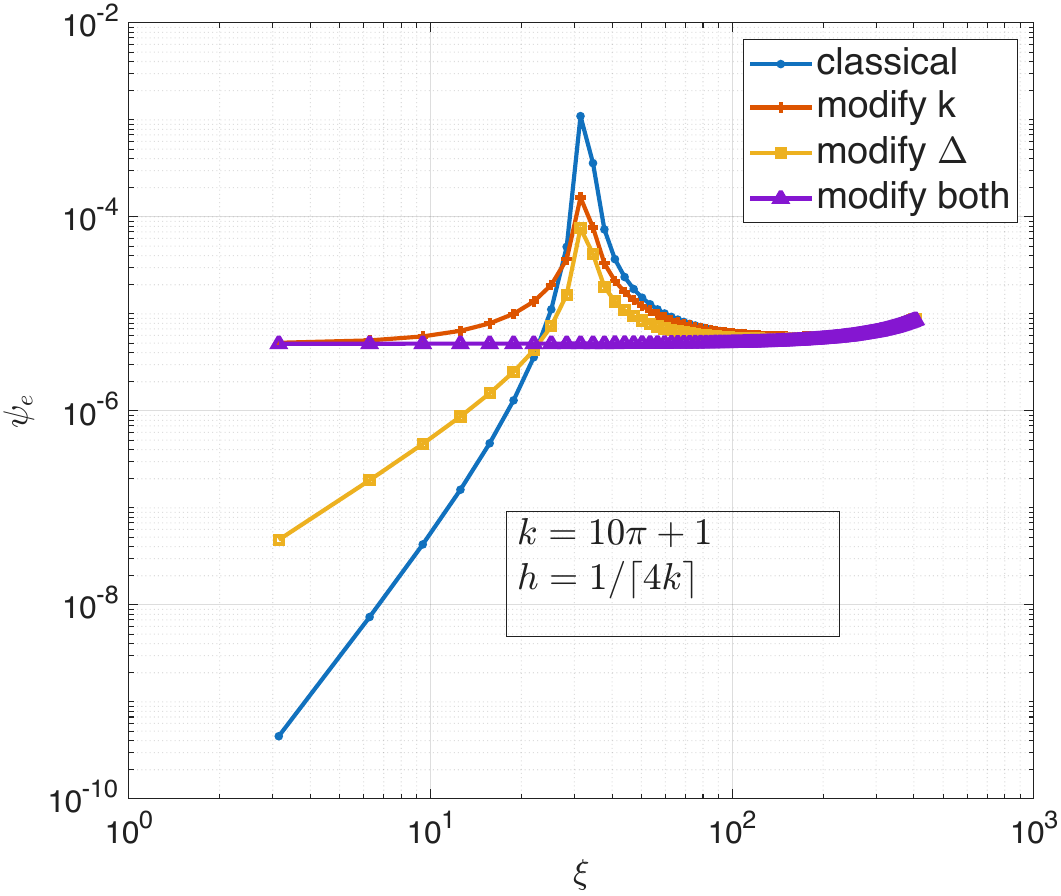}\\%
  \includegraphics[scale=.3,trim=5 0 0 0,clip]{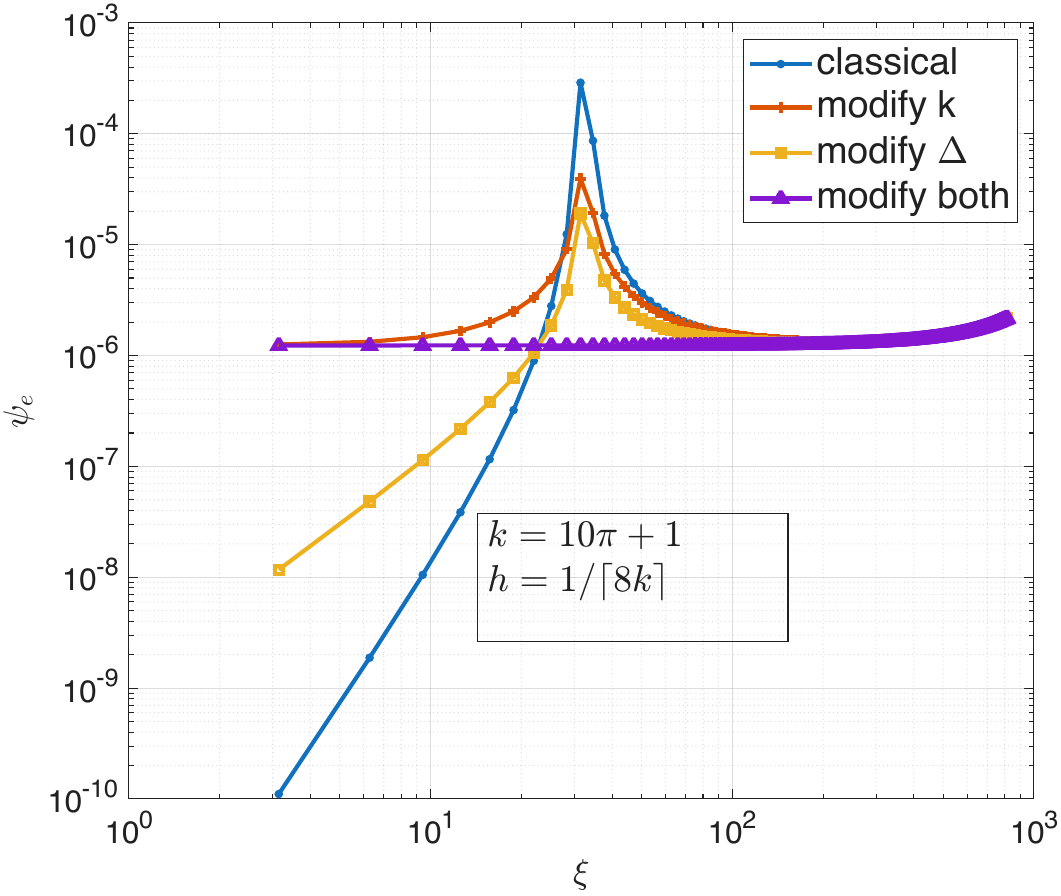}
  \includegraphics[scale=.3,trim=5 0 0 0,clip]{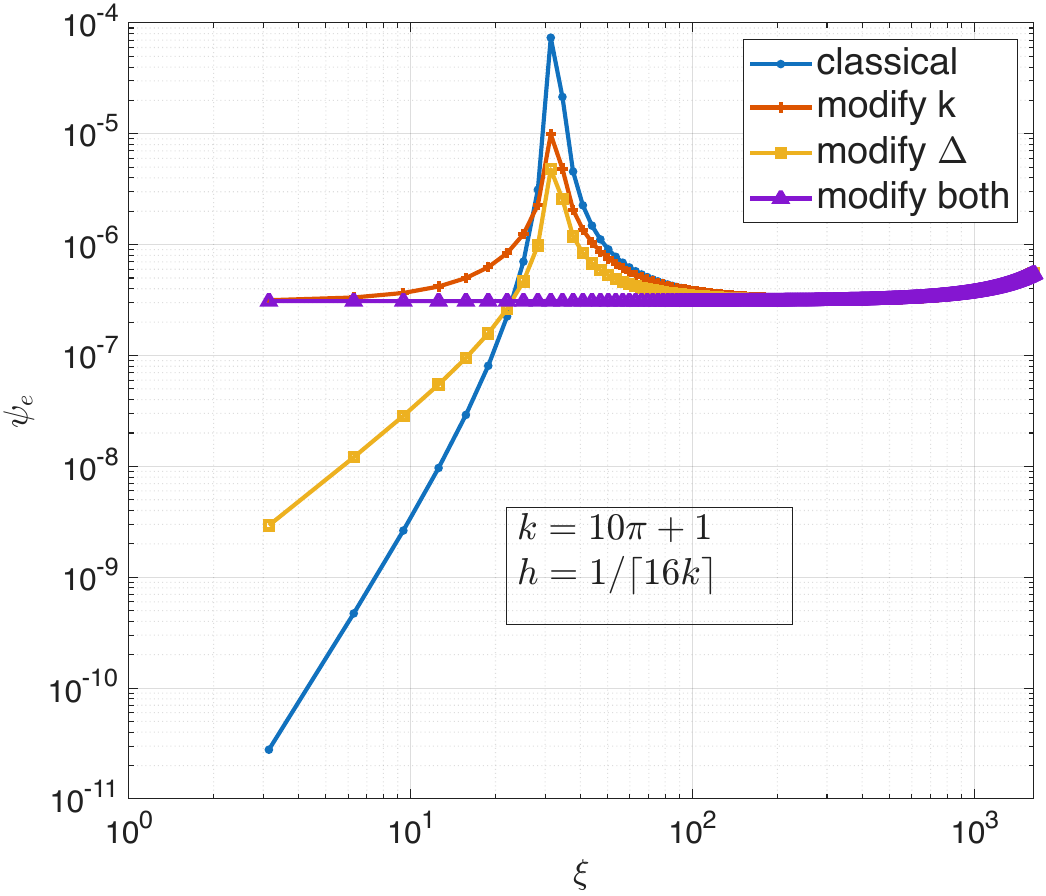}
  \includegraphics[scale=.3,trim=49 180 75 180,clip]{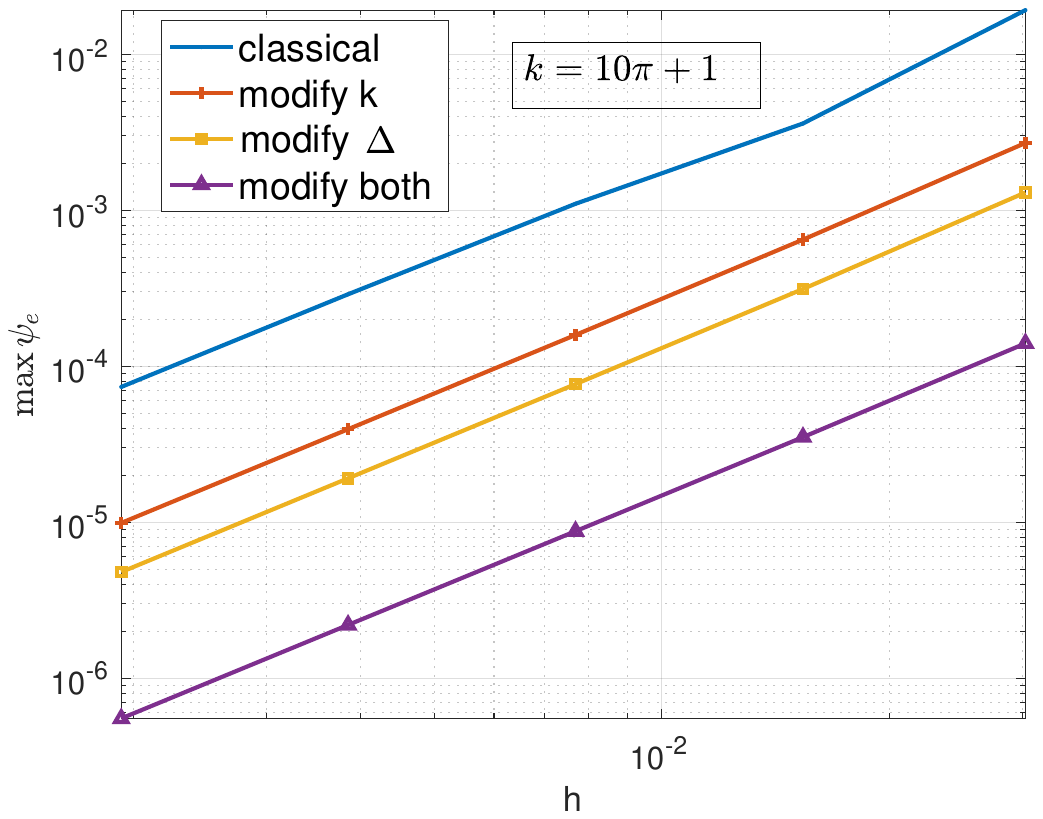}
  \caption{\small Symbol errors $\psi_e$ \eqref{symerr} for the classical \eqref{1d3pt} \&
    dispersion free \eqref{kmod}, \eqref{Lmod}, {\eqref{Lfmod}}: $h$-refinement}\label{href}
\end{figure}

Next, the influence of the wavenumber $k$ is demonstrated, see Figure~\ref{khref} in which $k$ is
quadrupled while $kh$ is halved. This clearly shows that the classical scheme can not converge in
this setting with $k^3h^2$ fixed. This corresponds to the order $k^3h^2$ proved in
Lemma~\ref{lem3ptmax}. Figure~\ref{khref} shows also that the schemes \eqref{kmod} and \eqref{Lmod}
have the symbol error of order $k^2h^2$ which translates to the same order of $|u-u^h|_1$ by
\eqref{E1h}{, while the scheme \eqref{Lfmod} is even better and of} \footnote{{This is because the
  order in $h$ for fixed $k$ must be 2, and it is observed that the scaling $k\to 4k$, $h\to h/8$
  makes the error scaled by $4^3/8^3=2^{-3}$, so that $4^{\alpha}/8^2=2^{-3}$ gives the right
  exponent $\alpha=\frac{3}{2}$ for $k$.\label{ft2}}} {order $k^{3/2}h^2$}.

\begin{figure}
  \centering
  \includegraphics[scale=.3,trim=5 0 2 0,clip]{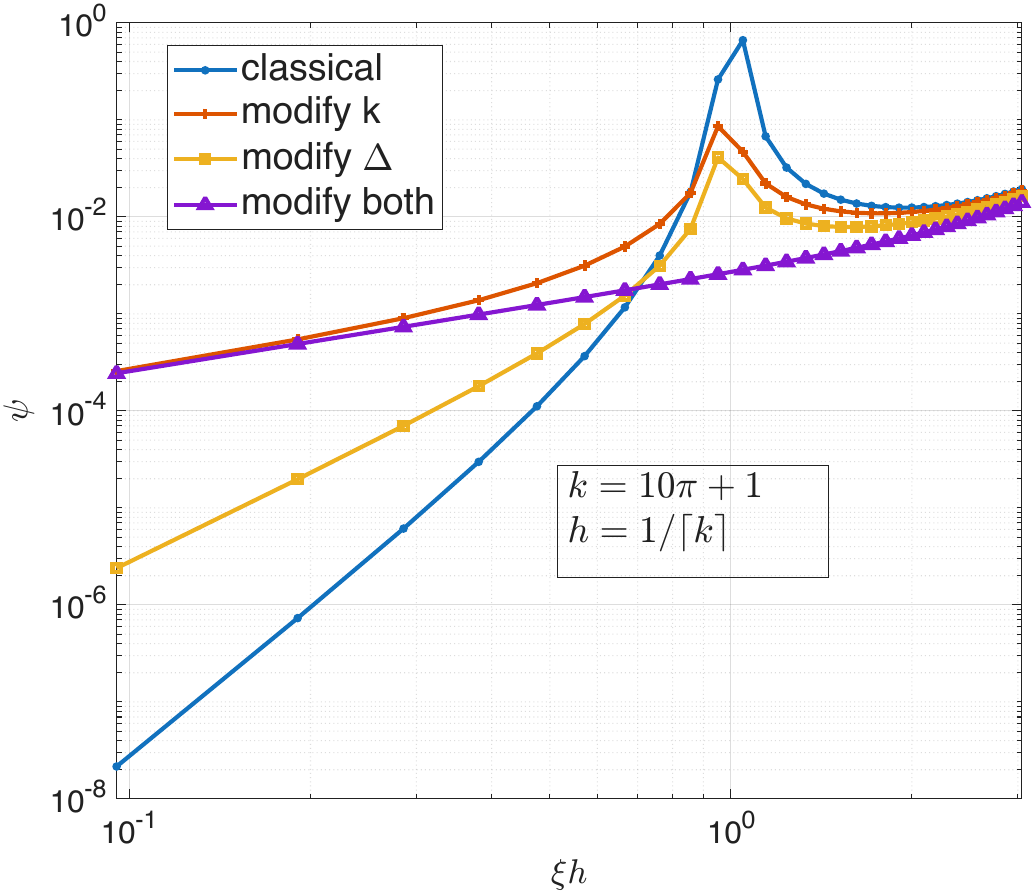}
  \includegraphics[scale=.3,trim=5 0 2 0,clip]{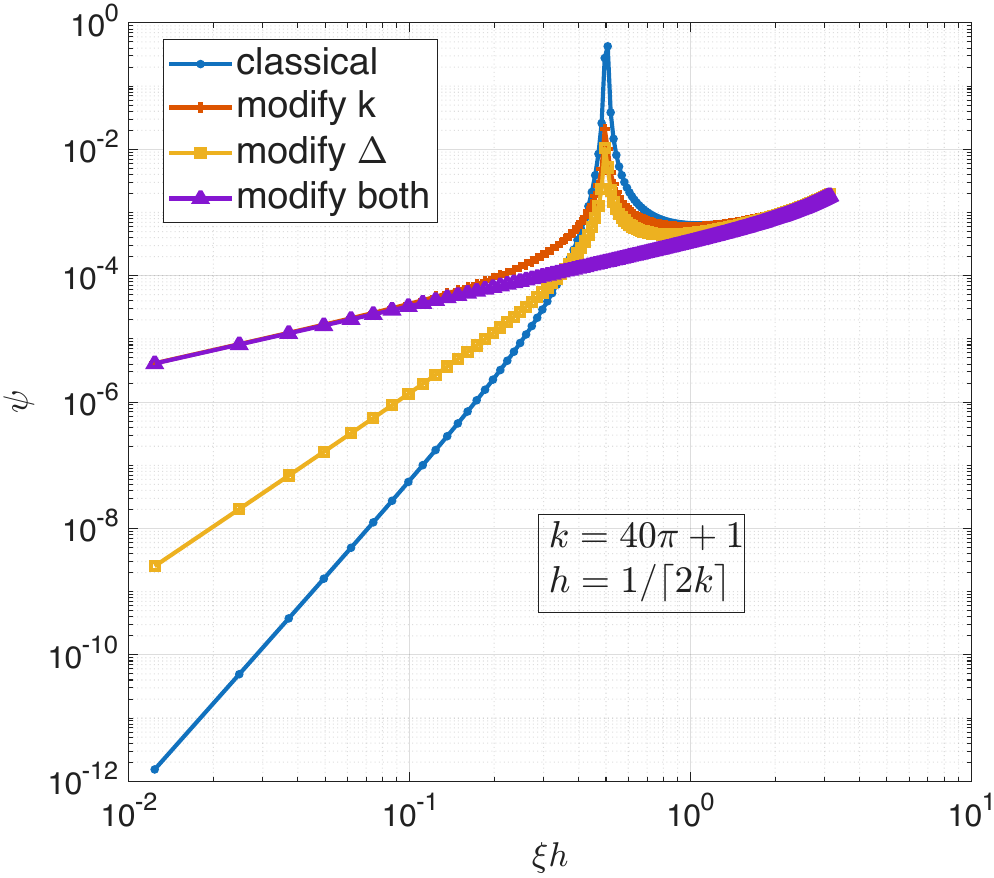}
  \includegraphics[scale=.3,trim=5 0 2 0,clip]{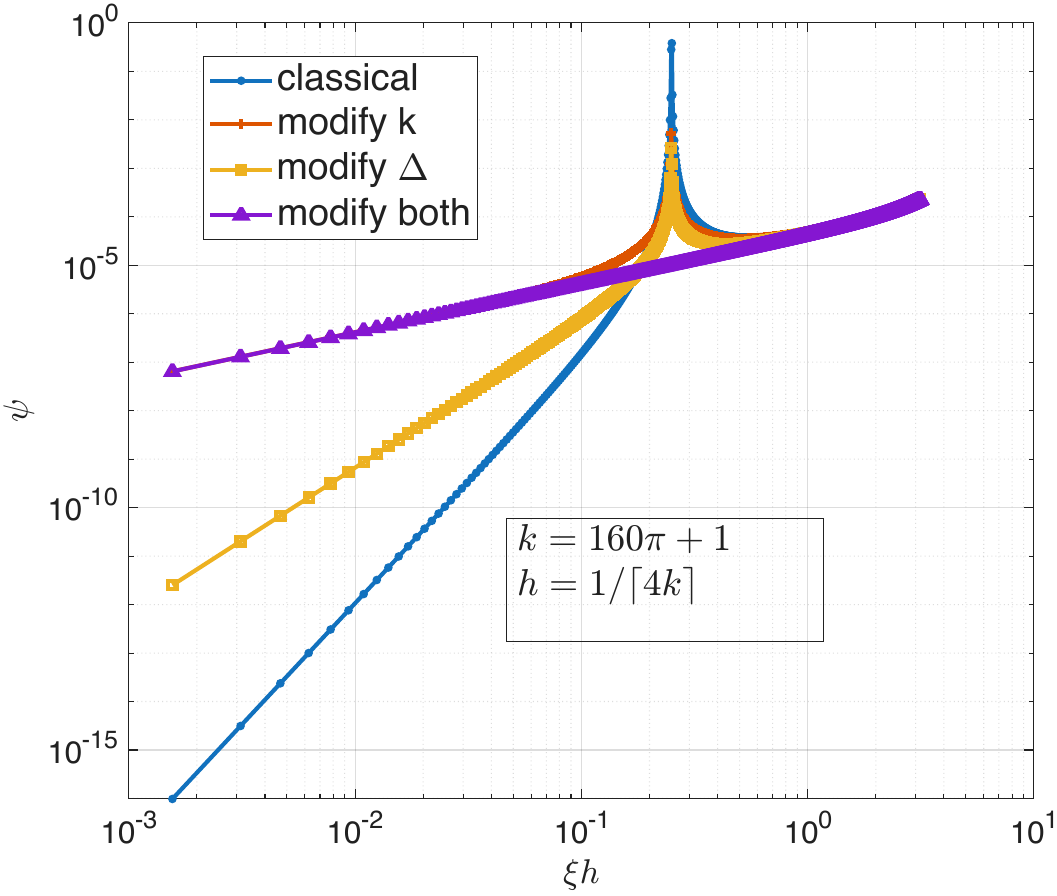}
  \includegraphics[scale=.3,trim=5 0 50 0,clip]{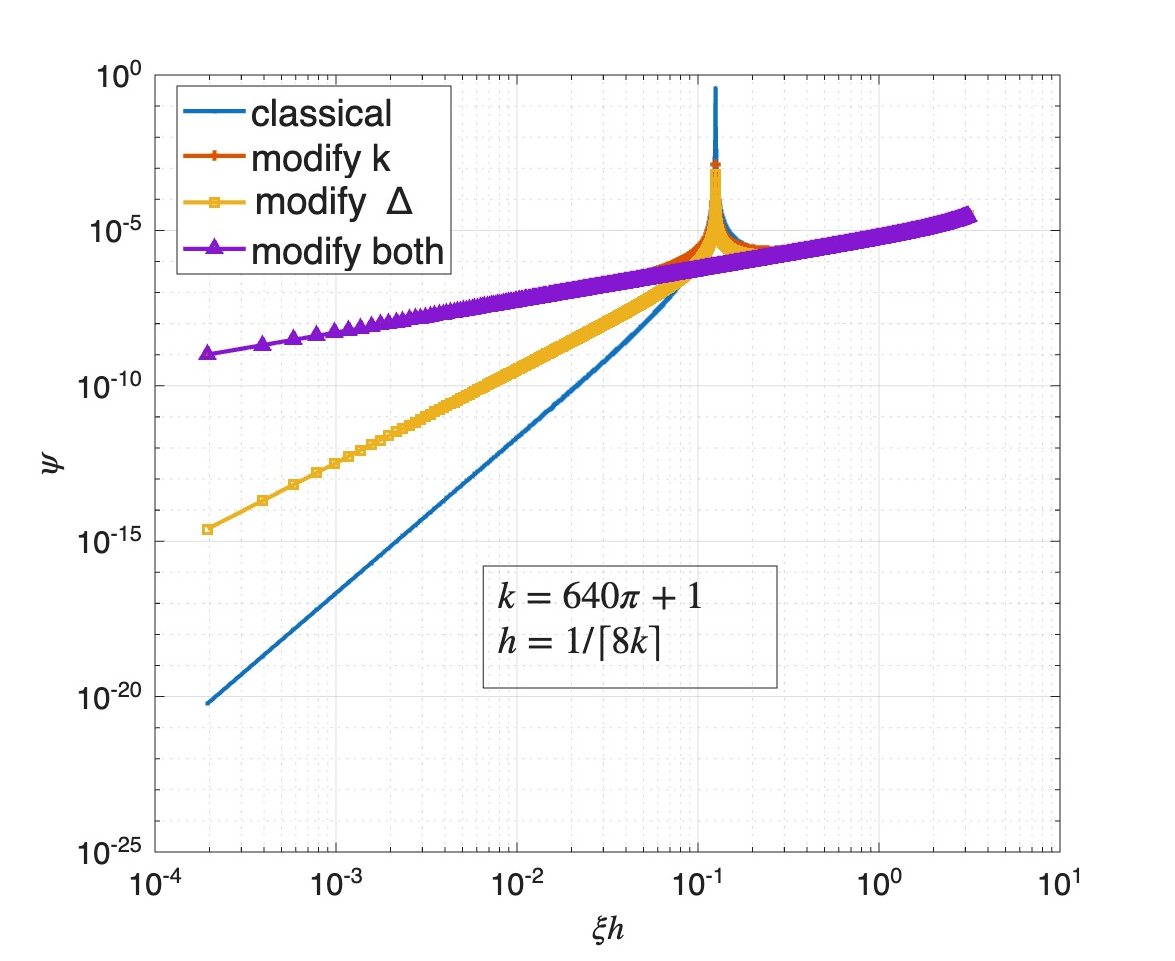}
  \includegraphics[scale=.3,trim=5 0 2 0,clip]{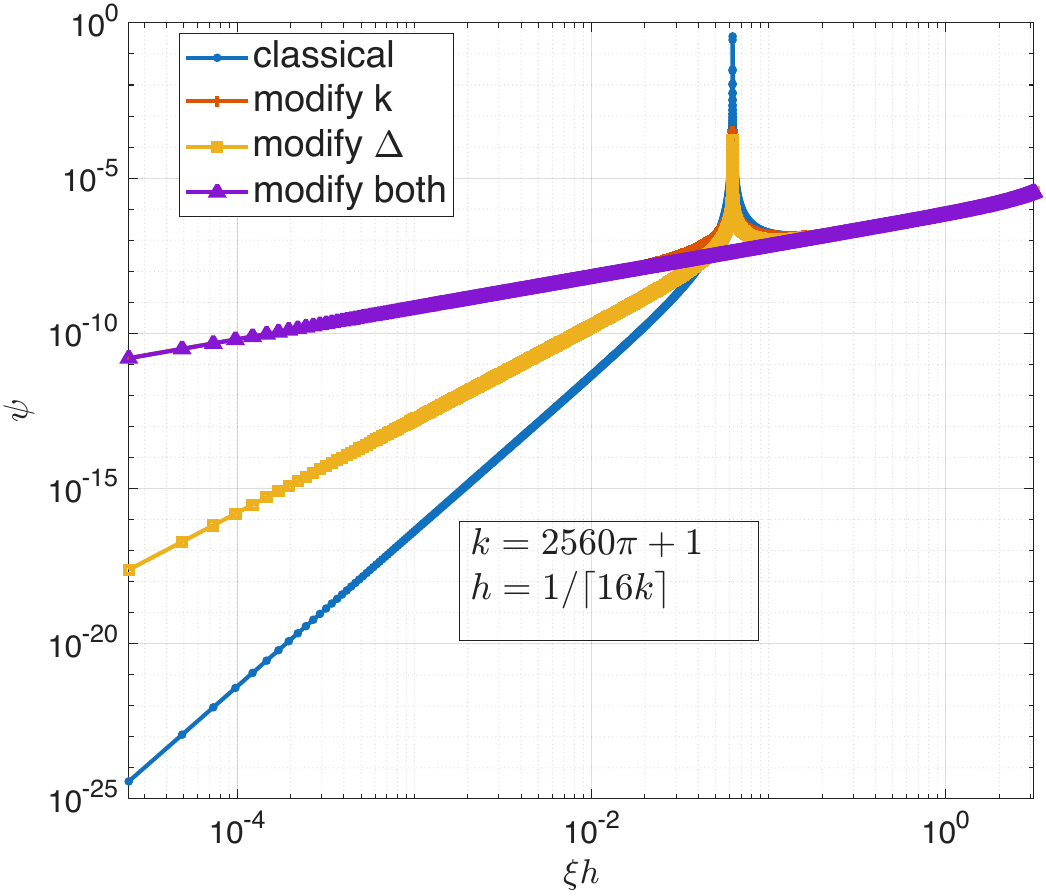}
  \includegraphics[scale=.3,trim=49 180 75 180,clip]{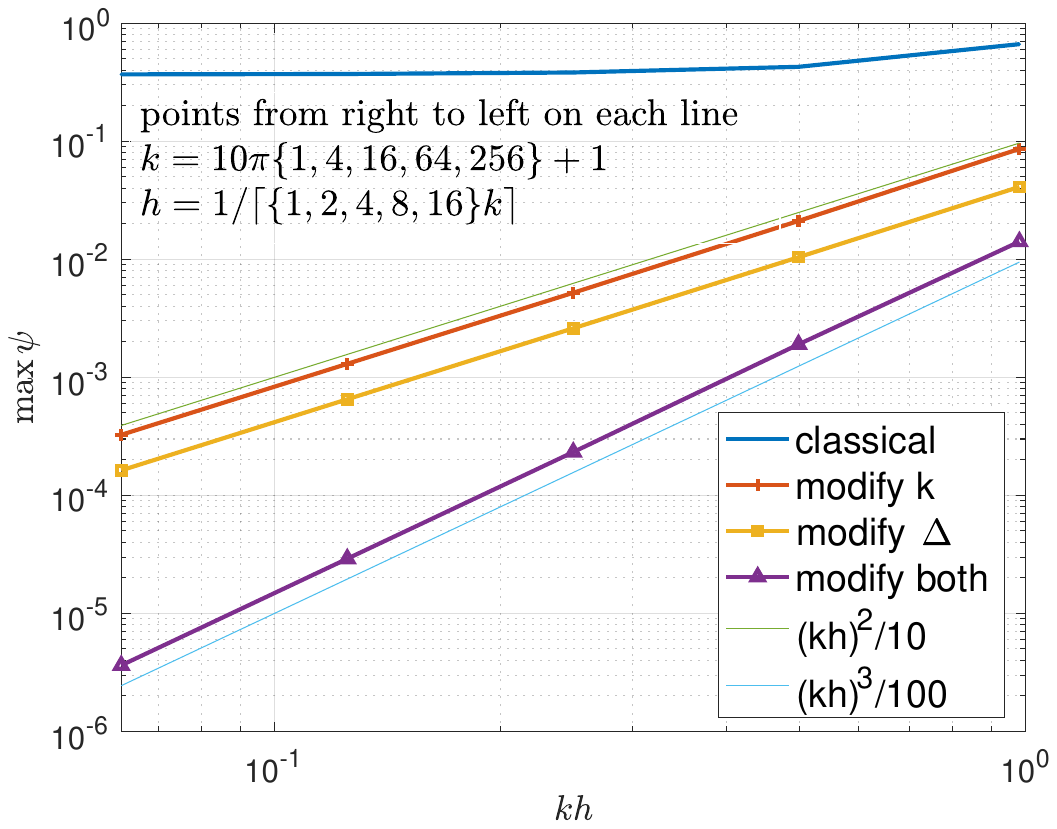}
  \caption{\small Symbol errors $\psi$ \eqref{symerr} for the classical \eqref{1d3pt} \& dispersion
    free \eqref{kmod}, \eqref{Lmod}, {\eqref{Lfmod}}: $kh$-refinement}\label{khref}
\end{figure}

\HZ{Next,} the relative symbol error $\psi_{rel}$ is visualized in Figure~\ref{khfixed} where $kh$
is fixed while $k$ is doubled. According to \eqref{E1hrelH1max}, $\max\psi_{rel}$ has to be
multiplied with $|u_{\mathrm{low}}|_3/|u|_1$ to give the upper bound of $|u-u^h|_1/|u|_1$, and when
the upper bound is attained, $|u_{\mathrm{low}}|_3/|u|_1$ may contribute some powers of $k$ (see
Theorem~\ref{relerr}). For the classical scheme, we see from Figure~\ref{khfixed} that $\psi_{rel}$
is of order $kh^2$ which corroborates Lemma~\ref{maxL2rel}. For the dispersion free schemes
\eqref{kmod} {and \eqref{Lfmod}}, the relative symbol error $\psi_{rel}$ does not decrease with $h$
(while $kh$ is fixed) at the first frequency $\xi=\pi$, which slows down also the convergence at the
other low frequencies. The dispersion free scheme \eqref{Lmod} modifying the discrete Laplacian has
an almost equidistributed $\psi_{rel}$ over $\xi$ as seen in the figure, and $\max\psi_{rel}$ is
attained at $\xi=(N-1)\pi$ (order $k$) and is of order $h^2$. So the relative error
$|u-u^h|_1/|u|_1$ of the scheme \eqref{Lmod} is of order $k^2h^2$.

\begin{figure}
  \centering
  \includegraphics[scale=.3,trim=5 0 0 0,clip]{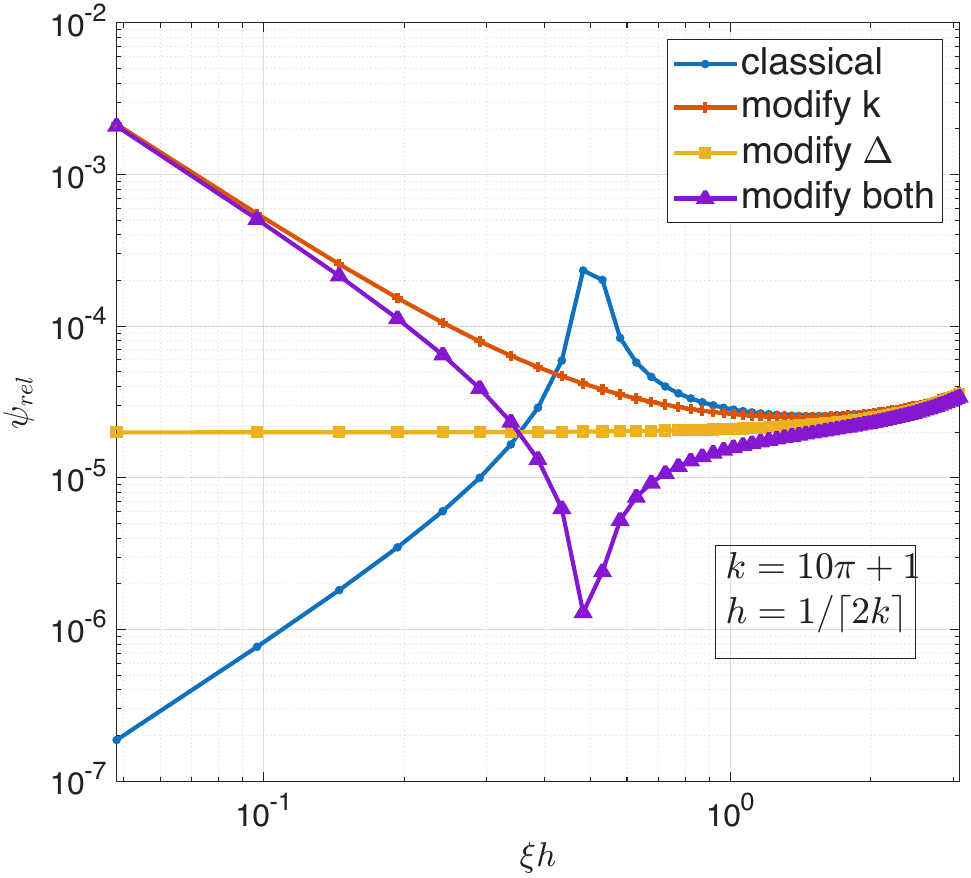}
  \includegraphics[scale=.3,trim=5 0 0 0,clip]{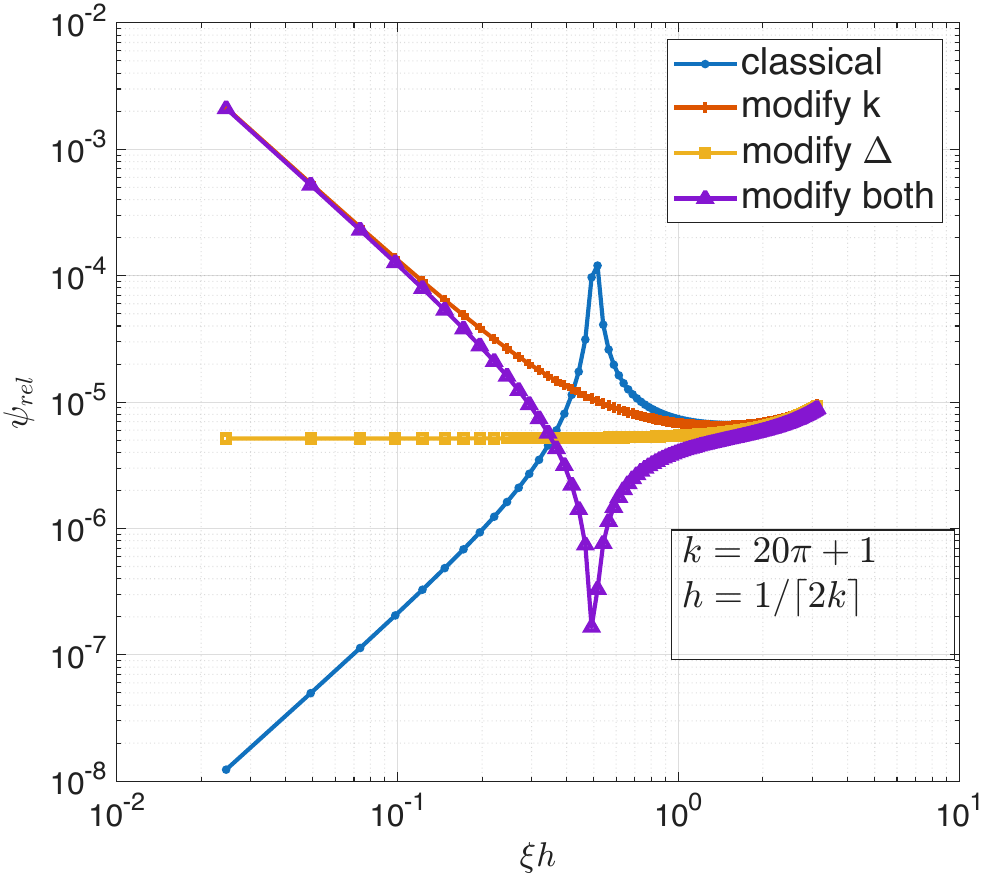}
  \includegraphics[scale=.3,trim=5 0 0 0,clip]{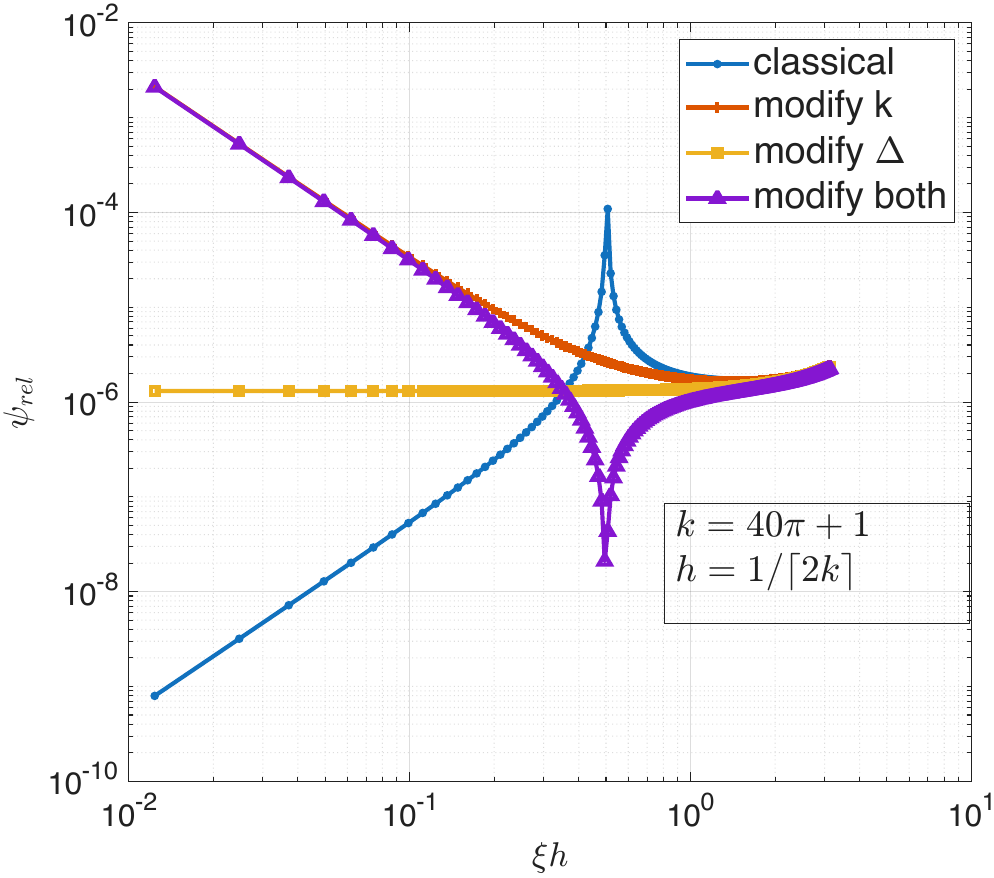}
  \includegraphics[scale=.3,trim=5 0 0 0,clip]{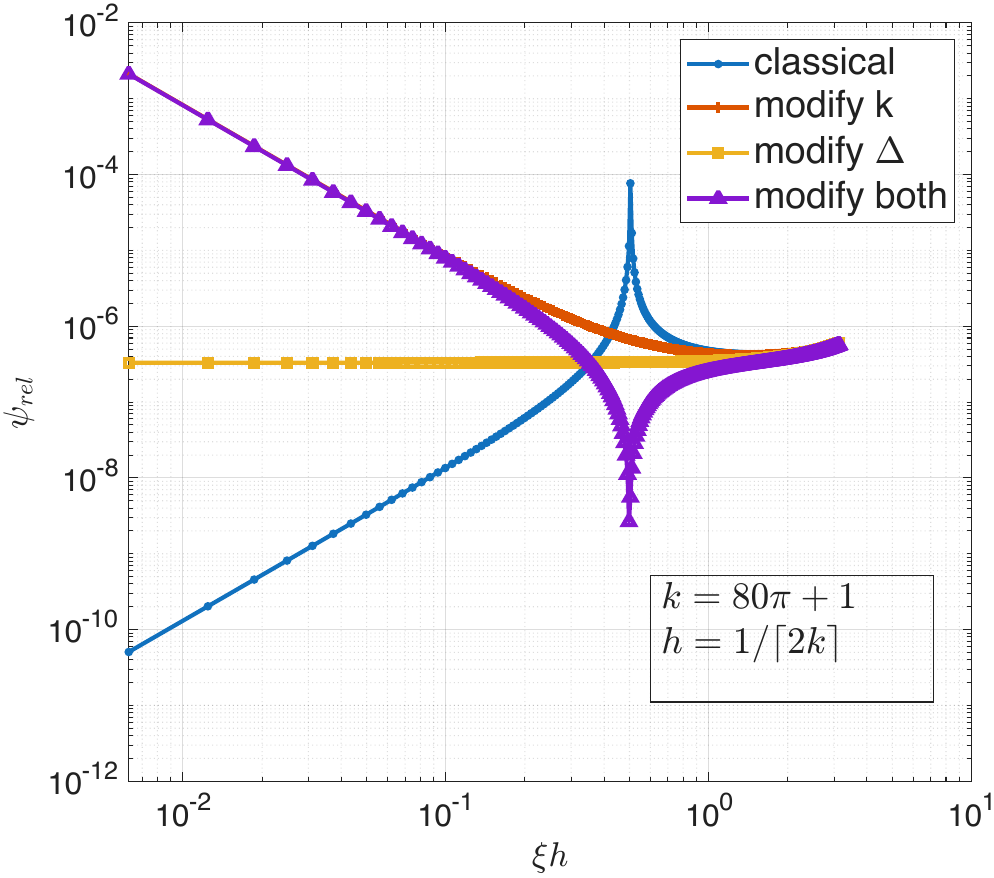}
  \includegraphics[scale=.3,trim=5 0 0 0,clip]{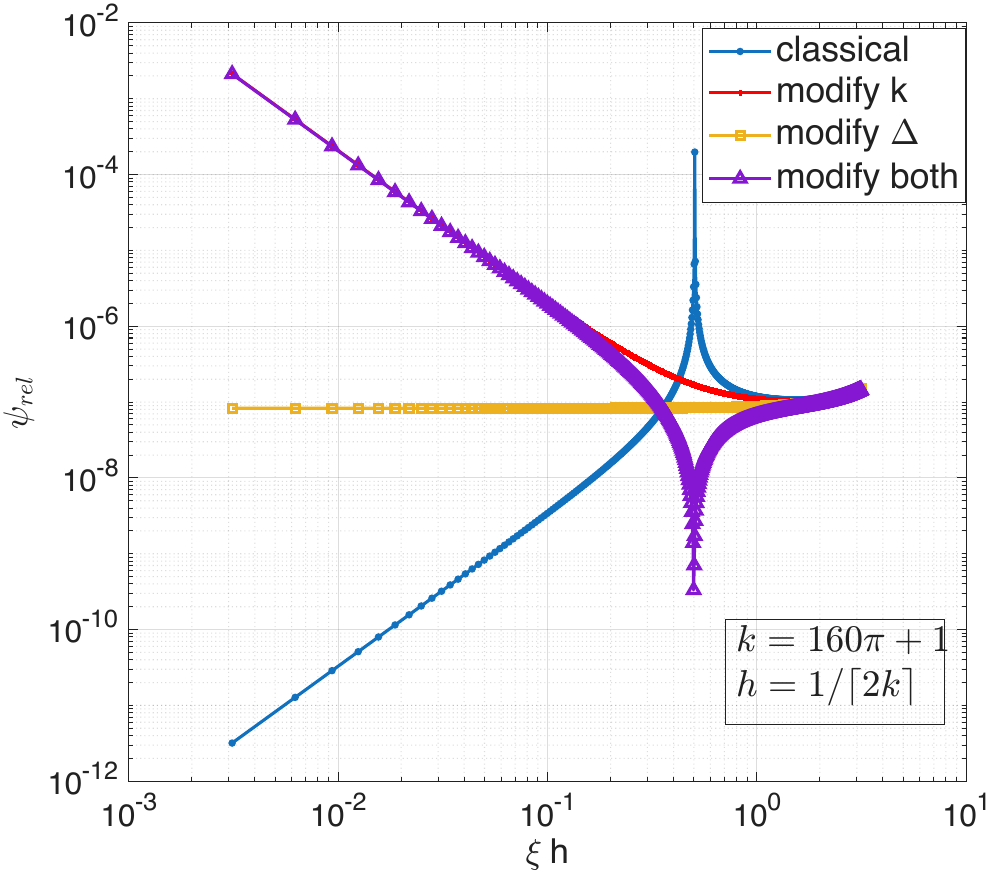}
  \includegraphics[scale=.3,trim=45 180 70 180,clip]{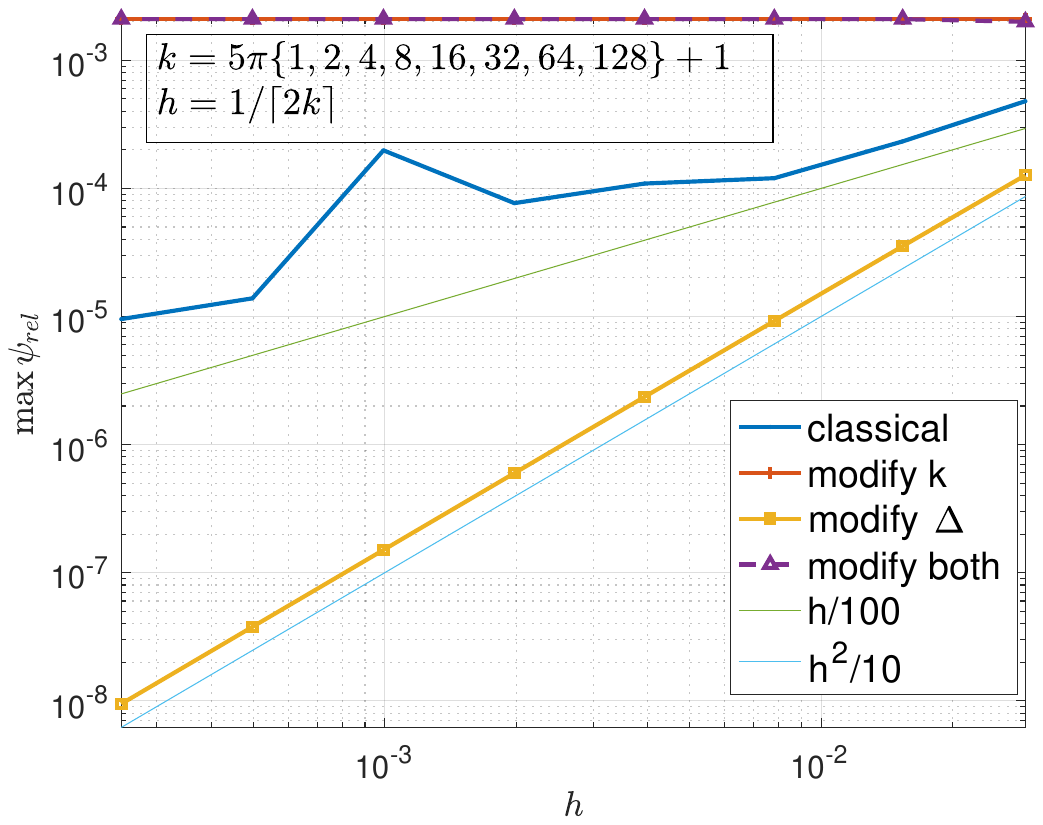}
  \caption{Symbol errors $\psi_{rel}$ \eqref{symerr} for the classical \eqref{1d3pt} \& dispersion
    free \eqref{kmod}, \eqref{Lmod}, {\eqref{Lfmod}}: $kh$ fixed}\label{khfixed}
\end{figure}

\HZ{Finally, the scaling with $\sigma_k$ for the relative symbol error $\psi_{rel}$ divided by
  $k^3h^2$ is demonstrated in Figure~\ref{relsigmak} where $k+\sigma_k=21\pi$, $\sigma_k$ is halved,
  and $h=1/\lceil 1/\sqrt{\sigma_k/k^3}\rceil$. \MG{We see} that the maximum relative symbol error of
  the classical scheme is inversely proportional to $\sigma_k$, while all the dispersion free
  schemes are independent of $\sigma_k$. In particular, for the third dispersion free scheme from
  \cite{stolk2016}, the symbol value near $k$ even decreases linearly with $\sigma_k$.}

\begin{figure}
  \centering
  \includegraphics[scale=.3,trim=0 0 0 0,clip]{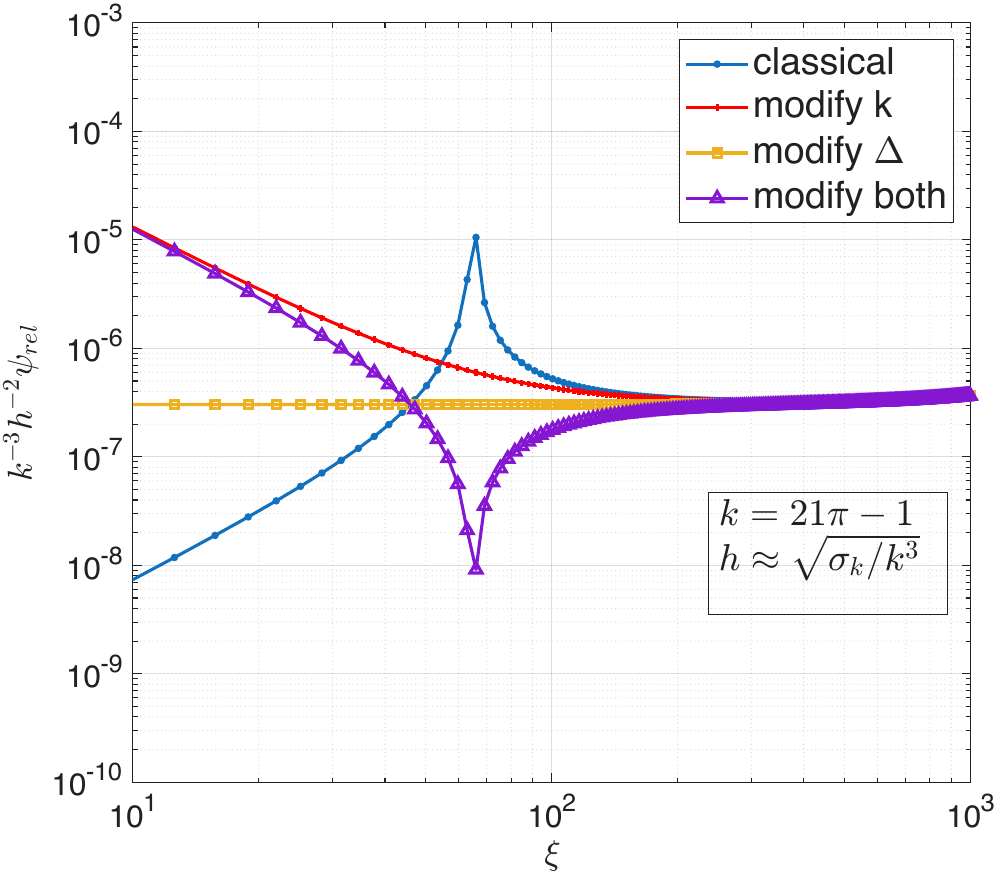}
  \includegraphics[scale=.3,trim=0 0 0 0,clip]{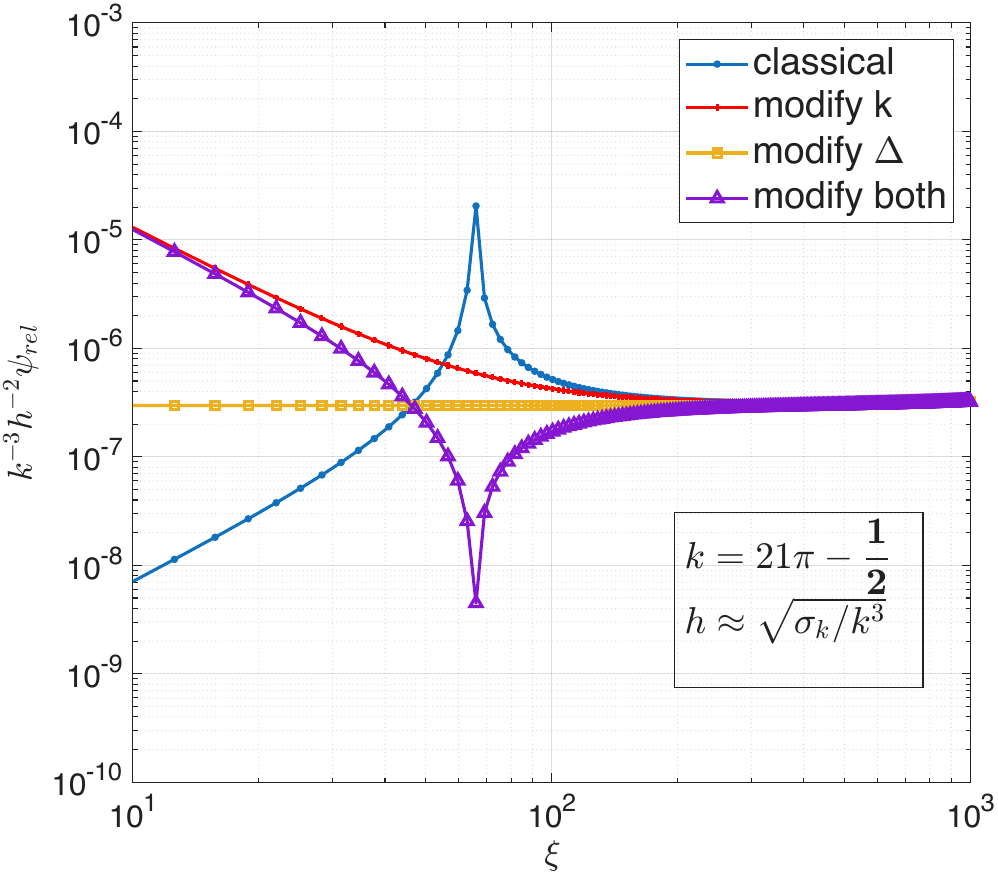}
  \includegraphics[scale=.3,trim=0 0 0 0,clip]{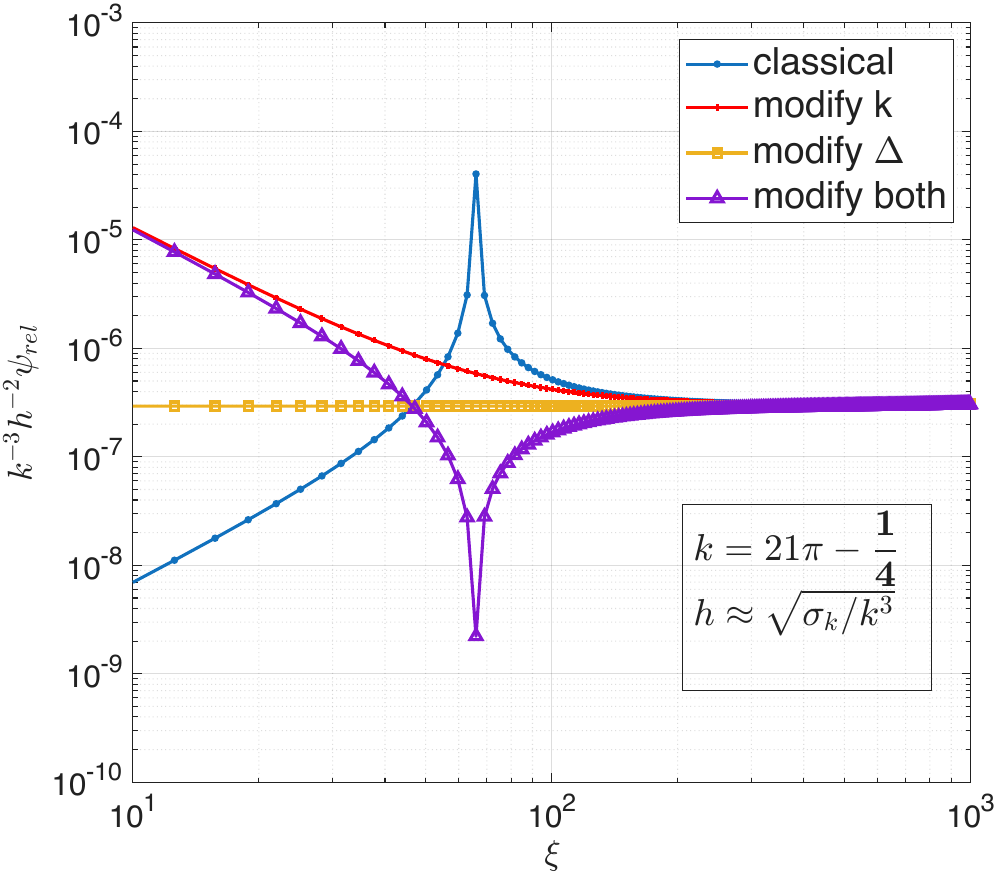}
  \includegraphics[scale=.3,trim=0 0 0 0,clip]{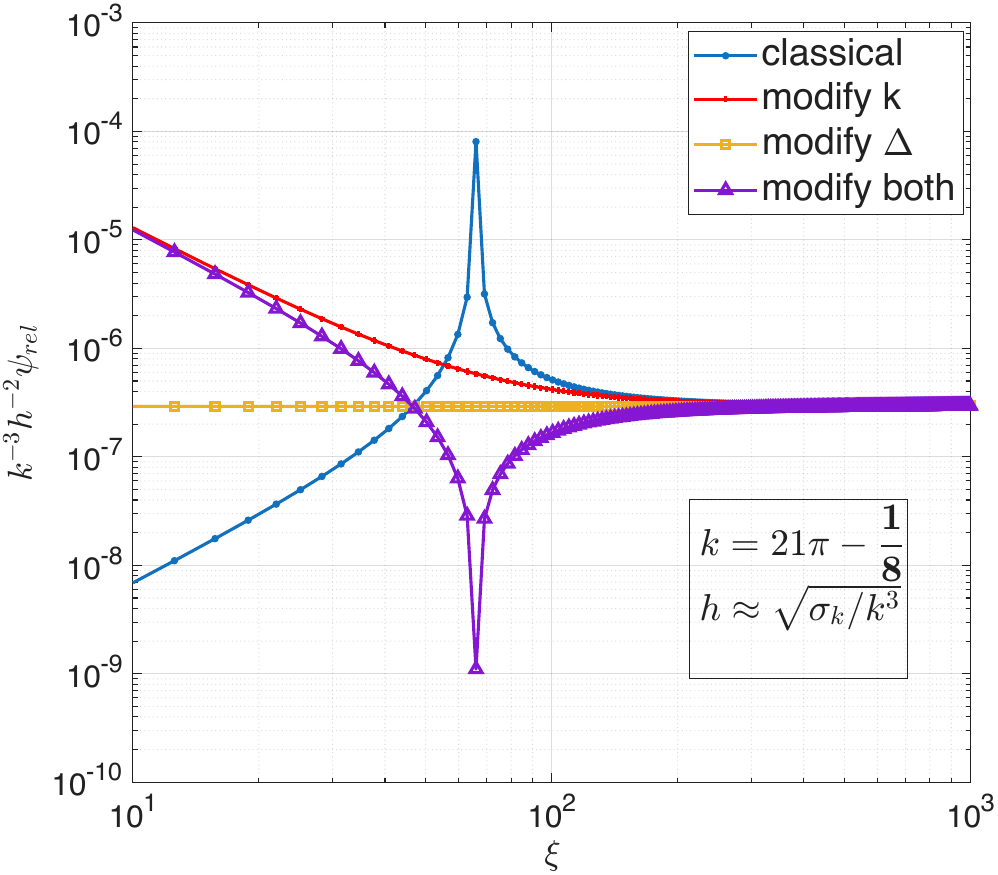}
  \includegraphics[scale=.3,trim=0 0 0 0,clip]{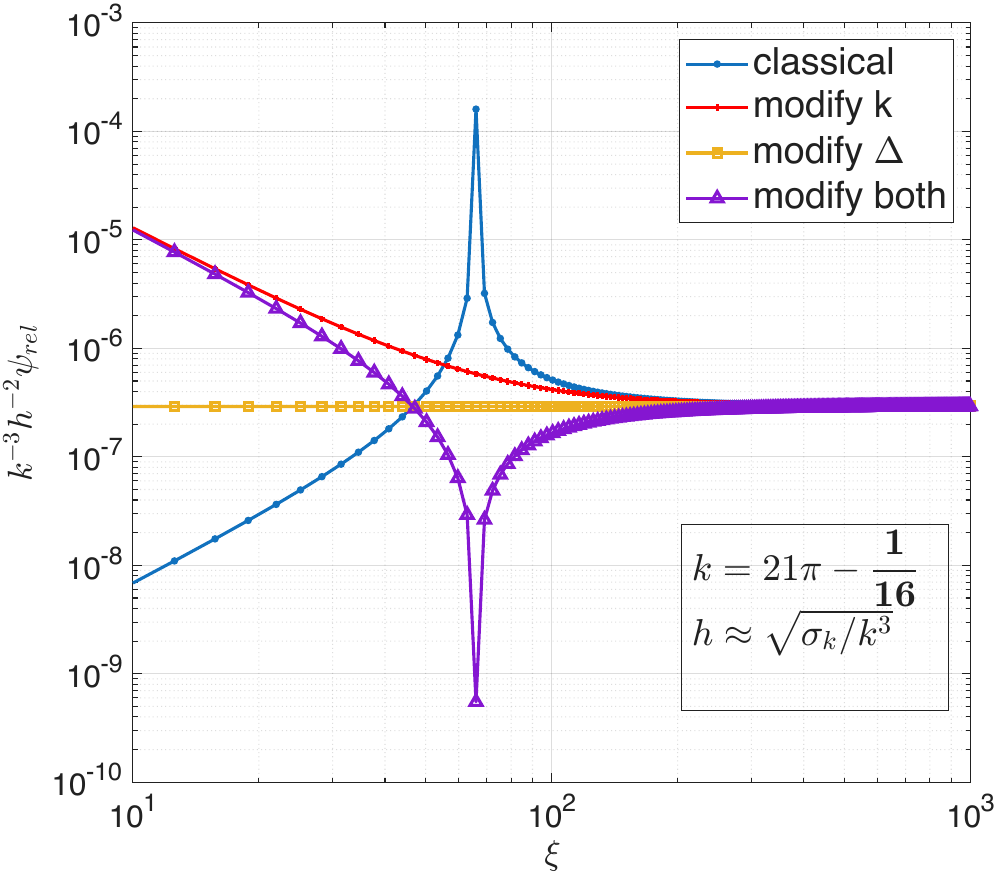}
  \includegraphics[scale=.3,trim=0 0 0 0,clip]{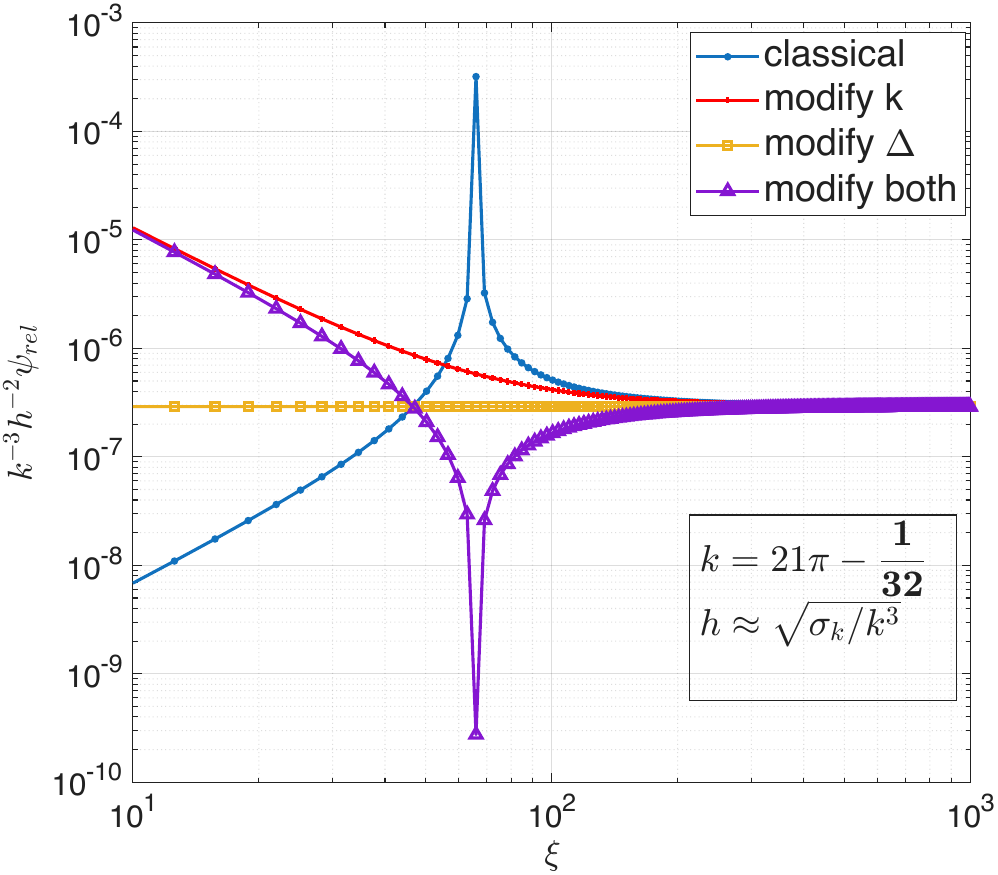}
  \caption{Divided symbol errors $\frac{\psi_{rel}}{k^3h^2}$ \eqref{symerr} for the classical
    \eqref{1d3pt} \& dispersion free \eqref{kmod}, \eqref{Lmod}, {\eqref{Lfmod}}}\label{relsigmak}
\end{figure}

\section{Numerical Experiments}\label{numer}

\HZ{We study numerically the scaling with $k$, $h$ and $\sigma_k$. We take $g_0=g_1=0$,
  $f(x)=\sin(\xi x)$ for some $\xi\in\pi\{1,\ldots,\frac{1}{h}-1\}$. So $\xi$ depends on the finite
  difference scheme and the error being considered.}

\HZ{First, we fix $\sigma_k=1$, double $k-\sigma_k$, choose $\xi$ to maximize the absolute error,
  and let $h=1/\lceil k^{3/2}\rceil$. The absolute errors are shown in Figure~\ref{numabs}, which
  confirms the results from the visual analysis in section~\ref{visual}. One needs to be careful to
  read the scaling shown in the figure. For example, the $L^2$ error of \eqref{Lfmod} in
  Figure~\ref{numabs} appears to scale as $(kh)^6$ when $k$ doubles and $h\asymp k^{-3/2}$, but
  since the actual order should be in the form $k^{\alpha}h^2$, we find
  $(2\cdot2^{-3/2})^6=2^{\alpha}\cdot (2^{-3/2})^2$ gives $\alpha=0$. Then in
  Figure~\ref{helm1dsigmak}, we study the dependence on $\sigma_k$, which corroborates
  Theorem~\ref{abserr} and the results from the visual analysis in section~\ref{visual}.}

  \begin{figure}
  \centering
  \includegraphics[scale=.5,trim=0 0 0 0,clip]{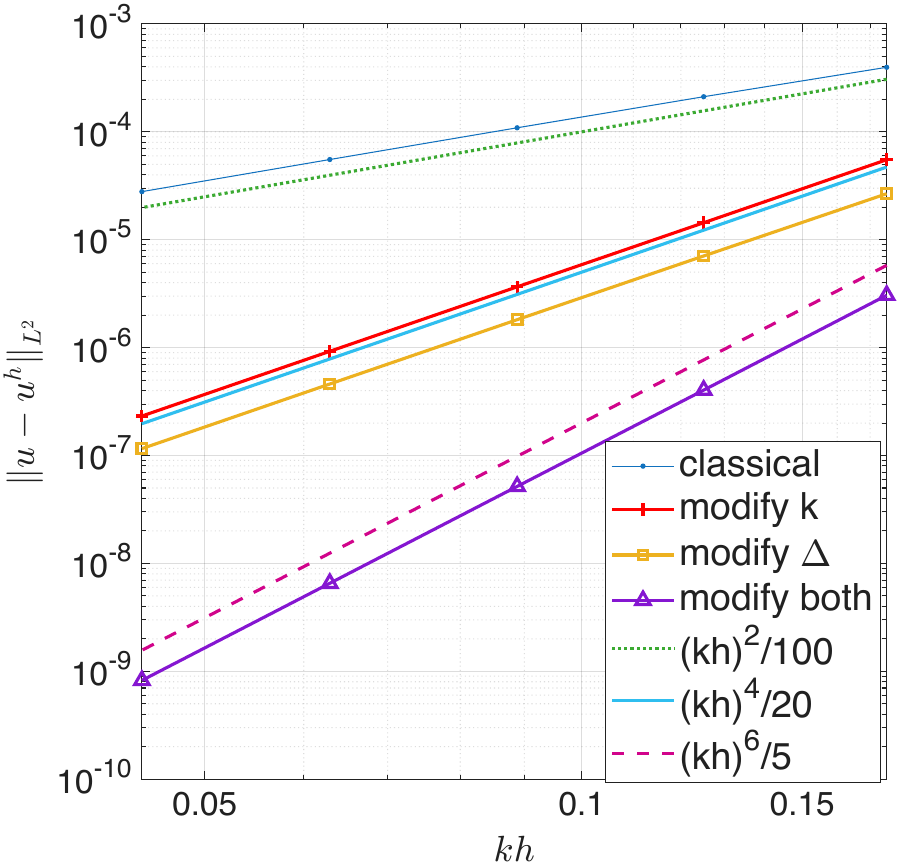}\qquad
  \includegraphics[scale=.5,trim=0 0 0 0,clip]{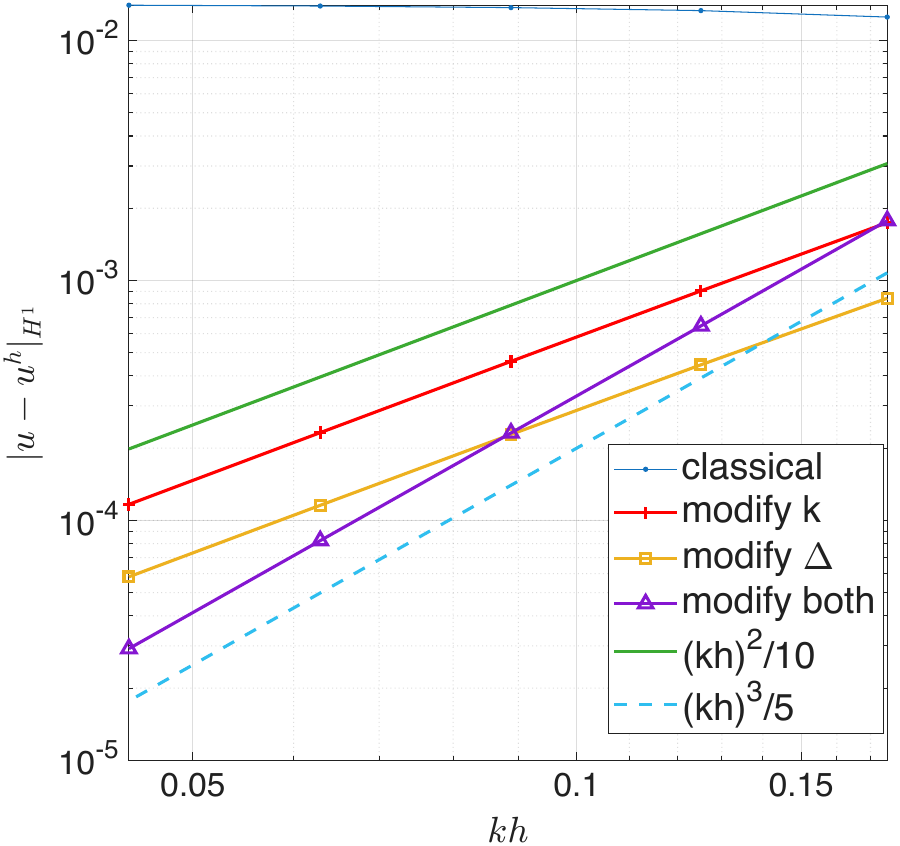}
  \caption{Numerical absolute errors with $g_0=g_1=0$, $f=\sin(\xi x)$, $k=10\pi\cdot 2^{0:4}+1$,
    $h=1/\lceil k^{3/2}\rceil$, and $\xi=k-1$ for the classical \eqref{1d3pt} \& dispersion free
    \eqref{kmod}, \eqref{Lmod}, but $\xi=\pi(\frac{1}{h}-1)$ for \eqref{Lfmod}. In this setting, the
    scalings shown are only apparent and must be reinterpreted in the form $k^{\alpha}h^2$ because
    all the schemes are second order in $h$ for fixed $k$. For example, comparing
    $(2k\cdot 2^{-3/2}h)^3=(2\cdot2^{-3/2})^3(kh)^3$ and
    $(2k)^{\alpha}(2^{-3/2}h)^2=2^{\alpha-3}k^{\alpha}h^2$, we obtain $-3/2=\alpha-3$ or
    $\alpha=3/2$. See also the footnote~\ref{ft2} for Figure~\ref{khref}.}\label{numabs}
\end{figure}

\begin{figure}
  \centering
  \includegraphics[scale=.55,trim=0 0 0 0,clip]{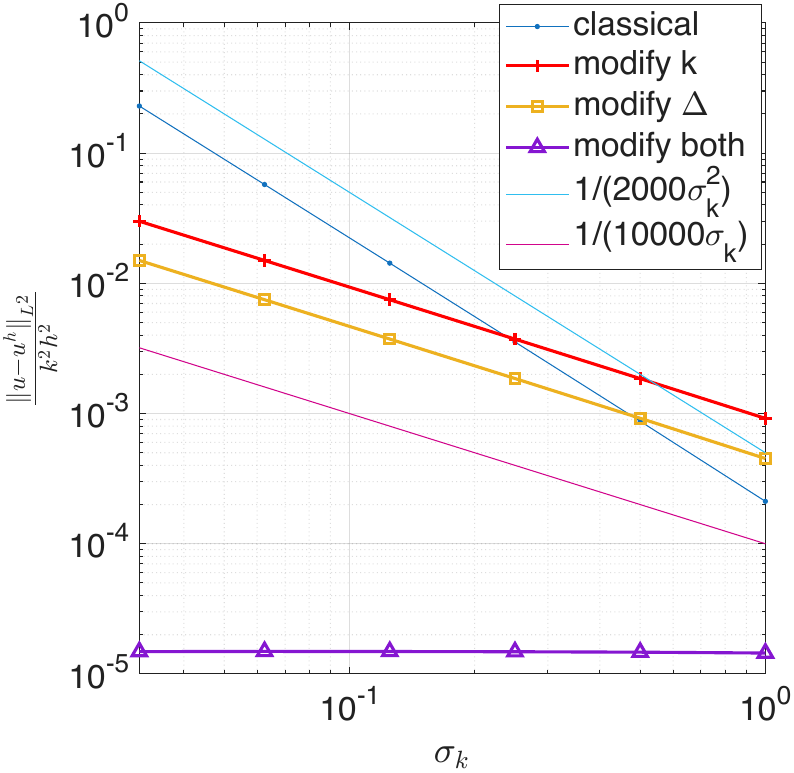}\quad
  \includegraphics[scale=.55,trim=0 0 0 0,clip]{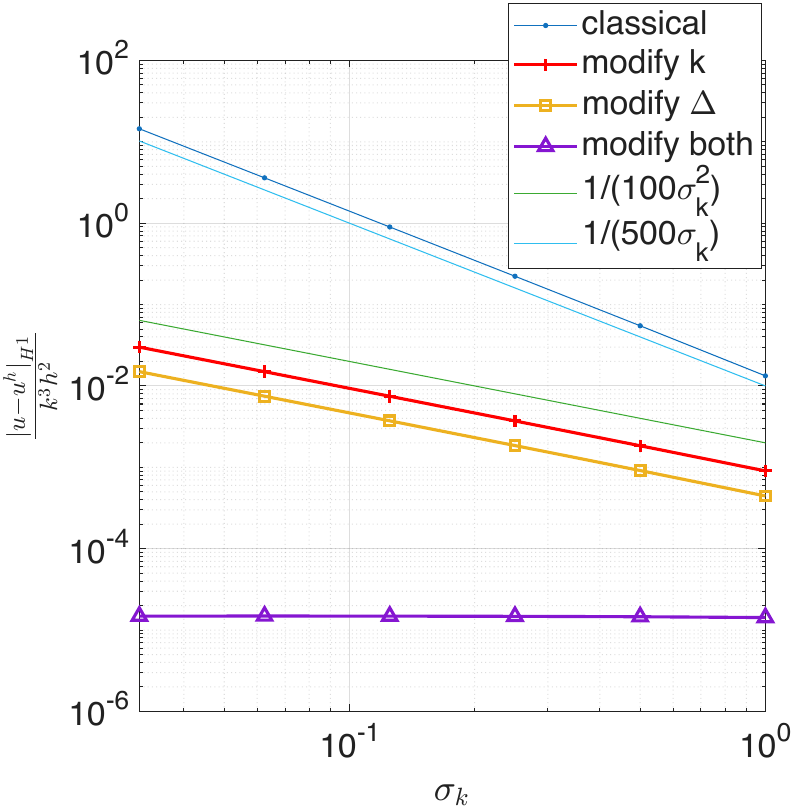}
  \caption{Numerical absolute errors (divided by $k^2h^2$ on the left and $k^3h^2$ on the right)
    against $\sigma_k=\operatorname{dist}(k,\pi\mathbb{N})$ with $k^3h^2\approx\sigma_k$,
    $k=20\pi+\sigma_k$, $f=\sin(20\pi x)$ and $g_0=g_1=0$ for the classical \eqref{1d3pt} \&
    dispersion free \eqref{kmod}, \eqref{Lmod}, \eqref{Lfmod}}\label{helm1dsigmak}
\end{figure}

\HZ{In the same settings as for the absolute errors, we numerically compute the relative errors. But
  note that the estimates \eqref{E1hrelL2max} and \eqref{E1hrelH1max} depend on $|u|_{p+2}/|u|_p$
  for $p=0,1$. In particular, when $\sigma_k=1$, $h\asymp k^{-3/2}$ and
  $f(x)=\sin(\pi(1-\frac{1}{h})x)$, we have $|u|_{p+2}/|u|_p\asymp k^3$, and we observed the
  numerical relative errors of order $k^3h^2=O(1)$ for all the finite difference schemes. Also, for
  the scheme \eqref{kmod} and \eqref{Lfmod}, taking $f(x)=\sin(\pi x)$ results $O(1)$ relative
  errors. Therefore, we shall assume here $f(x)=\sin(\xi x)$ with $\xi\asymp k$.  First, the scaling
  with $k$ and $h$ is shown in Figure~\ref{numrel}. Then we study the depdence on $\sigma_k$ in
  Figure~\ref{numrelsigmak}, which corroborates Theorem~\ref{relerr} as well as the results from the
  visual analysis in section~\ref{visual}.}

  \begin{figure}
  \centering
  \includegraphics[scale=.5,trim=0 0 0 0,clip]{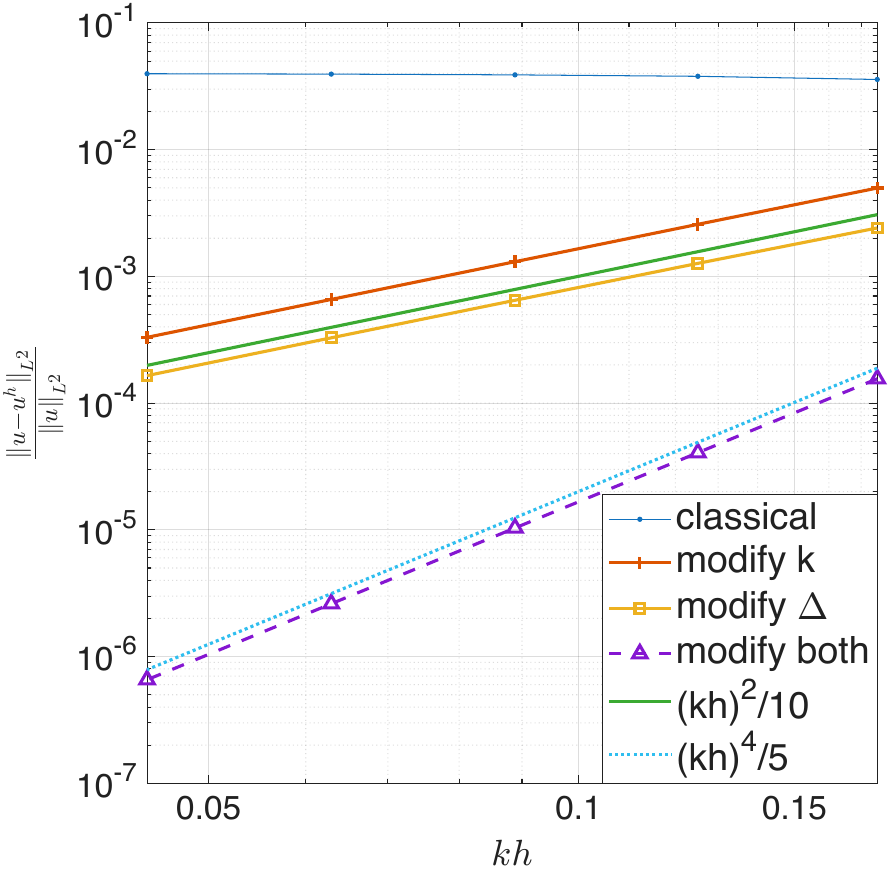}\qquad
  \includegraphics[scale=.5,trim=0 0 0 0,clip]{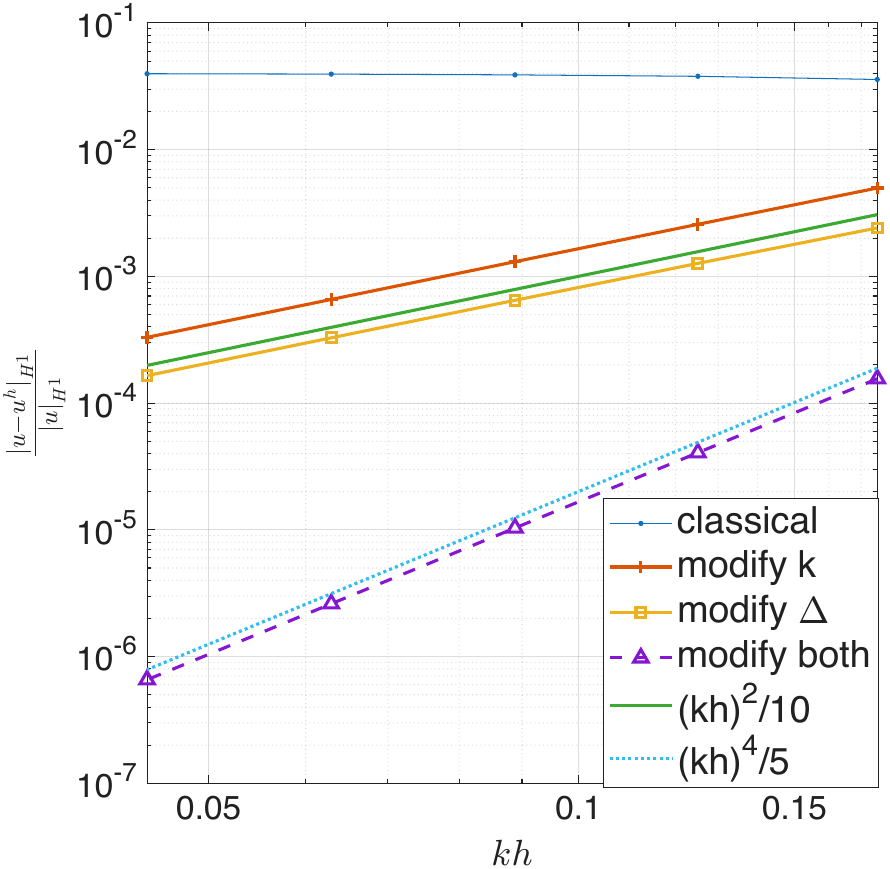}
  \caption{Numerical relative errors with $g_0=g_1=0$, $f=\sin(\xi x)$, $k=10\pi\cdot 2^{0:4}+1$,
    $h=1/\lceil k^{3/2}\rceil$, and $\xi=k-1$ for the classical \eqref{1d3pt} \& dispersion free
    \eqref{kmod}, \eqref{Lmod}, \eqref{Lfmod}. The $L^2$- and $H^1$-semi-norm relative errors are
    equal in this case. In this setting, the scalings shown are only apparent and must be
    reinterpreted in the form $k^{\alpha}h^2$ because all the schemes are second order in $h$ for
    fixed $k$. For example, comparing $(2k\cdot 2^{-3/2}h)^4=(2\cdot2^{-3/2})^4(kh)^4$ and
    $(2k)^{\alpha}(2^{-3/2}h)^2=2^{\alpha-3}k^{\alpha}h^2$, we obtain $-2=\alpha-3$ or
    $\alpha=1$.}\label{numrel}
\end{figure}

\begin{figure}
  \centering
  \includegraphics[scale=.5,trim=0 0 0 0,clip]{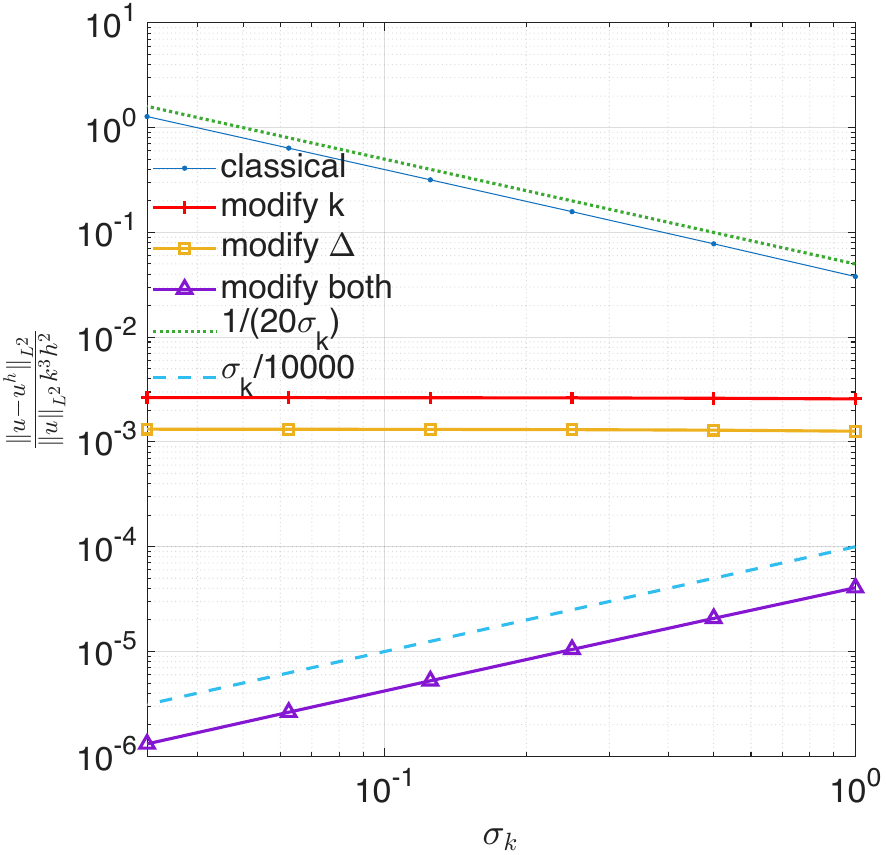}\quad
  \includegraphics[scale=.5,trim=0 0 0 0,clip]{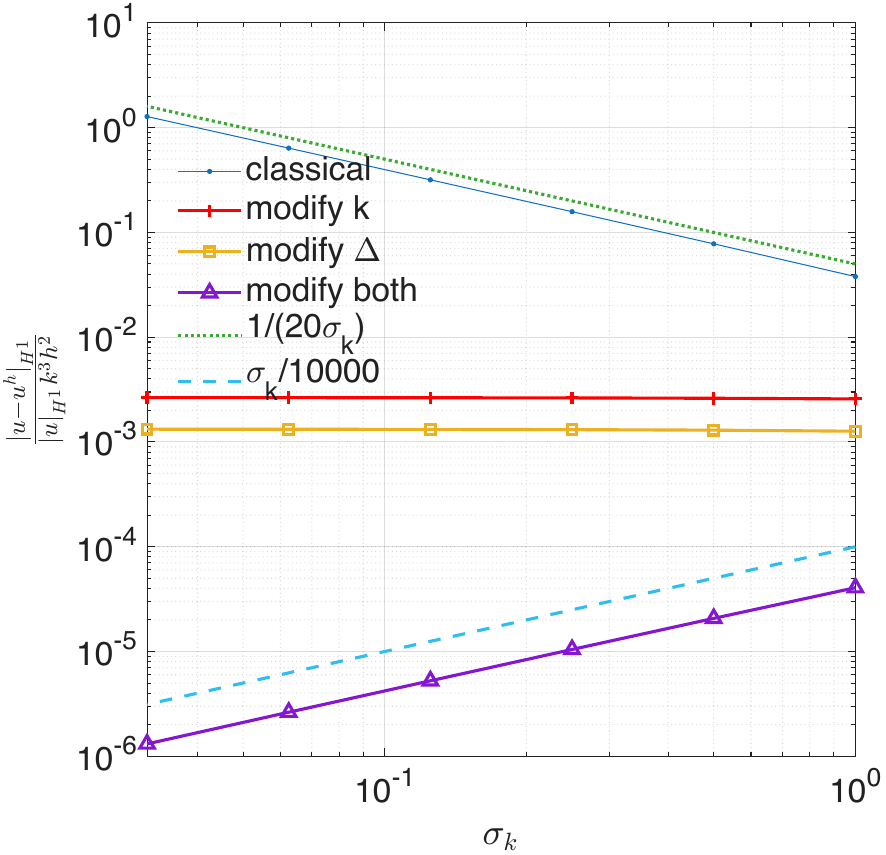}
  \caption{Numerical relative errors (divided by $k^3h^2$) against
    $\sigma_k=\operatorname{dist}(k,\pi\mathbb{N})$ with $k^3h^2\approx\sigma_k$,
    $k=20\pi+\sigma_k$, $f=\sin(20\pi x)$ and $g_0=g_1=0$ for the classical \eqref{1d3pt} \&
    dispersion free \eqref{kmod}, \eqref{Lmod}. The $L^2$- and $H^1$-semi-norm relative errors are
    equal in this case.}\label{numrelsigmak}
\end{figure}

\section{Conclusions and Discussions}

Based on Fourier decomposition, we analyzed the classical centered finite difference method for the
Helmholtz equation in 1D with Dirichlet conditions. \HZ{Such results have been considered standard
  for the closely related finite element method.} Our estimates are guaranteed by two-sided bounds
on the attainable worst case\HZ{, and the lower bounds are the main novelty}. That is, for each
wavenumber $k$, we can choose a source term to maximize the error order in $k$, and \HZ{the maximal
  error is bounded from below and above}. \HZ{The distance $\sigma_k$ of $k$ to the square-root of
  the spectrum of the continuous negative Laplacian is revealed in the error estimates.} We also
used the Fourier decomposition for evaluating three dispersion free finite difference methods, which
was impossible by the dispersion analysis (because all have zero dispersion error).

There are some limitations of the proposed approach in this work. First, it only works on uniform
grids of simple geometry. Second, with radiation boundary conditions the generalized Fourier
decomposition based on the Sturm-Liouville problem will lead to a non-orthogonal basis, and the
current approach does not work directly. {It might be interesting to try the singular value
  decomposition i.e. consider the least-squared ($\mathcal{L}^*\mathcal{L}$) Sturm-Liouville
  problem, which diagonalises the Helmholtz operator with two orthogonal bases.} Third, it works
only for constant coefficient problems.

However, the proposed approach provides a new tool for quantitatively evaluating linear schemes, in
complement to dispersion analysis. The latter considers only the zero source case and does not lead
to error estimates. In other words, the approach can be thought of as a quantitative tool for finite
difference methods in analogy to the local (or rigorous) Fourier analysis for multigrid methods. It
will be helpful for evaluating and designing finite difference methods for the Helmholtz
equation. Some ideas have been presented in \cite{dwarka2021} which inspired this work.

For future work, it is interesting to consider models in 2D/3D with simple geometries and analyze
some compact high-order schemes{, which is more challenging e.g. in finding the maximum of a
  multivariate symbol}. Moreover, it is worthwhile to try optimizing discretization parameters based
on the Fourier decomposition.

\HZ{We note that wavenumber explicit error estimates have been well developed for low and high order
  finite element methods. A parallel development for finite difference methods would be good to
  have, given that they are} actually used in applications (e.g. \cite{tournier2022}) for
\emph{variable coefficient} Helmholtz problems. The variational theory of \HZ{finite difference
  methods} developed in \cite{jova2014} may be useful.

\section*{Acknowledgement}

We thank Dr Haoran Chen from XJTLU and River Li from \url{math.stackexchange.com} for helpful
discussions on some analysis. \HZ{We thank the anonymous reviewer for many expert suggestions which
  substantially helped us to improve the manuscript.}

\section*{Funding}

MG acknowledges support from Swiss National Science Foundation under Grant 200021-236691.

\section*{Data Availability}

Data publicly available at https://bitbucket.org/mathbuddy/fourierfdm1d

\section*{Conflict of interest}
The authors declare that they have no conflict of interest.

\bibliographystyle{plain}
\bibliography{dispersion.bib}

\appendix

\section{{Proof of the claim \eqref{eqphi}}}
\label{phiproof}

The claim \eqref{eqphi} is
\[
  \phi(\theta)=\frac{\theta}{\sin^2\theta-\mu^2} - \frac{\theta}{\theta^2-\mu^2}
  \text{ has no local maximum for }\theta_k<\theta<\frac{\pi}{2}\text{ and }0<\mu<1,
\]
where $\theta_k:=\arcsin\mu$. The following is a proof.

\begin{proof}
  We calculate
  \[
    \phi'(\theta)=\frac{2 \theta^2}{\left(\theta^2-\mu^2\right)^2}-\frac{1}{\theta^2-\mu
     ^2}+\frac{1}{\sin^2\theta-\mu^2}-\frac{2 \theta \sin \theta \cos \theta}{\left(\sin
       ^2\theta-\mu^2\right)^2}=:\frac{\mathcal{P}(\theta)}{\left(\sin^2\theta-\mu
       ^2\right)^2\left(\theta^2-\mu^2\right)^2}.
  \]
  Let $\nu=\mu^2$. Then a critical point $\theta_*$ of $\phi(\theta)$ is the root of
  \[
    \begin{aligned}
      \mathcal{P}(\theta)=&-2 \theta^5 \sin \theta \cos \theta+\theta^4 \sin^2\theta+\theta ^2
      \sin^4\theta+\left(-\theta^4+4 \theta^3 \sin \theta \cos \theta-4 \theta^2 \sin
        ^2\theta+\sin^4\theta\right)\nu+\\
      &\left(3 \theta^2-\sin^2\theta-2 \theta \sin \theta \cos \theta\right)\nu^2=0.
    \end{aligned}
  \]
  The above equation is quadratic in $\nu$. By calculation, we find
  \[
    \nu=\frac{\theta^4-\sin^4\theta-4 \theta^3 \sin \theta \cos \theta + 4\theta^2
      \sin^2\theta\pm\frac{1}{4 \sqrt{2}}{(2 \theta^2+\cos (2 \theta )-1)}\sqrt{t}}{2 \left(3 \theta
        ^2-\sin \theta (\sin \theta+2 \theta \cos \theta)\right)}=:w_{\pm}(\theta),
  \]
  where
  \[
    t=128 \theta ^3 \sin\theta\cos\theta+ 8(\theta^2-\sin^2\theta)^2.
  \]
  We calculate
  \[
    w_{+}(\theta)-\sin^2\theta=\frac{1}{8}\frac{(\theta^2-\sin^2\theta) \left(4 \theta ^2-8 \theta
        \sin (2 \theta )+2 \cos (2 \theta )-2+\sqrt{2t}\right)}{3 \theta ^2-\sin \theta (\sin
      \theta+2 \theta \cos \theta)}.
  \]
  The second factor in the numerator is actually positive, which makes
  $w_{+}(\theta)>\sin^2\theta$. Indeed, we have
  \[
    2t-\left(4 \theta ^2-8 \theta \sin (2 \theta )+2 \cos (2 \theta )-2\right)^2
    =64 \theta  \sin \theta \cos \theta \left(6 \theta ^2-2 \theta  \sin (2 \theta )+\cos (2 \theta )-1\right)
  \]
  and
  \[
    \left(6 \theta ^2-2 \theta \sin (2 \theta )+\cos (2 \theta )-1\right)'=12 \theta -4 \sin (2
    \theta )-4 \theta \cos (2 \theta )>0\text{ for }0<\theta<\frac{\pi}{2}.
  \]
  So a critical point of $\phi$ must satisfy $\nu=w_{-}(\theta)$. We are going to show that
  $w_{-}'(\theta)>0$ for $0<\theta<\frac{\pi}{2}$. By calculation, we get
  \[
    w_{-}'(\theta)=\frac{16\theta v(\theta)}{8 \left(3 \theta ^2-\sin ^2\theta-2 \theta  \sin \theta \cos \theta\right)^2\sqrt{t}}\quad\text{ with }
  \]
  \[
    \begin{aligned}
      v(\theta)=&-16 \sqrt{2} \theta ^6+26 \sqrt{2} \theta ^6 \sin ^2\theta+16 \sqrt{2} \theta ^5 \sin ^3\theta \cos \theta-36 \sqrt{2} \theta ^5 \sin \theta \cos \theta-106 \sqrt{2} \theta ^4 \sin ^4\theta+\\
      &92 \sqrt{2} \theta ^4 \sin ^2\theta+16 \sqrt{2} \theta ^3 \sin ^5\theta \cos \theta-24
      \sqrt{2} \theta ^3 \sin ^3\theta \cos \theta+14 \sqrt{2} \theta ^2 \sin ^6\theta-8 \sqrt{2}
      \theta ^2 \sin ^4\theta+\\
      &2 \sqrt{2} \sin ^8\theta-4 \sqrt{2} \sin ^6\theta-4 \sqrt{2} \theta \sin ^5\theta \cos
      \theta-\theta ^4\sqrt{t}+5 \theta ^4 \sqrt{t} \sin
      ^2\theta+2 \theta ^3 \sqrt{t} \sin \theta \cos \theta-\\
      &4 \theta ^2 \sqrt{t} \sin ^4\theta-\sqrt{t} \sin ^6\theta+\sqrt{t} \sin ^4\theta-2 \theta
      \sqrt{t} \sin ^3\theta \cos \theta.
    \end{aligned}
  \]
  For $0<\theta\le1$, we claim
  \begin{equation}\label{sqrtt}
    a(\theta):=8 \sqrt{2} \theta ^2-\frac{8 \sqrt{2} \theta ^4}{3}<\sqrt{t}<\frac{7 \theta ^6}{30 \sqrt{2}}-\frac{8 \sqrt{2} \theta ^4}{3}+8 \sqrt{2} \theta ^2=:b(\theta).
  \end{equation}
  Indeed, $t=64 \theta ^3 \sin (2 \theta )+8 (\theta -\sin (\theta ))^2 (\theta +\sin (\theta ))^2$
  can be enlarged using
  \[
    \begin{aligned}
      \theta -\sin (\theta )&< \frac{\theta^3}{6}-\frac{\theta^5}{120}+\frac{\theta^7}{5040}\\
      \theta +\sin (\theta )&< 2 \theta-\frac{\theta^3}{6}+\frac{\theta^5}{120}\\
      \sin (2 \theta )&< 2 \theta-\frac{4 \theta^3}{3}+\frac{4 \theta^5}{15}-\frac{8 \theta
       ^7}{315}+\frac{4 \theta^9}{2835}
    \end{aligned}
  \]
  for $0<\theta<\frac{\pi}{2}$. Denote the enlarged $t$ as $\tilde{t}$. Then we have
  \[
    \begin{aligned}
      b(\theta)^2-\tilde{t}=&\left(\frac{584 \theta^{10}}{945}-\frac{2173 \theta
         ^{12}}{22680}\right) +\left(\frac{53 \theta^{14}}{18900}-\frac{1933 \theta
         ^{16}}{11907000}\right)+ \left(\frac{19 \theta^{18}}{2976750}-\frac{1921 \theta
         ^{20}}{11430720000}\right)+\\
      &\left(\frac{31 \theta^{22}}{11430720000}-\frac{\theta
         ^{24}}{45722880000}\right)>0\text{ for }0<\theta<1.
    \end{aligned}
  \]
  Hence, $\sqrt{t}<b(\theta)$. We can shrink $t$ using
  \[
    \begin{aligned}
      \theta -\sin (\theta )&> \frac{\theta^3}{6}-\frac{\theta^5}{120}\\
      \theta +\sin (\theta )&> 2 \theta-\frac{\theta^3}{6}\\
      \sin (2 \theta )&> 2 \theta-\frac{4 \theta^3}{3}+\frac{4 \theta^5}{15} -\frac{8 \theta^7}{315}
    \end{aligned}
  \]
  for $0<\theta<\frac{\pi}{2}$. Denote the shrunk $t$ as $\underset{\sim}t$. Then we have
  \[
    \underset{\sim}{t}-a(\theta)^2=\left(\frac{56 \theta ^8}{15}-\frac{352 \theta
        ^{10}}{189}\right)+ \left(\frac{47 \theta ^{12}}{2025}-\frac{2 \theta ^{14}}{2025}\right) +
    \frac{\theta ^{16}}{64800}>0 \text{ for }0<\theta<1.
  \]
  Substituting \eqref{sqrtt} and
  \[
    \theta-\frac{\theta^3}{6}<\sin\theta<\theta-\frac{\theta^3}{6}+\frac{\theta^5}{120}
  \]
  \[
    1-\frac{\theta^2}{2}+\frac{\theta^4}{24}-\frac{\theta^6}{720}<\cos\theta<1-\frac{\theta
    ^2}{2}+\frac{\theta^4}{24}
  \]
  into $v(\theta)$ with upper bounds into negative terms and lower bounds into positive terms,
  yields
  \[
    \begin{aligned}
      v(\theta)>& \left(\frac{389 \theta ^{10}}{90 \sqrt{2}}-\frac{1871 \theta ^{12}}{270
          \sqrt{2}}\right)+ \left(\frac{17071 \theta ^{14}}{5400 \sqrt{2}}-\frac{914 \sqrt{2} \theta
          ^{16}}{2025}\right)+
      \left(\frac{575611 \theta ^{18}}{2916000 \sqrt{2}}-\frac{4573 \theta ^{20}}{144000 \sqrt{2}}\right)+\\
      &\left(\frac{140789 \theta ^{22}}{38880000 \sqrt{2}}-\frac{1263749 \theta ^{24}}{4199040000
          \sqrt{2}}\right)+
      \left(\frac{58373 \theta ^{26}}{3110400000 \sqrt{2}}-\frac{49411 \theta ^{28}}{55987200000 \sqrt{2}}\right)+\\
      &\left(\frac{17221 \theta ^{30}}{559872000000 \sqrt{2}}-\frac{139 \theta ^{32}}{186624000000
          \sqrt{2}}\right)+ \left(\frac{\theta ^{34}}{89579520000 \sqrt{2}}-\frac{7 \theta
          ^{36}}{89579520000000 \sqrt{2}}\right)>0
    \end{aligned}
  \]
  for $0<\theta<\sqrt{\frac{1167}{1871}}$. Note that $\sqrt{\frac{1167}{1871}}>\frac{7}{10}$. For
  $\frac{7}{10}\le\theta\le\frac{\pi}{2}$, we make use of interval arithmetic in Mathematica to show
  $v(\theta)>0$. The code is given below.
\begin{verbatim}
a = 7/10; b = Pi/2; n = 100; (*Number of subdivisions*) subintervals = Subdivide[a, b, n];
allIntervalsPositive = True; (*Check each subinterval*)
For[i = 1, i <= n, i++, 
    smallInt = Interval[{subintervals[[i]], subintervals[[i + 1]]}];
    vSmallInt = v[smallInt];
    If[Min[vSmallInt] <= 0,  allIntervalsPositive = False;  Break[];];];
If[allIntervalsPositive, Print["v(x) > 0 in [", a, ", ", b, "]"], 
    Print["Potential root found in subinterval: ", smallInt]]
\end{verbatim}
  
  In summary, $w_{-}'(\theta)>0$ on $\theta\in(0,\frac{\pi}{2})$. Therefore, in
  $\theta\in(0,\frac{\pi}{2})$, $\nu=w_{-}(\theta)$ has at most one root, and $\phi(\theta)$ has at
  most one critical point. Note that
  $\mathcal{P}(\theta=\arcsin\mu)=-2 \theta (\theta -\sin \theta)^2 \sin \theta (\theta +\sin
  \theta)^2 \cos \theta<0$ for $0<\mu<1$, and
  $\mathcal{P}(\theta=\frac{\pi}{2})=-\frac{1}{16} (\mu^2-1) \left((16-12 \pi ^2) \mu^2+\pi ^4+4 \pi
    ^2\right)>0$ for $0<\mu<1$. So $\phi(\theta)$ in $\theta_k <\theta<\frac{\pi}{2}$ is decreasing
  before the only critical point and then increasing after. That means $\phi(\theta)$ has no local
  maximum in $\theta_k<\theta<\frac{\pi}{2}$.
\end{proof}

\section{{Proof of the claim \eqref{eqphie}}}
\label{a:phie}

We fist copy the claim \eqref{eqphie} here:
\[
  \phi_e(\theta) = \frac{1}{\sin^2\theta-\mu^2}-\frac{1}{\theta^2-\mu^2}\text{ has no local maximum for }\theta\in(\theta_k,\frac{\pi}{2})\text{ and }0<\mu<1,
\]
where $\theta_k:=\arcsin\mu$. Then we give the proof as follows.

\begin{proof}
  We calculate
  \[
    \phi_e'(\theta)=\frac{2 \theta }{\left(\theta^2-\mu^2\right)^2}-\frac{2 \sin \theta \cos
      \theta}{\left(\sin^2\theta-\mu^2\right)^2}=:2\frac{\mathcal{P}(\theta)}{\left(\theta^2-\mu
      ^2\right)^2 \left(\mu^2-\sin^2\theta\right)^2}.
  \]
  Let $\nu=\mu^2$. Then a critical point $\theta_*$ of $\phi_e(\theta)$ is the root of
  \[
    \mathcal{P}(\theta)=-\theta^4 \sin \theta\cos \theta+\theta \sin^4\theta+ \left(2 \theta
    ^2\sin \theta\cos \theta-2\theta\sin^2\theta\right)\nu+ (\theta -\sin \theta\cos
    \theta)\nu^2=0
  \]
  By solving for $\nu$ from the above equation, we find $\theta_*$ must satisfy one of the following
  equations
  \[
    \nu=\frac{\theta \sin^2\theta-\theta^2 \sin \theta \cos \theta\pm\sqrt{\theta } \sqrt{\sin
        \theta\cos\theta} \left(\theta^2-\sin^2 \theta\right)}{\theta -\sin \theta \cos
      \theta}=:w_{\pm}(\theta).
  \]
  We calculate and use $x>\sin x>0$ for $0<x<\pi$ to find
  \[
    w_{+}(\theta)-\sin^2\theta= 2\frac{\theta^2- \sin^2\theta}{\theta -\sin \theta \cos \theta}
    \left(\sqrt{2\theta } \sqrt{\sin (2\theta)}-\sin (2 \theta )\right)>0.
  \]
  So the only equation that $\theta_*$ needs to satisfy is $\nu=w_{-}(\theta)$. We claim that
  $w_{-}'(\theta)>0$ and thus $w_{-}(\theta)$ is increasing for $0<\theta<\frac{\pi}{2}$. This will
  be shown for $0<\theta\le 1$ and $1\le \theta<\frac{\pi}{2}$ separately. We calculate
  \[
    \begin{aligned}
      &\left(2\sqrt{\theta}\sqrt{\sin\theta\cos\theta}(\sin\theta\cos\theta-\theta)^2\right)w_{-}'(\theta)\\
      =\,&2\theta^{3/2}\sin^2\theta\cos^2\theta \sqrt{\sin\theta\cos\theta}-2 \theta^{7/2} \cos
    ^2\theta \sqrt{\sin\theta\cos\theta}-2 \theta^{3/2}\sin^4\theta
      \sqrt{\sin \theta \cos \theta}+\\
      &2 \theta^{7/2} \sin^2\theta \sqrt{\sin \theta \cos \theta}+2 \theta^{5/2} \sin \theta \cos
      \theta \sqrt{\sin \theta \cos \theta}+\theta^4 \sin
    ^2\theta-\theta^4 \cos^2\theta-\theta^3 \sin \theta \cos^3\theta+\\
      &\theta^3 \sin^3\theta \cos \theta-3 \theta^3 \sin \theta \cos \theta-\theta^2 \sin
    ^4\theta+10 \theta^2 \sin
    ^2\theta \cos^2\theta-3 \theta \sin^3\theta \cos^3\theta-\sin^4\theta \cos^2\theta-\\
      &\theta \sin^5\theta \cos \theta-\theta \sin^3\theta \cos \theta-2 \sqrt{\theta } \sin
    ^3\theta \cos \theta \sqrt{\sin \theta \cos \theta}=:v(\theta).
    \end{aligned}    
  \]
  We note that $v(\theta)$ is a sum of positive terms and negative terms.

  For $0<\theta\le 1$, we first show that
  \begin{equation}\label{sqrtest}
    \sqrt{\theta }-\frac{\theta^{5/2}}{3}+\frac{\theta^{9/2}}{90}-\frac{\theta^{13/2}}{180}
    <\sqrt{\sin\theta\cos\theta}<\sqrt{\theta }-\frac{\theta^{5/2}}{3}+\frac{\theta^{9/2}}{90},
  \end{equation}
  or equivalently
  \[
    a(\theta):=2\theta-\frac{4\theta^3}{3}+\frac{4\theta^5}{15}-\frac{\theta^7}{27}+
    \frac{31\theta^9}{4050}-\frac{\theta^{11}}{4050}+\frac{\theta^{13}}{16200}< \sin(2\theta)<
    2\theta-\frac{4\theta^3}{3}+\frac{4\theta^5}{15}-\frac{2\theta^7}{135}+\frac{\theta^9}{4050}=:b(\theta).
  \]
  By Maclaurin series, we can show for $0<\theta\le1$ that
  \[
    c(\theta):=2\theta-\frac{4\theta^3}{3}+\frac{4\theta^5}{15}-\frac{8\theta^7}{315}+\frac{4\theta^9}{2835}-\frac{8\theta^{11}}{155925}<
    \sin(2\theta)<2\theta-\frac{4\theta^3}{3}+\frac{4\theta^5}{15}-\frac{8\theta^7}{315}+\frac{4\theta^9}{2835}=:d(\theta)
  \]
  We note for $0<\theta\le1$ that
  \[
    c(\theta)-a(\theta)=\left(\frac{11\theta^7}{945}-\frac{59\theta^9}{9450}\right)+\left(\frac{61\theta^{11}}{311850}-\frac{\theta^{13}}{16200}\right)>0,\quad
    b(\theta)-d(\theta)=\frac{2\theta^7}{189}-\frac{11\theta^9}{9450}>0.
  \]
  We substitute \eqref{sqrtest} and
  \[
    \theta-\frac{\theta^3}{6}+\frac{\theta^5}{120}-\frac{\theta^7}{5040}<\sin\theta<\theta-\frac{\theta^3}{6}+\frac{\theta^5}{120}
  \]
  \[
    1-\frac{\theta^2}{2}+\frac{\theta^4}{24}-\frac{\theta^6}{720}<\cos\theta<1-\frac{\theta
    ^2}{2}+\frac{\theta^4}{24}
  \]
  into $v(\theta)$ with upper bounds into negative terms and lower bounds into positive terms, and
  get
  \[
    \begin{aligned}
      v(\theta)>&\left(\frac{92 \theta^{10}}{315}-\frac{7619 \theta^{12}}{37800}\right)+
      \left(\frac{5207 \theta^{14}}{85050}-\frac{678781 \theta^{16}}{63504000}\right)+
      \left(\frac{267361 \theta^{18}}{228614400}-\frac{9557 \theta^{20}}{120960000}\right)+
      \left(\frac{11827 \theta^{22}}{4572288000}\right)+\\
      &\left(\frac{299807\theta^{24}}{5761082880000}-\frac{2240017\theta^{26}}{230443315200000}\right)+\left(\frac{78433\theta^{28}}{172832486400000}-\frac{57977\theta^{30}}{5530639564800000}\right)+\\
      &\left(\frac{73 \theta^{32}}{592568524800000}-\frac{\theta^{34}}{1185137049600000}\right)>0
      \quad\text{for }0<\theta\le1.
  \end{aligned}
  \]
  Therefore, $w_{-}'(\theta)>0$ for $0<\theta\le 1$.

  For $1\le\theta<\frac{\pi}{2}$, we make use of interval arithmetic in Mathematica. The code is
  similar to that in Appendix~\ref{phiproof}.  Running the code, we get the result
  \verb|v(x)>0 in [1,Pi/2]|. Hence, $w_{-}'(\theta)>0$ for $1\le \theta<\frac{\pi}{2}$.

  Now we have proved $w_{-}$ is increasing in $\theta\in(0,\frac{\pi}{2})$. So $\nu=w_{-}(\theta)$
  has at most one root in $\theta\in(0,\frac{\pi}{2})$. That is, $\phi_e(\theta)$ has at most one
  critical point in $\theta\in(0,\frac{\pi}{2})$. Note that
  \[
    \mathcal{P}(\theta=\arcsin\mu)=-(\theta -\sin \theta)^2 \sin \theta (\theta +\sin\theta)^2
    \cos \theta<0\text{ and }\mathcal{P}(\theta=\frac{\pi}{2})=\frac{\pi}{2}(\nu-1)^2>0.
  \]
  So $\phi_e'(\theta)$ is negative before the only critical point of $\phi_e$ and positive after for
  $\theta\in(\arcsin\mu,\frac{\pi}{2})$. The only critical point of $\phi_e$ is a local
  minimum. That shows $\phi_e$ for $\theta\in(\arcsin\mu,\frac{\pi}{2})$ has no local maximum.
\end{proof}

\end{document}